\documentclass{article}
\usepackage{a4wide}
\usepackage{amsthm}
\usepackage{graphicx}
\usepackage{amssymb}
\usepackage{amsmath}
\usepackage{ascmac}
\usepackage{setspace}
\usepackage{float}
\usepackage{algorithm}
\usepackage{algorithmic}
\usepackage{setspace}
\usepackage{comment}

\newtheorem{Corollary}[equation]{Corollary}

\theoremstyle{definition}

\theoremstyle{remark}
\newtheorem{Remark}[equation]{Remark}
\numberwithin{equation}{section}
\newtheorem{Claim}[equation]{Claim}
\DeclareMathOperator{\Hom}{Hom}
\DeclareMathOperator{\End}{End}

\DeclareMathOperator{\id}{id}

\DeclareMathOperator{\ad}{ad}

\newcommand{\relmiddle}[1]{\mathrel{}\middle#1\mathrel{}}

\newtheorem{theorem}{Theorem}[section]
\newtheorem{proposition}[theorem]{Proposition}
\newtheorem{corollary}[theorem]{Corollary}
\newtheorem{lemma}[theorem]{Lemma}
\newtheorem{definition}[theorem]{Definition}
\newtheorem{remark}[theorem]{Remark}

\numberwithin{equation}{section}

\usepackage[
	pdfborder={0 0 0},
	colorlinks,
	plainpages,
	backref,
	citecolor=red!50!black,
	filecolor=Darkgreen,
	linkcolor=blue!40!black,
	urlcolor=cyan!50!black!90]{hyperref}

\newcommand{\nc}{\newcommand}

\nc{\around}{\scriptsize\rightturn}
\nc{\caround}{\scriptsize\leftturn}

\nc{\ot}{\otimes}
\nc{\op}{\oplus}
\nc{\ol}{\overline}
\nc{\un}{\underline}

\nc{\mc}{\mathcal}
\nc{\ms}{\mathsf}
\nc{\mf}{\mathfrak}
\nc{\mb}{\mathbf}
\nc{\bb}{\mathbb}
\nc{\mr}{\mathrm}

\nc{\al}{\alpha}
\nc{\bet}{\beta}
\nc{\eps}{\epsilon}
\nc{\del}{\delta}
\nc{\ga}{\gamma}
\nc{\Ga}{\Gamma}
\nc{\ka}{\kappa}
\nc{\la}{\lambda}
\nc{\om}{\omega}
\nc{\si}{\sigma}
\nc{\Si}{\Sigma}
\nc{\Ups}{\upsilon}
\nc{\vphi}{\varphi}

\nc{\gr}{\mathrm{gr}}

\nc{\Ug}{U\mathfrak{g}}
\nc{\Ub}{U\mathfrak{b}}

\nc{\Hk}{\mathsf{H}}
\nc{\ombH}{\overline{\mathbf{H}}}

\nc{\ud}{\underline}
\nc{\tl}{\tilde}
\nc{\wt}{\widetilde}
\nc{\wh}{\widehat}

\nc{\Ext}{\mathrm{Ext}}

\nc{\Ind}{\mathrm{Ind}}

\nc{\RHom}{\mathrm{RHom}}
\nc{\Sym}{\mathrm{Sym}}

\nc{\C}{\mathbb{C}}
\nc{\N}{\mathbb{N}}
\nc{\Z}{\mathbb{Z}}

\nc{\lan}{\langle}
\nc{\ran}{\rangle}

\nc{\equ}[1]{\begin{equation}#1\end{equation}}
\nc{\eqa}[1]{\begin{equation}\begin{alignedat}{50}#1\end{alignedat}\end{equation}}
\nc{\eqn}[1]{\begin{equation*}\begin{alignedat}{50}#1\end{alignedat}\end{equation*}}
\nc{\eqg}[1]{\begin{equation}\begin{gathered}#1\end{gathered}\end{equation}}

\nc{\ali}[1]{\begin{alignat}{50}#1\end{alignat}}
\nc{\als}[1]{\begin{subequations}\begin{alignat}{50}#1\end{alignat}\end{subequations}}
\nc{\aln}[1]{\begin{alignat*}{50}#1\end{alignat*}}

\nc{\gat}[1]{\begin{gather}#1\end{gather}}
\nc{\gas}[1]{\begin{subequations}\begin{gather}#1\end{gather}\end{subequations}}
\nc{\gan}[1]{\begin{gather*}#1\end{gather*}}

\nc\el{\nonumber\\}
\nc\nn{\nonumber}

\nc{\tx}[1]{\qu\text{#1}\qu}

\nc{\qu}{\quad}
\nc{\qq}{\qquad}

\nc{\iso}{\stackrel{\sim}{\longrightarrow}}
\nc{\into}{\hookrightarrow}
\nc{\onto}{\twoheadrightarrow}

\usepackage{color}
\usepackage[usenames,dvipsnames]{xcolor}
\nc{\red}{\color{red}}
\nc{\blu}{\color{blue}}

\newcommand{\msk}{\mathsf{k}}

\newcommand{\msm}{\mathsf{m}}

\newcommand{\msq}{\mathsf{q}}

\newcommand{\mfgl}{\mathfrak{g}\mathfrak{l}}
\newcommand{\mfsl}{\mathfrak{s}\mathfrak{l}}

\nc{\DY}{DY^c_\hbar(\mfgl_n)}
\nc{\DYC}{\wh{DY^c_\hbar}(\mfgl_n)}
\allowdisplaybreaks

\title{Affine super Yangians and deformed double current superalgebras}
\author{Nicolas Guay, Peggy Jankovic, Mamoru Ueda}
\date{}

\begin{document}
\maketitle

\begin{abstract}
We extend to the super Yangian of the special linear Lie superalgebra $\mathfrak{sl}_{m|n}$ and its affine version certain results related to Schur-Weyl duality. We do the same for the deformed double current superalgebra of $\mathfrak{sl}_{m|n}$, which is introduced here for the first time. 
\end{abstract}

\section{Introduction}

The study of quantum groups and of Lie superalgebras has found a lot of inspiration in theoretical physics, although both of those families of mathematical objects are of interest from a purely mathematical point of view. Their intersection is the theory of quantum superalgebras, which include, for instance, super Yangians and quantum loop superalgebras associated to classical Lie superalgebras like $\mathfrak{sl}_{m|n}$, the special linear Lie superalgebra. It is also possible to consider quantum superalgebras associated to affine Lie superalgebras, one family being the affine super Yangians and another one the related quantum toroidal superalgebras. One goal of this paper is to increase our understanding of the super Yangians by extending to them some results known for the Yangian of the general linear Lie algebra, e.g. Schur-Weyl duality. In the case of the super Yangian associated to $\mathfrak{sl}_{m|n}$ and to $\mathfrak{gl}_{m|n}$, such a result appears in part in \cite{LuMu}, but we expand the results in \textit{loc. cit.} -  see also \cite{Flicker} for the related quantum loop superalgebra of $\mathfrak{sl}_{m|n}$. For the affine super Yangian of $\mathfrak{sl}_{m|n}$, our main results, Theorem \ref{main theorem} and Theorem \ref{reverse}, are similar to the main ones in \cite{Lu} for quantum toroidal algebras.

There is a family of quantum algebras related to affine Yangians and quantum toroidal algebras, namely the deformed double current algebras, which can be defined for any simple Lie algebra \cite{Gu1,Gu2,GuYa}. In the present paper, we introduce for the first time deformed double current superalgebras associated to $\mathfrak{sl}_{m|n}$, which should be worthy of further study. They are realized as deformations of enveloping superalgebras of certain Steinberg Lie superalgebras in such a way that a Schur-Weyl functor connecting them to rational Cherednik algebras can be constructed.

The super Yangian $Y(\mathfrak{gl}_{m|n})$ of the general linear Lie superalgebra $\mathfrak{gl}_{m|n}$ can be defined in more than one way \cite{Gow}. One presentation is via the RTT-relation and Section \ref{SWRTT} studies Schur-Weyl duality between $Y(\mathfrak{gl}_{m|n})$ and the degenerate affine Hecke algebra of the symmetric group $S_l$ in the spirit of \cite{ara}: see Theorem \ref{SW-tt} in Subsection \ref{SWdahaY}, which also appeared in \cite{LuMu}. The main goal of Subsection \ref{CPequivcat} is to show that, under the condition that $m,n \ge l+1$, the Schur-Weyl functor of Theorem \ref{SW-tt} provides an equivalence of categories: this is a consequence of Theorem \ref{floridakilos-1}, whose proof follows the arguments used in \cite{CP96}, and of Proposition \ref{thm-sup}. 

The approach in Section \ref{SWRTT} does not extend directly to affine super Yangians, so Schur-Weyl duality is established in Section \ref{SWmini} for the super Yangian of $\mathfrak{sl}_{m|n}$ using the ``minimalistic'' current presentation of $Y(\mathfrak{sl}_{m|n})$ \cite{Ue}: see Theorems \ref{finite case-1} and \ref{flor}. Generators $J(h_i),J(x_i^{\pm})$ of $Y(\mathfrak{sl}_{m|n})$, analogous to generators denoted similarly in \cite{driJ} and \cite{Ue}, play a role in that section. Moreover, we also need this construction for some parity sequences that are shifts of the standard one: see Theorem \ref{finite case}, which plays a role later in Subsection \ref{SWasYang}. The main results of Section \ref{asYangDDAHA} are Theorem \ref{main theorem} and Theorem \ref{reverse} whose proofs rely on the important Lemma \ref{tau} which, in turn, combines ideas from \cite{Lu} and \cite{Gu1}. 

The last section extends to the Lie superalgebra $\mathfrak{sl}_{m|n}$ some of the results in \cite{Gu1,Gu2} about deformed double current algebras. Subsection \ref{Steinberg} establishes a presentation of the Steinberg Lie superalgebra of $\mathfrak{sl}_{m|n}\otimes_{\C}\C[u,v]$ using generators and relations of low degree in the variables $u$ and $v$. This is relevant for Definition \ref{moneypowerglory} of the deformed double current superalgebra $\mathcal{D}(\mathfrak{sl}_{m|n})$, which is new. The main result in the last section, namely Theorem \ref{syth}, connects modules over $\mathcal{D}(\mathfrak{sl}_{m|n})$ to modules over the rational Cherednik algebra of $S_l$ as in \cite{Gu2} via the Schur-Weyl construction.

\section{Lie superalgebras}
In this section, we set some notations relating to the Lie superalgebras $\mathfrak{sl}_{m|n}$ and $\mathfrak{gl}_{m|n}$ and recall the Schur-Weyl duality for $\mathfrak{gl}_{m|n}$ and the symmetric group given by Sergeev \cite{sergeev}.
\subsection{The Lie superalgebras $\mathfrak{sl}_{m|n}$ and $\mathfrak{gl}_{m|n}$}
Let us set $M_{k,l}(\mathbb{C})$ as the set of $k\times l$ matrices over $\mathbb{C}$. We define the Lie superalgebras $\mathfrak{sl}_{m|n}$ and $\mathfrak{gl}_{m|n}$ as follows;
\begin{gather*}
\mathfrak{gl}_{m|n}=\left\{\begin{pmatrix}
A&B\\
C&D
\end{pmatrix} \relmiddle| A \in M_{m,m}(\mathbb{C}), B \in M_{m,n}(\mathbb{C}), C \in M_{n,m}(\mathbb{C}),\text{ and }D \in M_{n,n}(\mathbb{C})\right\},\\
\mathfrak{sl}_{m|n}=\left\{\begin{pmatrix}
A&B\\
C&D
\end{pmatrix}\in\mathfrak{gl}_{m|n} \relmiddle| \text{str}\begin{pmatrix}A&B\\C&D\end{pmatrix} =  \text{tr}(A)-\text{tr}(D)=0\right\},
\end{gather*}
where tr is the trace of a matrix and we set the commutator relation $\left[\begin{pmatrix}
A&B\\
C&D
\end{pmatrix},\begin{pmatrix}
E&F\\
G&H
\end{pmatrix}\right]$ as
\begin{gather*}
\left[\begin{pmatrix}
A&B\\
C&D
\end{pmatrix},\begin{pmatrix}
E&F\\
G&H
\end{pmatrix}\right]=\begin{pmatrix}
AE-EA+(BG+FC)&AF+BH-(EB+FD)\\
CE+DG-(GA+HC)&DH-HD+(CF+GB)
\end{pmatrix}.
\end{gather*}
We denote by $E_{i,j}$ the matrix unit, that is, the matrix whose $(i,j)$ component is $1$ and other components are $0$. Then, $\{E_{i,j}\}_{1\leq i,j\leq m+n}$ becomes a basis of $\mathfrak{gl}_{m|n}$. By the definition of $\mathfrak{gl}_{m|n}$ and $\mathfrak{sl}_{m|n}$, for $m\neq n$, we have a decomposition
\begin{equation*}
\mathfrak{gl}_{m|n}=\mathfrak{sl}_{m|n}\oplus\C I_{m+n},
\end{equation*}
where $I_{m+n}=\sum_{1\leq i\leq m+n}E_{i,i}$.

The notion of the Cartan subalgebra works in the superalgebra setting as well. The Cartan subalgebra $\mf{h}$ of $\mfgl_{m|n}$ is made up of all diagonal matrices.
Let the parity for an integer $i\in\{1,\dots,m+n\}$ be defined as $$|i| = \begin{cases} 0&\text{ if }1\leq i\leq m, \\ 1&\text{ if }m+1\leq i\leq m+n.\end{cases}$$ 
The Cartan matrix of $\mfsl_{m|n}$, $C=(a_{i,j})^{m+n-1}_{i,j=1}$, has entries 
\begin{equation*}
a_{i,j}=\begin{cases}
(-1)^{|i|}+(-1)^{|i+1|}&\text{ if }i=j,\\
-(-1)^{|i+1|}&\text{ if }j=i+1,\\
-(-1)^{|i|}&\text{ if }i=j+1,\\
0&\text{ otherwise}.
\end{cases}
\end{equation*}

We can use the supertrace to define a non-degenerate supersymmetric bilinear form on $\mfgl_{m|n}$ given by
$$(\cdot,\cdot  ):  \mfgl_{m|n} \times \mfgl_{m|n} \rightarrow \C, \ (E_{i,j},E_{k,l}) =\text{str}(E_{i,j}E_{k,l})=\delta_{i,l}\delta_{j,k}(-1)^{|i|}.$$
If we restrict the bilinear form $(\cdot,\cdot)$ to $\mf{h}$, we obtain a non-degenerate symmetric bilinear form on $\mf{h}$ given by 
	$$ (E_{i,i},E_{j,j})=\delta_{i,j}(-1)^{|i|}.$$

Denote by $\{\eps_i\}_{1\leq i\leq m+n}$ the basis of $\mf{h}^{*}$ dual to $\{E_{i,i} \}_{1\leq i\leq m+n}  $. We can identify  $\eps_i$ with $((-1)^{|i|}E_{i,i},\cdot)$. 
Then, the form $(\cdot , \cdot)$  on  $\mf{h}$ induces a non-degenerate bilinear form on $\mf{h}^{*}$, which we can  also denote $(\cdot,\cdot )$, given by 
$$ (\eps_i , \eps_ j ) = \delta_{i,j}(-1)^{|i|}.$$
				
	The root system of $\mfgl_{m|n}$ can be expressed as $\Phi=\Phi_0  \cup \Phi_1$ with
$$\Phi_0 = \{ \eps_i - \eps_j  | i\ne j, |i|=|j| \}, \ \Phi_1 = \{ \pm(\eps_i-\eps_j) | i\le m<j \} .$$			

Let $\alpha_i$ be $\eps_i - \eps_{i+1}$. The Lie superalgebra $\mfsl_{m|n}$ has positive simple roots $\{\alpha_1, \cdots,\alpha_{m+n-1}\},$ where $\alpha_m$ is odd and the rest are even.

\subsection{Schur-Weyl Duality for $\mathfrak{gl}_{m|n}$ and the Symmetric group} 

In the classical case, the Schur-Weyl duality describes the relationship between representations of the general linear Lie group (or Lie algebra) and representations of the symmetric group. 
This duality was generalized to the superalgebra setting by Sergeev in 1985 \cite{sergeev}.

We set a super vector space $\mathbb{C}(m|n)=\bigoplus_{i=1}^{m+n}\limits \mathbb{C}e_i$, where the vectors $e_1,\cdots,e_m$ are even and $e_{m+1},\cdots,e_{m+n}$ are odd. We denote by $|v|$ the parity for a homogeneous element $v\in\mathbb{C}(m|n)$.
Let us denote the symmetric group on $l$ letters by $S_l$ and the permutation of $\{i,j\}$ by $\sigma_{i,j}$. We define the actions of $\mfgl_{m|n}$ and the symmetric group $S_l$ on $\C({m|n})^{\otimes l}$ by
 \begin{align}
 X\left(\bigotimes_{i=1}^lv_i\right) &=(-1)^{|X|\sum_{j\leq i-1}|j|}\sum_{i=1}^l\left(\bigotimes_{j=1}^{i-1}v_j\right)\otimes Xv_i\otimes\left(\bigotimes_{k=i+1}^{l}v_k\right), \\
 \sigma_{i,i+1}\left(\bigotimes_{i=1}^lv_i\right) &= (-1)^{|v_i||v_{i+1}|} \left(\bigotimes_{j=1}^{i-1}v_j\right)\otimes v_{i+1} \otimes v_i \otimes \left(\bigotimes_{k=i+1}^{l}v_k\right)
 \end{align}
for a homogeneous element $X\in\mfgl_{m|n}$ and homogeneous elements $v_1,\cdots,v_l \in \C({m|n})$.
Let $\Phi_l:U(\mfgl_{m|n})\rightarrow \End (\C({m|n})^{\otimes l})$ and  $\Psi_l: \mathbb{C}[S_l ] \rightarrow \End (\C({m|n})^{\otimes l})$ be the maps defining the actions of $\mfgl_{m|n}$ and $S_l$ respectively.

\begin{lemma} \label{lem38}[Lemma 8 in \cite{sergeev} and Lemma 3.1 in \cite{chengbook}]
The actions of $\mfgl_{m|n}$ and $S_l$ on $\C({m|n})^{\otimes l}$ commute with each other.
\end{lemma}

Berele and Regev \cite{BRhook} presented a generalization of Weyl's works into the superalgebra realm, introducing what they called ``hook analogues'' of classical objects. 
Weyl's Theorem provides a parametrization of irreducible representations of $GL_n(\C)$ by partitions whose Young diagrams lie inside the strip of a certain height $k$. The main result of Berele-Regev \cite{BRhook} provides a hook analogue of Weyl's Theorem. Berele-Regev \cite{BRhook} connect semistandard tableaux in two sets of variables to representations of $\mf{gl}_{m|n}(\C)$. The upshot of this is a super-analogue of Schur's double centralizing theorem.

\begin{theorem}\label{chengbook} 
(Schur-Sergeev Duality) [Theorem 1 in \cite{sergeev}, Theorem 3.3 in \cite{chengbook}, Theorem 3.9, 3.10 in \cite{CW12}]
\begin{enumerate}
\item The images of $\Phi_l$ and $\Psi_l$, being $\Phi_l(U(\mfgl_{m|n}))$ and $\Psi_l(\C [S_l])$, satisfy the double centralizer property, that is,
\begin{align}
	\Phi_l(U(\mfgl_{m|n}))&= \textnormal{End}_{\C[S_l]}((\C^{m|n})^{\otimes l}),\\
	\Psi_l( \C [S_l]) &=  \textnormal{End}_{U(\mfgl_{m|n})} ((\C^{m|n})^{\otimes l}).
\end{align}
\item For a right $\mathbb{C}[S_l]$-module $M$, we can define the left $U(\mathfrak{gl}_{m|n})$-module $SW(M)$ as $M\otimes_{\mathbb{C}[S_l]}\mathbb{C}(m|n)^{\otimes l}$.

We say that $N$ is of level $l$ as a $\mathfrak{gl}_{m|n}$-module if it is a direct sum of submodules of $\C(m|n)^{\otimes l}$. 
In the case that $l<(m+1)(n+1)$, the
functor 
\begin{equation*}
SW\colon M\mapsto M\otimes_{\mathbb{C}[S_l]}\mathbb{C}(m|n)^{\otimes l}
\end{equation*}
is an equivalence from the category of finite-dimensional right $\C[S_l]$-modules
to the category of finite-dimensional $U(\mathfrak{gl}_{m|n})$-modules of level $l$.
\end{enumerate}
\end{theorem}
\begin{remark}
\begin{enumerate}
\item The module $\C(m|n)^{\otimes l}$ is semisimple as a representation of $\mathfrak{gl}_{m|n}$.
\item Since we have the decomposition $\mathfrak{gl}_{m|n}=\mathfrak{sl}_{m|n}\oplus I_{m+n}$ for $m\neq n$, Theorem~\ref{chengbook} holds if we replace $\mathfrak{gl}$ with $\mathfrak{sl}$.
\end{enumerate}
\end{remark}

\section{Schur-Weyl duality for the Degenerate Affine Hecke algebra and the Super Yangian of $\mathfrak{gl}_{m|n}$}\label{SWRTT}

Classical Schur-Weyl duality was extended to the degenerate affine Hecke algebra of $S_l$ and the Yangian of $\mathfrak{sl}_n$ by V. Drinfeld in \cite{driJ} using the so-called $J$-presentation; the analogous result for the RTT-presentation of $Y(\mathfrak{gl}_n)$ can be found in the paper \cite{ara} by T. Arakawa. In the setup of classical Lie superalgebras, the Yangian of $\mathfrak{gl}_n$ is replaced by the super Yangian $Y(\mathfrak{gl}_{m|n})$, however, the degenerate affine Hecke algebra remains the same. 

\subsection{Degenerate Affine Hecke Algebra}
\begin{definition}\label{dahadef}
		Let $\kappa\in\C$. The degenerate affine Hecke algebra $H^{\deg}_{\kappa}(S_l)$ is the associative algebra with generators 
		$\sigma_1^{\pm 1}, \sigma_2^{\pm 1}, \dots, \sigma_{l-1}^{\pm 1}, \mathsf{u}_1,\mathsf{u}_2, \dots, \mathsf{u}_l$, and the following defining relations:
			\begin{align}
				\sigma_i \sigma_i^{-1} &= \sigma_i^{-1} \sigma_i=1,\label{affHecdef1}\\
				\sigma_i^2 &= 1,\label{affHecdef2}\\
				\sigma_i \sigma_j &= \sigma_j \sigma_i \text{ if } |i-j|>1,\label{affHecdef3}\\
			\sigma_i \sigma_{i+1} \sigma_i &= \sigma_{i+1}\sigma_i \sigma_{i+1},\label{affHecdef4}\\	
			\mathsf{u}_i \sigma_i &= \sigma_i \mathsf{u}_{i+1} -\kappa, \label{ldrna}\\ 
                \mathsf{u}_{i+1} \sigma_i&= \sigma_i \mathsf{u}_i+\kappa,\label{Idrna2}\\
               \mathsf{u}_i \sigma_j &= \sigma_j \mathsf{u}_{i}\text{ if } i\neq j,j+1, \label{ldrna3}  \\ 
			\mathsf{u}_i \mathsf{u}_j &= \mathsf{u}_j \mathsf{u}_i.	\label{affHecdef11}		
		 \end{align}
		\end{definition}

We identify the group algebra $\C[S_l]$ with the subalgebra of $H^{\deg}_{\kappa}(S_l)$ generated by $\sigma_1, \sigma_2, \dots, \sigma_{l-1}$. For $1\leq k\leq l$, we define elements $\mathsf{z}_k \in H^{\deg}_{\kappa}(S_l)$ by $\mathsf{z}_k=\sigma_{1,k}\cdot  \mathsf{u}_1 \cdot \sigma_{1,k}$. The element $\mathsf{z}_k$ has another presentation.
\begin{lemma}
The following relation holds:
\begin{align} \mathsf{z}_k= \mathsf{u}_k - \kappa\sum_{j<k}\sigma_{j,k}. \label{defyk}\end{align}
In paticular, we have $\mathsf{z}_k=\mathsf{u}_k+\sum_{1\le j<k}\sigma_{j,k}$ if $\kappa=-1$.
\end{lemma}

\begin{proof}			
We can prove the relation \eqref{defyk} by induction on $k$ using
\begin{align*}
\sigma_{1,k+1}\mathsf{u}_1 \sigma_{1,k+1} &= \sigma_{k,k+1}\sigma_{1,k} \sigma_{k,k+1} \mathsf{u}_1 \sigma_{k,k+1}\sigma_{1,k}\sigma_{k,k+1} = \sigma_{k,k+1}\sigma_{1,k}  \mathsf{u}_1 \sigma_{1,k}\sigma_{k,k+1} \\
 &= \sigma_{k,k+1} \left( \mathsf{u}_k - \kappa \sum_{1\le j<k} \sigma_{j,k} \right) \sigma_{k,k+1}
\end{align*}
for $k\geq 2$.
\end{proof}

\begin{comment}
\begin{proof}			
We can prove the relation \eqref{defyk} by induction on $k$.
In the case when $k=1$, \eqref{defyk} holds trivially.
	Let $k=2$. By the defining relation \eqref{ldrna}, we have
		\begin{align*}
		\mathsf{u}_1 \sigma_{1,2}&= \sigma_{1,2} \mathsf{u}_2+1.
		\end{align*}
Then, we obtain $\sigma_{1,2} \mathsf{u}_1 \sigma_{1,2}= \mathsf{u}_2 + \sigma_{1,2}$.
Assume that $\sigma_{1,k}\cdot  \mathsf{u}_1 \cdot \sigma_{1,k} = \mathsf{u}_k + \sum_{j<k}\sigma_{j,k}$ holds for some $k\geq 2$. We obtain
\begin{align*}
	\sigma_{1,k+1}\mathsf{u}_1 \sigma_{1,k+1} &= \sigma_{k,k+1}\sigma_{1,k} \sigma_{k,k+1} \mathsf{u}_1 \sigma_{k,k+1}\sigma_{1,k}\sigma_{k,k+1} \\
	&= \sigma_{k,k+1}\sigma_{1,k}  \mathsf{u}_1 \sigma_{1,k}\sigma_{k,k+1} \\
 	&= \sigma_{k,k+1} \left( \mathsf{u}_k + \sum_{j<k} \sigma_{j,k} \right) \sigma_{k,k+1} \\
	&= \sigma_{k,k+1} \mathsf{u}_k  \sigma_{k,k+1} + \sigma_{k,k+1} \left( \sum_{j<k} \sigma_{j,k} \right) \sigma_{k,k+1} \\
	&= \mathsf{u}_{k+1}  + \sum_{j\leq k} \sigma_{j,k+1},
\end{align*}
where the first equality is due to the definition of $S_l$, the second equality is due to \eqref{affHecdef2}, \eqref{ldrna3} and the assumption that $k\geq 2$, the third equaltiy is due to the induction hypothesis and the 5-th equality is due to \eqref{ldrna} and the definition of $S_l$.
\end{proof}
\end{comment}

\begin{lemma}
The elements $\mathsf{z}_i$ satisfy
\begin{gather}
\sigma \mathsf{z}_i = \mathsf{z}_{\sigma(i)} \sigma, \label{gather159}\\
[\mathsf{z}_i,\mathsf{z}_j]=-\kappa (\mathsf{z}_i-\mathsf{z}_j)\sigma_{i,j}\label{gather160}
\end{gather}
for all $\sigma\in S_l$ and $1\leq i,j\leq l$.
\end{lemma}

\begin{proof}
The relation \eqref{gather159} can be proven by the induction on the length of $\sigma$. For \eqref{gather160}, we verify it first when $i=1,j=2$, and extend to arbitrary $i,j$ using \eqref{gather159} and the permutation $\sigma_{1,i},\sigma_{2,j}$:
	\begin{align}
	[\mathsf{z}_1,\mathsf{z}_2]&=\mathsf{z}_1 \mathsf{z}_2 - \mathsf{z}_2 \mathsf{z}_1 =\mathsf{u}_1 \sigma_{1,2} \mathsf{u}_1 \sigma_{1,2} - \sigma_{1,2} \mathsf{u}_1 \sigma_{1,2} \mathsf{u}_1 = \mathsf{u}_1 \sigma_{1,2} \left(\sigma_{1,2} \mathsf{u}_2 -\kappa \right) - \sigma_{1,2} \left( \sigma_{1,2} \mathsf{u}_2 -\kappa \right)\mathsf{u}_1\nonumber\\
		&= \mathsf{u}_1 \mathsf{u}_2 -\kappa \mathsf{u}_1 \sigma_{1,2} - \mathsf{u}_2 \mathsf{u}_1 + \kappa\sigma_{1,2}\mathsf{u}_1 = -\kappa \mathsf{u}_1 \sigma_{1,2} +\kappa \sigma_{1,2} \mathsf{u}_1 \nonumber\\
  &= -\kappa \left(\mathsf{u}_1  - \sigma_{1,2} \mathsf{u}_1 \sigma_{1,2} \right) \sigma_{12} =-\kappa\left( \mathsf{z}_1 - \mathsf{z}_2\right) \sigma_{1,2}.\label{gather161}
	\end{align}
Conjugating the equality \eqref{gather161} by $\sigma_{2,j}$ and $\sigma_{1,i}$ gives
\begin{equation}
\sigma_{1,i} \sigma_{2,j} [\mathsf{z}_1,\mathsf{z}_2] \sigma_{2,j}\sigma_{1,i} = -\kappa\sigma_{1,i} \sigma_{2,j} (\mathsf{z}_1-\mathsf{z}_2)\sigma_{1,2}\sigma_{2,j}\sigma_{1,i}.\label{y1y2}
\end{equation}
The left-hand side of \eqref{y1y2} can be rewritten as $[\mathsf{z}_i,\mathsf{z}_j]$, and its right-hand side of \eqref{y1y2} can be rewritten as $-\kappa(\mathsf{z}_i - \mathsf{z}_j )\sigma_{i,j}$. \end{proof}

\begin{proposition}\label{prop201}
The relations \eqref{defyk}, \eqref{gather159} and \eqref{gather160} induce the defining relations \eqref{ldrna}-\eqref{affHecdef11}.
\end{proposition}
\begin{proof}
The defining relations \eqref{ldrna}-\eqref{ldrna3} follow from the relation \eqref{defyk} and \eqref{gather159}. Thus, it is enough to show that \eqref{affHecdef11} is derived from \eqref{gather160} in the case when $i<j$. By \eqref{defyk}, we can rewrite \eqref{gather160} as
\begin{equation}
[\mathsf{u}_i,\mathsf{u}_j]-\kappa[\mathsf{u}_i,\sum_{j>k}\sigma_{j,k}]+\kappa[\mathsf{u}_j,\sum_{i>k}\sigma_{i,k}]+\kappa^2[\sum_{i>k}\sigma_{i,k},\sum_{j>k}\sigma_{j,k}]=-\kappa(\mathsf{z}_i-\mathsf{z}_j)\sigma_{i,j}.\label{gather160.1}
\end{equation}
We obtain
\begin{align}
[\mathsf{u}_i,\sum_{j>k}\sigma_{j,k}]&=[\mathsf{z}_i+\kappa\sum_{i>k}\sigma_{i,k},\sum_{j>k}\sigma_{j,k}]\nonumber\\
&=[\mathsf{z}_i,\sum_{j>k}\sigma_{j,k}]+\kappa[\sum_{i>k}\sigma_{i,k},\sum_{j>k}\sigma_{j,k}]=(\mathsf{z}_i-\mathsf{z}_j)\sigma_{i,j}+\kappa[\sum_{i>k}\sigma_{i,k},\sum_{j>k}\sigma_{j,k}]\label{gather160.2}
\end{align}
and
\begin{align}
[\mathsf{u}_j,\sum_{i>k}\sigma_{i,k}]&=[\mathsf{z}_j+\kappa\sum_{j>k}\sigma_{j,k},\sum_{i>k}\sigma_{i,k}]\nonumber\\
&=[\mathsf{z}_j,\sum_{i>k}\sigma_{i,k}]+\kappa[\sum_{j>k}\sigma_{j,k},\sum_{i>k}\sigma_{i,k}]=0+\kappa[\sum_{j>k}\sigma_{j,k},\sum_{i>k}\sigma_{i,k}],\label{gather160.3}
\end{align}
where the first equalities of \eqref{gather160.2} and \eqref{gather160.3} are due to \eqref{defyk} and the second equalities of \eqref{gather160.2} and \eqref{gather160.3} are due to \eqref{gather159}.
By the definition of the symmetric group, since $i<j$, we have
\begin{align}
[\sum_{i>k}\sigma_{i,k},\sum_{j>k}\sigma_{j,k}]&=\sum_{i>k}[\sigma_{i,k},\sigma_{j,i}]+\sum_{i>k}[\sigma_{i,k},\sigma_{j,k}]=0.\label{gather160.4}
\end{align}
By applying \eqref{gather160.2}-\eqref{gather160.4} to \eqref{gather160.1}, we obtain $[\mathsf{u}_i,\mathsf{u}_j]=0$.
\end{proof}

\subsection{The RTT presentation of the Super Yangian}

The super Yangian associated with $\mfgl_{m|n}$ was first introduced in \cite{N}. We will follow the notation used by Gow \cite{Gow}.

\begin{definition} 
The super Yangian $Y(\mfgl_{m|n})$ is the associative superalgebra generated by the elements $t_{i,j}^{(r)}$ for $1\leq i,j\leq m+n$ and $r \in \Z_{>0}$ satisfying:
\begin{gather}
[t_{i,j}(u), t_{k,l}(v)] = \frac{(-1)^{\nu (i,j;k,l)} }{(u-v)} \left( t_{k,j}(u)  t_{i,l}(v) - t_{k,j}(v) t_{i,l}(u) \right), \label{alizarin} \\
\nu (i,j;k,l)=|i||j|+|i||k|+|j||k|,
\end{gather}
where the generating power series $t_{i,j}(u)$ is defined by
\begin{equation*}
t_{i,j}(u)=\delta_{i,j}+\sum_{r\geq1}\limits t^{(r)}_{i,j}u^{-r}
\end{equation*}
and the parity of $t^{(r)}_{i,j}$ is $|i|+|j|$.
Expanding (\ref{alizarin}) yields useful explicit relations for the generators:
\begin{align}
[t_{i,j}^{(r)},t_{k,l}^{(s)}]&=  (-1)^{\nu (i,j;k,l)} \sum_{a=0}^{\textrm{min}(r,s)-1} 
		\left( t_{k,j}^{(a)}  t_{i,l}^{(r+s-1-a)} - t_{k,j}^{(r+s-1-a)} t_{i,l}^{(a)} \right), \label{RTT2}\\
[t_{i,j}^{(r+1)}, t_{k,l}^{(s)}] - [t_{ij}^{(r)}, t_{k,l}^{(s+1)}] &=  (-1)^{\nu (i,j;k,l)} \left(t_{k,j}^{(r)} t_{i,l}^{(s)} -t_{k,j}^{(s)} t_{i,l}^{(r)} \right),
\end{align}
where we set $t^{(0)}_{ij}=\delta_{i,j}$.
\end{definition}
Note the special case when $r=1$: \[ [t_{i,j}^{(1)},t_{k,l}^{(s)}]=\left(\delta_{k,j} (-1)^{|j|}t_{i,l}^{(s)} - \delta_{i,l}(-1)^{|i|} t_{k,j}^{(s)} \right).\]
This is compatible with the embedding $U(\mfgl_{m|n}) \hookrightarrow Y(\mfgl_{m|n})$ given by $E_{i,j} \mapsto (-1)^{|i|} t_{i,j}^{(1)}$. 

This presentation of the super Yangian is called the RTT presentation, due to the generators $t_{i,j}^{(r)}$ satisfying the following relation in $Y(\mfgl_{m|n}) \otimes \End(\C(m|n)) \otimes \End(\C(m|n))(u,v)$:
\begin{equation} T_2(v) T_1(u) R_{1,2}(u-v) = R_{1,2}(u-v)T_1(u)T_2(v), \end{equation}
where
\begin{align}
P&=\sum_{a,b=1}^{m+n} (-1)^{|b|} E_{a,b} \otimes E_{b,a}\in\End(\C(m|n)) \otimes \End(\C(m|n)),\\
R_{1,2}(u) &=\id\otimes\left( 1 - \frac{P}{u}\right),\\
T(u) &= \sum_{i,j=1}^{m+n} (-1)^{|j|(|i|+1)} t_{i,j}(u) \otimes E_{i,j},\\
T_1(u) &= \sum_{i,j=1}^{m+n} (-1)^{|j|(|i|+1)} t_{i,j}(u) \otimes E_{i,j}\otimes\id,\\
T_2(u) &= \sum_{i,j=1}^{m+n} (-1)^{|j|(|i|+1)} t_{i,j}(u) \otimes \id\otimes E_{i,j}.
\end{align}

\subsection{Schur-Weyl functor between the degenerate affine Hecke algebras and the super Yangian}\label{SWdahaY}

Hereafter, let us denote $H^{\textrm{deg}}_{-1}(S_l)$ by $H^{\textrm{deg}}(S_l)$. The goal of this subsection is to produce the functor $SW$ from the category of right $H^{\textrm{deg}}(S_l)$-modules to the category of left $Y(\mfgl_{m|n})$-modules determined on objects by
\begin{equation*}
M\mapsto SW(M)=M\otimes_{\C[S_l]} \C(m|n)^{\otimes l}
\end{equation*}
for any right $H^{\textrm{deg}}(S_l)$-module $M$.

Fix a right $H^{\textrm{deg}}(S_l)$-module $M$. Since $H^{\textrm{deg}}(S_l) \simeq \C[\mathsf{u}_1,\dots, \mathsf{u}_l] \otimes_{\C} \C[S_l]$ as vector spaces, $\mathsf{u}_i$ can be seen as an operator on $M$. Thus, we find that $R(u-\mathsf{u}_i)=1 - \dfrac{P}{u-\mathsf{u}_i} = 1 -\dfrac{u^{-1} P}{1-\mathsf{u}_i u^{-1}}$ belongs to $\End(M)\otimes_{\C}\textrm{End}(\C(m|n))^{\otimes 2}[[u^{-1}]]$. Set
\begin{equation*}
R_{i,l+1}(u-\mathsf{u}_i )= 1 -\frac{P_{i,l+1}}{u-\mathsf{u}_i},
\end{equation*}
where
\begin{equation*}
P_{i,l+1}= \sum_{a,b=1}^{m+n} (-1)^{|b|} \textrm{id}_M \otimes \textrm{id}^{\otimes i-1} \otimes E_{a,b} \otimes \textrm{id}^{\otimes l-i} \otimes E_{b,a}.
\end{equation*}
Let us denote $(-1)^{|b|+|a||b|}E_{a,b}$ by $E_{b,a}^t$. For $A=A_1\otimes\cdots \otimes A_l\in\End(\C(m|n))^{\otimes l}$, we define $A^{t_k}\in\End(\C(m|n))^{\otimes l}$ by
\begin{equation}
(A_1\otimes\cdots \otimes A_l)^{t_k}=A_1\otimes \cdots\otimes A_{k-1}\otimes A_k^t\otimes A_{k+1}\otimes\cdots\otimes A_l.\label{transpose}
\end{equation}
By the definition of $A^{t_k}$, we find that
\begin{align}
P^{t_k}_{k,l+1}&= \sum_{a,b=1}^{m+n} (-1)^{|b|} \textrm{id}_M \otimes \textrm{id}^{\otimes k-1} \otimes E_{a,b}^{t}\otimes \textrm{id}^{\otimes l-k} \otimes E_{b,a}\nonumber\\
&=\sum_{a,b=1}^{m+n} (-1)^{|a|+|b|+|a||b|} \textrm{id}_M \otimes \textrm{id}^{\otimes k-1} \otimes E_{b,a} \otimes \textrm{id}^{\otimes l-k} \otimes E_{b,a}.\label{transpose-1}
\end{align}
\begin{theorem}
For a right $H^{\textrm{deg}}(S_l)$-module $M$, we can define an action of $Y(\mfgl_{m|n})$ on $M\otimes_{\C} \C(m|n)^{\otimes l}$ by 
\begin{equation} \label{thrips} 
T(u) \mapsto R_{1,l+1}^{t_1}(\mathsf{u}_1-u) R_{2,l+1}^{t_2}(\mathsf{u}_2-u) \cdots R_{l,l+1}^{t_l}(\mathsf{u}_l-u), 
\end{equation}
where
\begin{gather*}
R_{i,l+1}^{t_i}(\mathsf{u}_i-u)=1 + \frac{P_{i,l+1}^{t_i}}{u-\mathsf{u}_i}.
\end{gather*}
\end{theorem}
\begin{proof}
By the definition \eqref{transpose}, we obtain
\begin{align}
&\quad\left( (E_{a,b}\otimes E_{b,a}\otimes 1)(E_{c,d}\otimes 1 \otimes E_{d,c})\right)^{t_1}=  (-1)^{(|a|+|b|)(|c|+|d|)}(E_{a,b} E_{c,d})^{t_1} \otimes E_{b,a}\otimes E_{d,c} \nonumber\\
&=E_{c,d}^{t_1} E_{a,b}^{t_1} \otimes E_{b,a} \otimes E_{d,c} = (E_{c,d}\otimes 1 \otimes E_{d,c})^{t_1} (E_{a,b}\otimes E_{b,a}\otimes 1)^{t_1}.\label{244-1}
		\end{align}
By the relation $R_{1,2}(u-v)R_{1,3}(u-w)R_{2,3}(v-w)=R_{2,3}(v-w)R_{1,3}(u-w)R_{1,2}(u-v)$ and \eqref{244-1}, we obtain
\begin{equation} R_{1,3}^{t_1}(u-w)R_{1,2}^{t_1}(u-v)R_{2,3}(v-w)=R_{2,3}(v-w)R_{1,2}^{t_1}(u-v)R_{1,3}^{t_1}(u-w). \label{heatwave} \end{equation}
\begin{comment}
Since we find that
\begin{align*}
			R_{2,3}(w-v) R_{2,3}(v-w) &= \left( 1- \frac{P}{w-v} \right)\left( 1+ \frac{P}{w-v} \right)= 1- \frac{1}{(w-v)^2}
		\end{align*}
by a direct computation, we obtain
\begin{equation}
			R_{23}(v-w) R_{13}^{t_1}(w-u) R_{12}^{t_1} (v-u) =  R_{12}^{t_1} (v-u) R_{13}^{t_1}(w-u)R_{23}(v-w) \label{heatwave-1}
		\end{equation}
by multiplying (\ref{heatwave}) by $R_{23}(v-w)$ on both sides.
\end{comment}

In the tensor product
		$$\End (M)\otimes_{\C} \End_{\C}(\C(m|n))^{\otimes l+2},$$
		we refer to the first factor as the $0^{th}$ position and the following ones as the $1^{st}$ to $(l+2)^{th}$ positions. 
  Let us set
\begin{align*}
T_{l+1}(u)&=\sum_{i,j=1}^{m+n} (-1)^{|j|(|i|+1)} t_{i,j}(u) \otimes\id^{\otimes l}\otimes E_{i,j}\otimes\id,\\
T_{l+2}(u)&=\sum_{i,j=1}^{m+n} (-1)^{|j|(|i|+1)} t_{i,j}(u) \otimes\id^{\otimes l}\otimes\id\otimes E_{i,j}.
\end{align*}
It is enough to show that
\begin{equation}
R_{l+1,l+2}(u-v)T_{l+1}(u)T_{l+2}(v)=T_{l+2}(v)T_{l+1}(u)R_{l+1,l+2}(u-v).\label{transpose3}
\end{equation}
  By using (\ref{heatwave}), we obtain
		\begin{align*} 
&\quad R_{l+1,l+2}(u-v)T_{l+1}(u)T_{l+2}(v)\\
&=R_{l+1,l+2}(u-v)\Big(R_{1,l+1}^{t_1} (\mathsf{u}_1-u)   \cdots R_{l,l+1}^{t_l} (\mathsf{u}_l-u)\Big)\Big(R_{1,l+2}^{t_1}(\mathsf{u}_1-v)\cdots R_{l,l+2}^{t_l}(\mathsf{u}_l-v)\Big) \text{ by } \eqref{thrips}\\
&=R_{l+1,l+2}(u-v)\Big(R_{1,l+1}^{t_1} (\mathsf{u}_1-u)    R_{1,l+2}^{t_1} (\mathsf{u}_1-v)\Big) \cdots\Big(R_{l,l+1}^{t_l}(\mathsf{u}_l-u)\cdots R_{l,l+2}^{t_l}(\mathsf{u}_l-v)\Big) \\
&=\Big( R_{1,l+2}^{t_1} (\mathsf{u}_1-v) R_{1,l+1}^{t_1} (\mathsf{u}_1-u) \Big) \cdots\Big(R_{l,l+2}^{t_l}(\mathsf{u}_l-v)\cdots R_{l,l+1}^{t_l}(\mathsf{u}_l-u)\Big)R_{l+1,l+2}(u-v) \text{ by } \eqref{heatwave}\\
&=T_{l+2}(v)T_{l+1}(u)R_{l+1,l+2}(u-v) \text{ by } \eqref{thrips}.
\end{align*}
\end{proof}
The map (\ref{thrips}) can be simplified as in Proposition 5 in \cite{ara} or Proposition 5.2 in \cite{nazQ}. For the preparation, we note the following lemma.
\begin{lemma}\label{P-1}
The following relation holds:
\begin{equation} 
P_{i,l+1}^{t_i} P_{k,l+1}^{t_k} = P_{i,k}  P_{k,l+1}^{t_k}. \label{llbean} \end{equation}
\end{lemma}
\begin{proof}
It is enough to check this for three different indices. We show the case $i=1,k=2,l+1=3$.
By a direct computation, we obtain
  		\begin{equation}
	P_{1,3}^{t_1} P_{2,3}^{t_2} = \sum_{a,b,c}(-1)^{|a|+|b|+|c|+|b||c|} E_{b,a}^{(1)} \otimes E_{a,c}^{(2)} \otimes E_{b,c}^{(3)} = P_{1,2} P_{2,3}^{t_2}.\label{PPP1}
		\end{equation}	
\begin{comment}
  		On the other hand, by a direct computation, we also obtain	
		\begin{align}
		P_{12}&  P_{23}^{t_2} = \left( \sum_{a,d}(-1)^{|a|}E_{d,a}^{(1)} \otimes E_{a,d}^{(2)} \otimes 1 \right)
			\left( \sum_{b,c} (-1)^{|b||c|+|b|+|c|} 1 \otimes E_{b,c}^{(2)} \otimes E_{b,c}^{(3)} \right)\nonumber\\
		&= \sum_{a,b,c,d} (-1)^{|a|+|b|+|c|+|b||c|}\delta_{b,d} E_{d,a}^{(1)} \otimes E_{a,c}^{(2)} \otimes E_{b,c}^{(3)}.\label{PPP2}
		\end{align}
Since the right hand sides of \eqref{PPP1} and \eqref{PPP2} coincide with each other, we have proven \eqref{llbean}.
\end{comment}
\end{proof}
Let $V$ be the subspace of $M \otimes_{\C} \C(m|n)^{\otimes l}$ spanned by $\wt{\mathbf{m}}\sigma\otimes\wt{\underline{\mathbf{v}}} - \wt{\mathbf{m}} \otimes \sigma \wt{\underline{\mathbf{v}}}$ for all $\sigma \in S_l$, $\wt{m} \in M$. Let us set $SW(M):= (M\otimes_{\C}\C(m|n)^{\otimes l})/V=M\otimes_{\mathbb{C}[S_l]}\C(m|n)^{\otimes l}$.

\begin{proposition} \label{glerup}
The difference between (\ref{thrips}) and $I+\sum_{k=1}^l \dfrac{1}{u-\mathsf{z}_k} P^{t_k}_{k,l+1}$ as operators on $M \otimes_{\C} \C(m|n)^{\otimes l}$
\begin{equation}
		T(u) - \left(I + \sum_{k=1}^l \frac{1}{u-\mathsf{z}_k} P_{k,l+1}^{t_k}\right)
		\end{equation}
maps $M \otimes_{\C} \C(m|n)^{\otimes l}$ to $V$.
	\end{proposition}
This is essentially the same result as Lemma 4.6 in \cite{LuMu}.

Let us denote $1^{\otimes (k-1)}\otimes E_{a,b}\otimes1^{\otimes(l-k)}$ by $E^{(k)}_{a,b}$. By Proposition~\ref{glerup}, as an operator on $SW(M)$, we can write that $$T(u)=\sum_{k=1}^l \left(\frac{1}{u-\mathsf{z}_k}\right) P_{k,l+1}^{t_k},$$ which gives $$t_{ab}^{(2)}=\sum_{k=1}^l (-1)^{|a|}  \mathsf{z}_k \otimes E_{ab}^{(k)}.$$ The following theorem is naturally derived from Proposition~\ref{glerup}.
\begin{theorem}\label{SW-tt}
We can define a functor $SW$ from the category of right modules over $H^{deg}(S_l)$ to the category of left modules over the super Yangian $Y(\mathfrak{gl}_{m|n})$ by
\begin{equation*}
M\mapsto SW(M)=M \otimes_{\C[S_l]} \C(m|n)^{\otimes l},\ \varphi\mapsto\varphi\otimes\id^{\otimes l}
\end{equation*}
for a right  $H^{deg}(S_l)$-module $M$ and any homomorphism $\varphi:M_1\rightarrow M_2$ between two right $H^{deg}(S_l)$-modules $M_1$ and $M_2$.
\end{theorem}

\subsection{Equivalence of categories}\label{CPequivcat}

The goal of this section is to strengthen Theorem \ref{SW-tt} using ideas from \cite{CP96}: see Theorem \ref{floridakilos-1}. 
\begin{definition}
We say that a $Y(\mathfrak{gl}_{m|n})$-module $N$ is of level $l$ if $N$ is of level $l$ as a $\mathfrak{gl}_{m|n}$-module. 
\end{definition}
\begin{theorem}\label{floridakilos-1}
		Let $N$ be  a $Y(\mathfrak{gl}_{m|n})$-module that is finite-dimensional of level $l$. Suppose that $m,n\geq l+1$. Then, there exists a module $M$ over the degenerate affine Hecke algebra $H^{\deg}(S_l)$ such that \[N\cong M \otimes_{\C[S_l]} \C(m|n)^{\otimes l}=SW(M).\]
\end{theorem}
\begin{remark}
In \cite{CP96}, the authors describe a functor from the category of finite-dimensional representations of the affine Hecke algebra of $S_l$ to the category of finite-dimensional representations of the quantum loop algebra $U_q(\mc{L}\mf{sl}_n)$. In particular, if $l <n$, this functor is an equivalence of categories between all finite-dimensional representations of the affine Hecke algebra of $S_l$ and a certain subcategory of representations of $U_q(\mc{L}\mf{sl}_n)$. 
Theorem~\ref{floridakilos-1} is an analogue of Theorem 4.2 and Lemma 4.5 in \cite{CP96}.
\end{remark}
The rest of this section is devoted to the proof of Theorem~\ref{floridakilos-1}.
The following lemma is used for the proof of Theorem~\ref{floridakilos-1}.
\begin{lemma} \label{chp} (Proposition 14.1 in \cite{Flicker}) Suppose that $\underline{\mathbf{w}}=e_{i_1} \otimes \cdots \otimes e_{i_l}$ for $i_j$ all distinct. Then, $\underline{\mathbf{w}}$ generates $\C(m|n)^{\otimes l}$ under the action of $U(\mathfrak{gl}_{m|n})$.
\end{lemma}

Since $N$ is finite-dimensional of level $l$,  $N$ is the direct sum of certain $\mathfrak{gl}_{m|n}$-modules $N_i$, where each $N_i$ is an irreducible representation of $\mf{gl}_{m|n}$ contained in $\C(m|n)^{\otimes l}$. By Theorem~\ref{chengbook} (Schur-Sergeev Duality), there exists a right $\mathbb{C}[S_l]$-module $M$ satisfying 
\begin{equation*}
N=M \otimes_{\C[S_l]} \C(m|n)^{\otimes l}.
\end{equation*}
It is enough to show that $M$ is a right module over $H^{\textrm{deg}}(S_l)$. Let $1 \le k \le l$ and $1\le a, b \le m+n$, and choose indices $i_1,\ldots,i_{k-1}, i_{k+1},\ldots,i_l$ such that $i_1,\ldots,i_{k-1}, a, i_{k+1},\ldots,i_l$ are all distinct and so are $i_1,\ldots,i_{k-1}, b, i_{k+1},\ldots,i_l$. As in \cite{CP96}, we introduce two elements in $\mathbb{C}(m|n)^{\otimes l}$:
		$$ \underline{\mathbf{v}}^{(k)} = \left(\bigotimes_{j=1}^{k-1}e_{i_j}\right)\otimes e_b \otimes \left(\bigotimes_{j=k+1}^{l}e_{i_j}\right),\  \underline{\mathbf{w}}^{(k)} =\left(\bigotimes_{j=1}^{k-1}e_{i_j}\right)\otimes e_a \otimes \left(\bigotimes_{j=k+1}^{l}e_{i_j}\right).$$

	The $\mf{h}$-weight of $\underline{\mathbf{v}}^{(k)}$ is $\sum_{i=1}^{k-1}\limits\eps_i+\eps_b+\sum_{j=k+1}^{l}\limits\eps_j$, while the $\mf{h}$-weight of $t_{a,b}^{(r)}(\mathbf{m}\otimes \underline{\mathbf{v}}^{(k)})$ for $\mathbf{m} \in M$ is $\eps_a -\eps_b + \sum_{i=1}^{k-1}\limits\eps_i+\eps_b + \sum_{j=k+1}^{l}\limits\eps_j$, which equals $\sum_{i=1}^{k-1}\limits\eps_i+\eps_a+\sum_{j=k+1}^{l}\limits\eps_j$, the $\mf{h}$-weight of $\underline{\mathbf{w}}^{(k)}$. 
Let us set
\begin{equation*}
\widetilde{S}=\{\widetilde{\underline{\mathbf{w}}}=\bigotimes_{i=1}^le_{a_i}|1\leq a_i\leq m+n,\text{ $\widetilde{\underline{\mathbf{w}}}$ has a weight $\sum_{i=1}^{k-1}\limits\eps_i+\eps_a+\sum_{j=k+1}^{l}\limits\eps_j$}\}.
\end{equation*}
We can write down
\begin{equation*}
t_{a,b}^{(r)}(\mathbf{m} \otimes \underline{\mathbf{v}}^{(k)}) = \sum_{\widetilde{\underline{\mathbf{w}}}\in\widetilde{S}}\limits \widetilde{\mathbf{m}}_{\widetilde{\underline{\mathbf{w}}}}\otimes\widetilde{\underline{\mathbf{w}}}
\end{equation*}
for some elements $\widetilde{\mathbf{m}}_{\widetilde{\underline{\mathbf{w}}}}\in M$. By the definition of the action of $S_l$ on $\mathbb{C}(m|n)^{\otimes l}$, any element $\widetilde{\underline{\mathbf{w}}}$ in $\widetilde{S}$ can be written as $\tau(\underline{\mathbf{w}}^{(k)})$ for some $\tau\in S_l$. Thus, we obtain
\begin{align*}
\sum_{\widetilde{\underline{\mathbf{w}}}\in\widetilde{S}}\limits \widetilde{\mathbf{m}}_{\widetilde{\underline{\mathbf{w}}}}\otimes\widetilde{\underline{\mathbf{w}}}&=\sum_{\tau\in S_l}  \widetilde{\mathbf{m}}_{\tau} \otimes \tau(\underline{\mathbf{w}}^{(k)})=\sum_{\tau\in S_l}\mathbf{m}_{\tau} \otimes\underline{\mathbf{w}}^{(k)}
\end{align*}
for some elements $\mathbf{m}_{\tau}\in M$. We thus obtain a linear map $\mathbf{m}\mapsto\sum_{\tau\in S_l} \mathbf{m}_{\tau}$ and denote it by $(-1)^{|a|}\chi^{a,b}_{k,r}$. In particular, when $r=2$, we sometimes denote $\chi^{a,b}_{k,2}$ by $\chi^{a,b}_{k}$. It has the property that $$t_{a,b}^{(r)}(\mathbf{m} \otimes \underline{\mathbf{v}}^{(k)}) = (-1)^{|a|} \chi^{a,b}_{k,r}(\mathbf{m}) \otimes \underline{\mathbf{w}}^{(k)}.$$

\begin{lemma}\label{idbits-lem}
			(Analogue of Lemma 4.5 \cite{CP96}) 
			Assume that $m,n\geq l+1$. For all $\underline{\mathbf{v}} \in \C(m|n)^{\otimes l}$, we obtain
			\begin{align} t_{a,b}^{(2)}(\mathbf{m} \otimes \underline{\mathbf{v}}) = (-1)^{|a|}\sum_{j=1}^l \chi_j^{a,b}(\mathbf{m}) \otimes E_{a,b}^{(j)} (\underline{\mathbf{v}}). \label{idbits} \end{align}
		\end{lemma}
\begin{proof}
Let $\underline{\mathbf{v}}$ be $e_{i_1} \otimes \cdots \otimes e_{i_l}$. In the case that $\{i_1,\cdots,i_l\}$ does not contain $b$, we obtain $t_{ab}^{(2)}(\mathbf{m}\otimes \underline{\mathbf{v}})=0$ by weight considerations. We consider the case that $\{i_1,\cdots,i_l\}$ contains $b$. Let $S_a$ (respectively $S_b$) be the subset of $\{1,\cdots,l\}$ satisfying that $j\in S_a\ (\text{respectively }S_b)$ if and only if $e_{i_j}=e_a\ (\text{respectively }e_b)$. By Lemma~\ref{chp} and the defining relation of $Y(\mathfrak{gl}_{m|n})$, it is enough to consider the case when the $e_i\ (i\not\in S_a\cup S_b)$ are all distinct. 

Set $r=|S_a|$ and $s=|S_b|$. Let us label the elements in the sets $S_a$ and $S_b$ by $j_1,\ldots,j_r$ and $j'_1, \ldots, j'_s$ such that 
\begin{equation*}
S_a=\{j_p\mid 1\leq p\leq r,j_p<j_{p+1}\},\ S_b=\{j'_p\mid 1\leq p\leq s,j'_p<j'_{p+1}\}.
\end{equation*}
We prove the formula (\ref{idbits}) by double induction on $r$ and $s$.
	
First, we show the case when $s=1$. If $r=0$ and $s=1$, (\ref{idbits}) follows from the definition of $\chi_k^{ab}$. Assume the formula \eqref{idbits} holds with $s=1$ and up to $r-1$. Let us consider $j_1<j_2<\cdots<j_r$. Pick $c\neq a,b$. We take $\underline{\mathbf{v}'}=\bigotimes_{j=1}^le_{i_j}$ satisfying that $e_{i_{j_p}}=e_a$ for $1\leq p\leq r-1$, $e_{i_{j'_1}}=e_b$, $e_{i_{j_r}}=e_c$ and $e_{i_j}\neq e_a,e_b,e_c$ if $j\neq j'_1,j_p$. It is enough to show that \eqref{idbits} hold when $\underline{\mathbf{v}}$ equals $E_{a,c}\underline{\mathbf{v}'}$.

Since we have $(-1)^{|a|}E_{a,c}=t_{a,c}^{(1)}$ and $[t_{a,b}^{(2)},E_{a,c}]=0$ by \eqref{RTT2}, we obtain
		\begin{equation} t_{a,b}^{(2)} (\mathbf{m} \otimes \underline{\mathbf{v}}) =t_{a,b}^{(2)} E_{a,c}(\mathbf{m}\otimes \underline{\mathbf{v}}') =  (-1)^{(|a|+|c|)(|a|+|b|)} E_{ac} t_{ab}^{(2)}(\mathbf{m}\otimes\underline{\mathbf{v}}') \label{rel1}
   \end{equation}
and 
\begin{align}
E_{a,c} t_{a,b}^{(2)}(\mathbf{m}\otimes\underline{\mathbf{v}}')&=(-1)^{|a|}E_{a,c}(\chi_{j_1'}^{a,b}(\mathbf{m})\otimes E_{a,b}^{(j_1')}(\underline{\mathbf{v}}')) =  (-1)^{|a|+}\chi_{j_1'}^{a,b}(\mathbf{m})\otimes E_{a,c}E^{(j_1')}_{a,b}(\underline{\mathbf{v}}')\nonumber \\
&= (-1)^{|a|+(|a|+|b|)(|a|+|c|)} \chi_{j_1'}^{a,b}(\mathbf{m})\otimes E_{a,b}^{(j_1')}E_{a,c}\underline{\mathbf{v}}' =(-1)^{|a|+(|a|+|b|)(|a|+|c|)} \chi_{j_1'}^{a,b}(\mathbf{m})\otimes E_{a,b}^{(j_1')}\underline{\mathbf{v}},\label{rel2}
\end{align}
where the first equality is due to the induction hypothesis. Applying \eqref{rel2} to \eqref{rel1}, we have obtained
$$t^{(2)}_{a,b}(\mathbf{m}\otimes\underline{\mathbf{v}}) =  (-1)^{|a|}\chi_{j_1'}^{ab}(\mathbf{m})\otimes E_{a,b}^{(j_1')}(\underline{\mathbf{v}}). $$
Thus, we have proved that \eqref{idbits} holds in the case when $s=1$.
		
Assume that \eqref{idbits} holds up to $s-1$ of $e_b$. By the assumption that $m,n\geq l+1$, we can choose $c\neq a,b$ such that $|b|=|c|$. We take $\underline{\mathbf{v}'}=\bigotimes_{j=1}^l e_{i_j}$ satisfying that $e_{i_{j_p}}=e_a$ for $1\leq p\leq r$, $e_{i_{j'_1}}=e_{i_{j'_2}}=e_c$, $e_{i_{j'_p}}=e_b$ for $3\leq p\leq s$ and $e_{i_j}\neq e_a,e_b,e_c$ if $i\neq j_p,j'_p$. 
Observe that
\begin{equation}
\underline{\mathbf{v}} =\frac{1}{2} E_{b,c}^{2} (\underline{\mathbf{v}}').\label{gazz93}
\end{equation}
and let us show that \eqref{idbits} holds for $\underline{\mathbf{v}}$.

Since $a,b,c$ are all distinct and $(-1)^{|b|}E_{b,c}=t_{b,c}^{(1)}$ holds, we have 
$[t_{b,c}^{(1)},[t_{b,c}^{(1)},t_{a,b}^{(1)}]]=0$. Further, by the defining relations of the super Yangian, we obtain $[t_{b,c}^{(1)},[t_{b,c}^{(1)},t_{a,b}^{(2)}]]=0$. This expands to 
\begin{equation}
0=[t_{b,c}^{(1)},[t_{b,c}^{(1)},t_{a,b}^{(2)}]] =(t_{b,c}^{(1)} )^2 t_{a,b}^{(2)} -2t_{b,c}^{(1)} t_{a,b}^{(2)} t_{b,c}^{(1)}+  t_{a,b}^{(2)} (t_{b,c}^{(1)})^2,\label{gazz94}
\end{equation}
since, by assumption $|b|=|c|$.	
By \eqref{gazz93} and \eqref{gazz94}, we obtain
	\begin{align}
	t_{a,b}^{(2)} &(\mathbf{m}\otimes \underline{\mathbf{v}})= 
	 \frac{1}{2} t_{a,b}^{(2)} (t_{b,c}^{(1)})^2 (\mathbf{m} \otimes \underline{\mathbf{v}}' ) = \left( t_{b,c}^{(1)} t_{ab}^{(2)} t_{b,c}^{(1)}  - \frac12
	 (t_{b,c}^{(1)})^2 t_{a,b}^{(2)} \right) (\mathbf{m} \otimes \underline{\mathbf{v}}').\label{gazz95}
	\end{align}
Since $\underline{\mathbf{v}}'$ has $(s-1)$ vectors equal to $e_b$ in the tensor product, 
we obtain
	\begin{align}
		(t_{b,c}^{(1)})^2 t_{a,b}^{(2)} (\mathbf{m} \otimes \underline{\mathbf{v}}') 
		&= (-1)^{|a|}(t_{b,c}^{(1)})^2 \sum_{k=1}^l \chi^{a,b}_{k} (\mathbf{m})
				\otimes E_{a,b}^{(k)} (\underline{\mathbf{v}}') = 2  (-1)^{|a|}\sum_{t \ge 3}  \chi^{a,b}_{j'_t} (\mathbf{m})
		\otimes E_{a,b}^{(j_t')}(\underline{\mathbf{v}})\label{gazz96}
	\end{align}
by induction.
Let us set
\begin{gather*}
\underline{\mathbf{v}}''=E^{(j_1')}_{b,c}\underline{\mathbf{v}}',\ 
\underline{\mathbf{v}}'''=E^{(j_2')}_{b,c}\underline{\mathbf{v}}'.
\end{gather*}
By the definition of $\underline{\mathbf{v}}''$ and $\underline{\mathbf{v}}'''$, we have
$t_{b,c}^{(1)}(\underline{\mathbf{v}}')= (-1)^{|b|} (\underline{\mathbf{v}}''+  \underline{\mathbf{v}}''')$.
Then, by the induction hypothesis, we have
	\begin{align}
&\quad t_{a,b}^{(2)} (t_{b,c}^{(1)}(\mathbf{m}\otimes\underline{\mathbf{v}}'))=t_{a,b}^{(2)} ( \mathbf{m}\otimes( \underline{\mathbf{v}}''+\underline{\mathbf{v}}'''))\nonumber\\
&=(-1)^{|a|+|b|} \left( 
	\sum_{t\ge 1, t\ne 2}  \chi^{a,b}_{j'_t} (\mathbf{m}) \otimes E_{a,b}^{(j_t')} (\underline{\mathbf{v}}'') 
	+ 
	 \sum_{t\ge 2}  \chi^{a,b}_{j'_t} (\mathbf{m}) \otimes E_{a,b}^{(j'_t)}(\underline{\mathbf{v}}''')\right).\label{gazz90}
\end{align}
By the definition of $\underline{\mathbf{v}}''$ and $\underline{\mathbf{v}}'''$ and $|b|=|c|$, we obtain
\begin{align}
t_{b,c}^{(1)}&\left( 
	 \sum_{t\ge 1, t\ne 2}  \chi^{a,b}_{j'_t} (\mathbf{m})\otimes E_{a,b}^{(j_t')} (\underline{\mathbf{v}}'')  \right) 
	  = (-1)^{|b|}
	\sum_{t\ge 1, t\ne 2} \chi^{a,b}_{j'_t} (\mathbf{m})\otimes E_{a,b}^{(j_t')} (\underline{\mathbf{v}})\label{align89}
\end{align}
and
\begin{align}
t_{b,c}^{(1)} &\left( 
	 \sum_{t\ge 2}  \chi^{a,b}_{j'_t} (\mathbf{m})\otimes E_{a,b}^{(j'_t)}(\underline{\mathbf{v}}''') \right)
	=(-1)^{|b|} \sum_{t\ge 2} \chi^{a,b}_{j'_t} (\mathbf{m})  \otimes E_{a,b}^{(j'_t)}(\underline{\mathbf{v}}).\label{align88}
\end{align}
By \eqref{gazz90}-\eqref{align88}, we find that
\begin{equation}
 t_{b,c}^{(1)} t_{a,b}^{(2)} t_{b,c}^{(1)}(\mathbf{m}\otimes\underline{\mathbf{v}}) =(-1)^{|a|}\sum_{t\ge 1, t\ne 2} \chi^{a,b}_{j'_t} (\mathbf{m})\otimes E_{a,b}^{(j_t')} (\underline{\mathbf{v}})+ (-1)^{|a|}\sum_{t\ge 2} \chi^{a,b}_{j'_t} (\mathbf{m})  \otimes E_{a,b}^{(j'_t)}(\underline{\mathbf{v}}).\label{align87}
\end{equation}	
By applying \eqref{gazz96}-\eqref{align87} to \eqref{gazz95}, we obtain \begin{equation*}
t_{a,b}^{(2)}(\mathbf{m} \otimes\underline{\mathbf{v}}) = (-1)^{|a|}\sum_{t \ge 1}  \chi^{a,b}_{j'_t} (\mathbf{m}) \otimes E_{a,b}^{(j_t')} (\underline{\mathbf{v}}).
\end{equation*}

\begin{comment}
\begin{align*} 
&\quad t_{a,b}^{(2)}(\mathbf{m} \otimes\underline{\mathbf{v}})\\
&= (-1)^{|a|}\sum_{t\ge 1, t\ne 2}  \chi^{a,b}_{j'_t} (\mathbf{m})  \otimes E_{a,b}^{(j_t')} (\underline{\mathbf{v}}) +(-1)^{|a|} 
	\sum_{t\ge 2} \chi^{a,b}_{j'_t} (\mathbf{m}) \otimes E_{a,b}^{(j'_t)}(\underline{\mathbf{v}})  - (-1)^{|a|}\sum_{t \ge 3} \chi^{a,b}_{j'_t} (\mathbf{m}) \otimes E_{a,b}^{(j_t')} (\underline{\mathbf{v}})\\
&=(-1)^{|a|}\sum_{t \ge 1}  \chi^{a,b}_{j'_t} (\mathbf{m}) \otimes E_{a,b}^{(j_t')} (\underline{\mathbf{v}}) \end{align*}
\end{comment}
By the induction, we find that (\ref{idbits}) holds $\forall \, r\ge 1$ and $\forall \, s\ge 1$. \end{proof}
\begin{proposition}\label{Prop256}
Assume that $m,n\geq l+1$. The map $\chi^{a,b}_{k}$ does not depend on the choice of $a$ and $b$.
\end{proposition}

\begin{proof}
 Let us take $e\neq a,b$. It is enough to show that $\chi_k^{a,b}=\chi_k^{a,e}$ and $\chi_k^{a,b}=\chi_k^{e,b}$. We only show the second equation: the first equation can be proven in a similar way.
By the assumption that $b\neq e$, we have $E_{e,b}=[E_{e,a},E_{a,b}]$ and $t_{e,b}^{(2)}=(-1)^{|e|}[t_{e,a}^{(1)},t_{a,b}^{(2)}]$. We obtain
\begin{align*}
(-1)^{|e|}\sum_{k=1}^l \chi_k^{e,b} (\mathbf{m}) \otimes E_{e,b}^{(k)} (\underline{\mathbf{v}}) &= t_{e,b}^{(2)}(\mathbf{m}\otimes \underline{\mathbf{v}})\\
&= (-1)^{|e|}[t_{e,a}^{(1)}, t_{a,b}^{(2)}](\mathbf{m}\otimes \underline{\mathbf{v}}) \\
&= (-1)^{|e|}\sum_{k=1}^l \chi_k^{a,b}(\mathbf{m})\otimes [E^{(k)}_{e,a} , E_{a,b}^{(k)} ](\underline{\mathbf{v}})\\ 
&= (-1)^{|e|}\sum_{k=1}^l \chi_{k}^{a,b}(\mathbf{m})\otimes E_{e,b}^{(k)} (\underline{\mathbf{v}}),
\end{align*}
where the first and third equalities are due to \eqref{idbits}. Thus, we find that $\chi_k^{a,b}=\chi_k^{e,b}$.
\end{proof}
Due to the previous proposition, we can define $\chi_k$ as being equal to $\chi_k^{a,b}$ for any choice of $a$ and $b$. We can now prove Theorem~\ref{floridakilos-1}.
\begin{proof}[Proof of Theorem~\ref{floridakilos-1}]
Since $M$ is an $\C[S_l]$-module, it is enough to define an action of $\mathsf{z}_k$ on $M$ which is compatible with \eqref{gather159} and \eqref{gather160} by Proposition~\ref{prop201}. We set $\mathbf{m} \mathsf{z}_k = \chi_k (\mathbf{m})$ for $\mathbf{m}\in M$. By the definition of $\chi_k$, we have
\begin{equation}
t_{a,b}^{(2)}(\mathbf{m}\otimes\underline{\mathbf{v}})= (-1)^{|a|}\sum_{k=1}^{l} \mathbf{m}\mathsf{z}_k \otimes E_{a,b}^{(k)} (\underline{\mathbf{v}})\label{gazz1234}
\end{equation}
for any $a\neq b$, $\mathbf{m}\in M$, and $\underline{\mathbf{v}}\in\mathbb{C}(m|n)^{\otimes l}$.

First, we prove the compatibility with \eqref{gather160}. Let us take $d\neq a\neq b\neq c$. By the assumption that $m+n\leq 2l+2>l+2$, we can choose $\underline{\mathbf{v}}$ to be $\underline{\mathbf{v}} =\bigotimes_{1\leq p\leq l}e_{i_p}$ satisfying that the $i_p$ are distinct, $i_{k_1}=b$, $i_{k_2}=d$ and $i_p\neq b,d$ if $p\neq k_1,k_2$.
By Lemma~\ref{chp}, it is enough to show that
\begin{equation}
\mathbf{m}[\mathsf{z}_{k_1},\mathsf{z}_{k_2}]\otimes E_{a,b}^{(k_2)} E_{c,d}^{(k_1)} (\underline{\mathbf{v}})=\mathbf{m}\left(\mathsf{z}_{k_1}-  \mathsf{z}_{k_2}  \right)\sigma_{k_1,k_2} \otimes  E_{a,b}^{(k_2)} E_{c,d}^{(k_1)}(\underline{\mathbf{v}})\label{gazz551}
\end{equation}
for any $\mathbf{m}\in M$ and for our specific choice of $\underline{\mathbf{v}}$. We can rewrite the left hand side of \eqref{gazz551} as follows:
\begin{align}
\mathbf{m}[\mathsf{z}_{k_1},\mathsf{z}_{k_2}]\otimes & E_{a,b}^{(k_2)} E_{c,d}^{(k_1)} (\underline{\mathbf{v}})\nonumber\\
&=\mathbf{m}\mathsf{z}_{k_1}\mathsf{z}_{k_2}\otimes E_{a,b}^{(k_2)} E_{c,d}^{(k_1)} (\underline{\mathbf{v}}) - \mathbf{m}\mathsf{z}_{k_2} \mathsf{z}_{k_1} \otimes E_{a,b}^{(k_2)}  E_{c,d}^{(k_1)} (\underline{\mathbf{v}})\nonumber\\
&=\mathbf{m}\mathsf{z}_{k_1}\mathsf{z}_{k_2}\otimes E_{a,b}^{(k_2)} E_{c,d}^{(k_1)} (\underline{\mathbf{v}})-(-1)^{(|a|+|b|)(|c|+|d|)}\mathbf{m}\mathsf{z}_{k_2} \mathsf{z}_{k_1} \otimes E_{c,d}^{(k_1)} E_{a,b}^{(k_2)} (\underline{\mathbf{v}})\nonumber\\
&=(-1)^{|a|}t_{a,b}^{(2)} \left( \sum_{j=1}^{l}\mathbf{m}y_j \otimes E_{c,d}^{(j)} (\underline{\mathbf{v}})\right) 
-(-1)^{|c|+(|a|+|b|)(|c|+|d|)}	t_{c,d}^{(2)} \left( \sum_{i=1}^{l}\mathbf{m}\mathsf{z}_i \otimes E_{a,b}^{(i)} (\underline{\mathbf{v}}) \right) \nonumber\\
&=(-1)^{|a|+|c|}[t_{a,b}^{(2)},t_{c,d}^{(2)}](\mathbf{m}\otimes \underline{\mathbf{v}}),\label{11r}
\end{align}
where the second equality is due to the assumption that $d\neq a\neq b\neq c$ and the third equality is due to the definition of $\underline{\mathbf{v}}$ and the action of $\mathsf{z}_i$. 
By a direct computation, we obtain
\begin{align}
(-1)^{|a|+|c|}[t_{a,b}^{(2)},t_{c,d}^{(2)}]& (\mathbf{m}\otimes \underline{\mathbf{v}}) \nonumber \\
& =(-1)^{|a||b|+|a||c|+|b||c|+|a|+|c|} \left( t_{c,b}^{(1)} t_{a,d}^{(2)} - t_{c,b}^{(2)} t_{a,d}^{(1)} \right)(\mathbf{m}\otimes \underline{\mathbf{v}}) \text{ by \eqref{RTT2};}\nonumber\\
&=(-1)^{|a||b|+|a||c|+|b||c|+|c|}\left( t_{c,b}^{(1)} \sum_{j=1}^{l}\mathbf{m}\mathsf{z}_j \otimes E_{a,d}^{(j)} (\underline{\mathbf{v}}) - t_{c,b}^{(2)} \sum_{j=1}^l \mathbf{m} \otimes E_{a,d}^{(j)}(\underline{\mathbf{v}}) \right) \text{ by \eqref{gazz1234};} \nonumber \\ 
&=(-1)^{|a||b|+|a||c|+|b||c|}\left(\sum_{i,j=1}^{l} \mathbf{m}\mathsf{z}_j \otimes E_{c,b}^{(i)} E_{a,d}^{(j)} (\underline{\mathbf{v}})-\sum_{i,j=1}^l \mathbf{m}\mathsf{z}_j \otimes E_{c,b}^{(i)} E_{a,d}^{(j)}(\underline{\mathbf{v}})\right) \text{ by \eqref{gazz1234};}\nonumber \\
&=(-1)^{|a||b|+|a||c|+|b||c|}\left( \mathbf{m}(\mathsf{z}_{k_1} - \mathsf{z}_{k_2})\otimes E_{c,b}^{(k_2)} E_{a,d}^{(k_1)}(\underline{\mathbf{v}})
			\right),\label{11r3}
\end{align}
where the last equality is due to the choice of $\underline{\mathbf{v}}$.
By a direct computation, we have
	\begin{equation}
	\sigma_{k_1,k_2} E_{a,b}^{(k_2)} E_{c,d}^{(k_1)}(\underline{\mathbf{v}}) = (-1)^{|a|} E_{c,a}^{(k_2)} E_{a,c}^{(k_1)} E_{a,b}^{(k_2)}E_{c,d}^{(k_1)} (\underline{\mathbf{v}}) =(-1)^{|a||b|+|b||c|+|a||c|} E_{c,b}^{(k_2)} E_{a,d}^{(k_1)} (\underline{\mathbf{v}}).\label{11r4}
	\end{equation} 
Applying \eqref{11r}-\eqref{11r4} to \eqref{gazz551}, we have
\begin{align}
\mathbf{m}[\mathsf{z}_{k_1},\mathsf{z}_{k_2}]\otimes E_{a,b}^{(k_2)} E_{c,d}^{(k_1)} (\underline{\mathbf{v}})&=(-1)^{|a|+|c|}[t_{a,b}^{(2)},t_{c,d}^{(2)}](\mathbf{m}\otimes \underline{\mathbf{v}}) \nonumber\\
&=\mathbf{m}(\mathsf{z}_{k_1} -\mathsf{z}_{k_2})\otimes\sigma_{k_1,k_2} E_{a,b}^{(k_2)} E_{c,d}^{(k_1)}(\underline{\mathbf{v}})\nonumber\\
&= \mathbf{m}(\mathsf{z}_{k_1} -\mathsf{z}_{k_2})\sigma_{k_1,k_2}\otimes E_{a,b}^{(k_2)} E_{c,d}^{(k_1)}(\underline{\mathbf{v}}).\label{11r5}
\end{align}
By \eqref{11r} and \eqref{11r5}, we have obtained \eqref{gazz551}.

Next, we check the compatibility of the action of $\mathsf{z}_i$ on $M$ with \eqref{gather159}. We take $\underline{\mathbf{v}}=\bigotimes_{j=1}^l v_j$ satisfying that the $v_j$ are all pairwise distinct, $v_j\neq a$ and $v_i=e_b$ for one fixed index $i$. 
By Lemma~\ref{chp}, it is enough to show the following relation:
\begin{equation}
\mathbf{m}\sigma \mathsf{z}_i \otimes E_{a,b}\underline{\mathbf{v}}=\mathbf{m}\mathsf{z}_{\sigma(i)}  \sigma\otimes E_{a,b}\underline{\mathbf{v}}.\label{gazz5110}
\end{equation}
By \eqref{gazz1234}, we can rewrite the left hand side of \eqref{gazz5110} as:
\begin{equation} 
		\mathbf{m}\sigma \mathsf{z}_i \otimes E_{a,b}\underline{\mathbf{v}} = (-1)^{|a|} t_{a,b}^{(2)} (\mathbf{m}\sigma \otimes \underline{\mathbf{v}}). \label{tt1}
\end{equation}
We can also rewrite the right hand side of \eqref{gazz5110} as follows:
\begin{equation}
\mathbf{m}\mathsf{z}_{\sigma(i)}  \sigma\otimes E_{a,b}\underline{\mathbf{v}}=\mathbf{m}\mathsf{z}_{\sigma(i)} \otimes \sigma E_{a,b}\underline{\mathbf{v}}=\mathbf{m}\mathsf{z}_{\sigma(i)} \otimes E_{a,b}\sigma \underline{\mathbf{v}}=(-1)^{|a|}t_{a,b}^{(2)} (\mathbf{m} \otimes \sigma\underline{\mathbf{v}})=(-1)^{|a|}t_{a,b}^{(2)} (\mathbf{m}\sigma \otimes \underline{\mathbf{v}}),\label{tt2}
\end{equation}
where the second equality is due to Lemma~\ref{lem38} and the third equality is due to the definition of the action of $\mathsf{z}_i$.
By (\ref{tt1}) and (\ref{tt2}), we have obtained \eqref{gazz5110}. 
\end{proof}
\begin{proposition}\label{thm-sup}
The functor
 	 $$SW: \Hom_{H^{\deg}_{\kappa}(S_l)} (M_1,M_2) \rightarrow \Hom_{Y(\mfgl_{m|n})} (SW(M_1),SW(M_2)),\ f\mapsto f\otimes 1^{\otimes l}$$
is bijective.
\end{proposition}
It is a consequence of this proposition and of Theorem \ref{floridakilos-1} that the function $SW$ provides an equivalence between the category of finite-dimensional right modules over $H^{\deg}_{\kappa}(S_l)$ and the category of finite-dimensional left modules over $Y(\mfgl_{m|n})$ of level $l$ when $m,n\ge l+1$.

\begin{proof}
The injectivity follows from that if $SW(f)=0$, then $f\otimes 1^{\otimes l}=0$ which is equivalent to $f=0$. For the surjectivity, we take $\varphi \in\Hom_{Y(\mfgl_{m|n})} (SW(M_1),SW(M_2))$.  Since the super Yangian contains the universal enveloping algebra of $\mfgl_{m|n}$, $\varphi$ is a homomorphism of $\mfgl_{m|n}$-modules. 
By Theorem~\ref{chengbook} (Schur-Sergeev duality), we can rewrite $\varphi=f \otimes 1^{\otimes l}$, where $f$ is a homomorphism of right $S_l$-modules.
Thus, it is enough to show the relation
\begin{equation*}
f(\mathbf{m} \mathsf{z}_i)=f(\mathbf{m})\mathsf{z}_i
\end{equation*}
for $1\leq i\leq l$ and $\mathbf{m}\in M$. By the assumption that $l<2l+2\leq m+n$, we can take $\mathbf{v}=\bigotimes_{1\leq u\leq l}e_{j_u}$ satisfying that $j_i=m+n$ and $j_u\neq m+n$ if $u\neq i$. By the definition of the action of $y_i$, we obtain
\begin{align}
		 f(\mathbf{m} \mathsf{z}_i )\otimes e_1 \otimes \cdots \otimes e_{i-1} \otimes e_{m+n} \otimes e_{i+1} & \otimes\cdots\otimes e_l \nonumber\\
		  &=\varphi( m \mathsf{z}_i \otimes e_1 \otimes \cdots \otimes e_{i-1} \otimes e_{m+n} \otimes e_{i+1} \otimes\cdots\otimes e_l )\nonumber\\ 
		 &=\varphi ( t_{m+n,i}^{(2)} (m\otimes e_1 \otimes \cdots \otimes e_{i-1}  \otimes e_i \otimes e_{i+1} \otimes \cdots\otimes e_l )).\label{9788}
\end{align}
Since $\varphi\in\Hom_{Y(\mfgl_{m|n})} (SW(M_1),SW(M_2))$, $t_{m+n,i}^{(2)}$ and $\varphi$ commute with each other. Thus, we obtain
\begin{align}
\varphi ( t_{m+n,i}^{(2)} (m\otimes e_1 \otimes \cdots \otimes e_{i-1}  \otimes e_i \otimes e_{i+1} & \otimes \cdots\otimes e_l ))\nonumber\\
		 &=t_{m+n,i}^{(2)} ( \varphi ( m\otimes e_1 \otimes \cdots \otimes e_{i-1}  \otimes e_i \otimes e_{i+1}\otimes \cdots \otimes e_l )	 ) \nonumber\\
		   &=t_{m+n,i}^{(2)} ( f (m) \otimes e_1  \otimes \cdots \otimes e_l )	 \nonumber\\
		   &= f(m) \mathsf{z}_i \otimes e_1 \otimes \cdots \otimes e_{i-1} \otimes e_{m+n} \otimes e_{i+1} \otimes\cdots\otimes e_l,\label{9799}
		 \end{align} 
where the last equality is due to the definition of the action of $y_i$. By \eqref{9788} and \eqref{9799}, we obtain $f(\mathbf{m} \mathsf{z}_i)\otimes \underline{\mathbf{v}}=f(m) \mathsf{z}_i\otimes \underline{\mathbf{v}}$. This completes the proof. 
\end{proof}

\section{Schur-Weyl duality for the Degenerate Affine Hecke algebra and the Super Yangian of $\mathfrak{sl}_{m|n}$}\label{SWmini}

In \cite{driJ}, Drinfeld established Schur-Weyl duality for the degenerate affine Hecke algebra of $S_l$ and the Yangian associated with $\mathfrak{sl}_{n}$ using the original so-called $J$-presentation of $Y_{\lambda}(\mathfrak{sl}_{n})$. Since $Y_{\lambda}(\mathfrak{sl}_{n})$ can be viewed as a subalgebra of $Y_{\lambda}(\mathfrak{gl}_{n})$ in the RTT-presentation and we have the tensor decomposition $$Y_{\lambda}(\mathfrak{gl}_{n}) \cong Y_{\lambda}(\mathfrak{sl}_{n})\otimes Z(Y_{\lambda}(\mathfrak{gl}_{n}))$$ where $Z(Y_{\lambda}(\mathfrak{gl}_{n}))$ is the center of the Yangian $Y_{\lambda}(\mathfrak{gl}_{n})$, which is generated by the coefficients of the quantum determinant, the results in \cite{ara} are applicable to $Y_{\lambda}(\mathfrak{sl}_{n})$. The same can be said about the super Yangians $Y(\mathfrak{gl}_{m|n})$ and $Y_{\lambda}(\mathfrak{sl}_{m|n})$, with the quantum determinant replaced by the Berezinian when $m\neq n$: see \cite{Gow}. It should then be possible to translate them to the (minimalistic) current presentation of $Y(\mathfrak{sl}_{m|n})$ and determine the action of the generators of the latter on $SW(M)$, however we prefer to do this directly without relying on the results of the previous section, in part because we will need to consider in the next section parity sequences other than the standard one and extending results to affine super Yangians will require us to use the current presentation. We also take this opportunity to introduce generators $J(x)$ for $x\in\mathfrak{sl}_{m|n}$ analogous to the original generators of $Y_{\lambda}(\mathfrak{sl}_{n})$ in \cite{driJ}.

\subsection{The super Yangian associated with $\mathfrak{sl}_{m|n}$}
Let us recall the definition of the super Yangian associated with $\mathfrak{sl}_{m|n}$ (see Definition 2.1 in \cite{St1}, Proposition 5 in \cite{Gow2}). In this section and the next, we will always assume that $m, n\geq2$ and $m\neq n$.

\begin{definition}[]\label{deffinsup}
Let $\lambda\in\mathbb{C}$.  The super Yangian $Y_{\lambda}(\mathfrak{sl}(m|n))$ is the associative superalgebra over $\C$ generated by $X_{i,r}^{+}, X_{i,r}^{-}, H_{i,r}$ $(i \in \{1,\cdots, m+n-1\}, r = 0,1)$ subject to the following defining relations:
\begin{gather}
[H_{i,r}, H_{j,s}] = 0,\label{Eq2.1-1}\\
[X_{i,0}^{+}, X_{j,0}^{-}] = \delta_{i,j} H_{i, 0},\label{Eq2.2-1}\\
[X_{i,1}^{+}, X_{j,0}^{-}] = \delta_{i,j} H_{i, 1} = [X_{i,0}^{+}, X_{j,1}^{-}],\label{Eq2.3-1}\\
[H_{i,0}, X_{j,r}^{\pm}] = \pm a_{i,j} X_{j,r}^{\pm},\label{Eq2.4-1}\\
[\tilde{H}_{i,1}, X_{j,0}^{\pm}] = \pm a_{i,j} X_{j,1}^{\pm},\label{Eq2.5-1-1}\\
[X_{i, 1}^{\pm}, X_{j, 0}^{\pm}] - [X_{i, 0}^{\pm}, X_{j, 1}^{\pm}] = \pm \dfrac{\lambda a_{ij}}{2} \{X_{i, 0}^{\pm}, X_{j, 0}^{\pm}\},\label{Eq2.8-1}\\
(\ad X_{i,0}^{\pm})^{1+|a_{i,j}|} (X_{j,0}^{\pm})= 0 \ \ (i \neq j), \label{Eq2.10-1}\\ 
[X^\pm_{m,0},X^\pm_{m,0}]=0,\label{Eq2.11-1}\\
[[X^\pm_{m-1,0},X^\pm_{m,0}],[X^\pm_{m,0},X^\pm_{m+1,0}]]=0,\label{Eq2.12-1}
\end{gather}
where $\{X_{i, 0}^{\pm}, X_{j, 0}^{\pm}\} = X_{i, 0}^{\pm} X_{j, 0}^{\pm} + X_{j, 0}^{\pm} X_{i, 0}^{\pm}$,  the generators $X^\pm_{m, r}$ are odd, all other generators are even and we have set $\wt{H}_{i,1} = H_{i,1}-\dfrac{\lambda}2 H_{i,0}^2$. 
\end{definition}
\begin{Remark}
In \cite{Gow2}, the generators of $Y_{\lambda}(\mathfrak{sl}(m|n))$ are $\{H_{i,r},X^\pm_{i,r}\mid 1\leq i,j\leq m+n-1,r\in\mathbb{Z}_{\geq0}\}$. By the same way as Theorem~3.13 in \cite{Ue}, we can obtain the presentation given in Definition~\ref{deffinsup}.
\end{Remark}
 We note that there exists a homomorphism from $U(\mathfrak{sl}(m|n))$ to $Y_{\lambda}(\mathfrak{sl}(m|n))$ given by $x^\pm_i\mapsto X^\pm_{i,0},\ h_i\mapsto H_{i,0}$, where $x_i^+ = E_{i,i+1}$, $x_i^- = (-1)^{|i|} E_{i+1,i}$ and $h_i = (-1)^{|i|}E_{i,i} - (-1)^{|i+1|} E_{i+1,i+1}$. We denote the image of $x\in U(\mathfrak{sl}(m|n))$ also by $x$.

For $1\leq i\leq m+n$, let $i^{(0)}$ be $\sum_{j=i+2}^{m+n}\limits (-1)^{|j|}$.
For $1\leq i\leq m+n-1$, we define new generators of $Y_{\lambda}(\mathfrak{sl}_{m|n})$ as follows:
\begin{align*}
J(h_i)&=\widetilde{H}_{i,1}+ \lambda \vartheta_i,\ 
J(x^\pm_i)=X^\pm_{i,1}+ \lambda \varpi^\pm_i,
\end{align*}
where
\begin{align*}
\vartheta_i&=-\dfrac{(i-1)^{(0)}}{2}H_{i,0}+\dfrac{(-1)^{|i|}}{2}\sum_{k\neq i}\limits (-1)^{|k|}\text{sign}(k-i)E_{i,k}E_{k,i}\\
&\quad-(-1)^{|i+1|}\dfrac{\lambda}{2}\sum_{k\neq i+1}\limits (-1)^{|k|}\text{sign}(k-i-1)E_{i+1,k}E_{k,i+1},\\
\varpi^+_i&=-\dfrac{i^{(0)}}{2} X^+_{i,0}+\dfrac{1}{2}\sum_{k\neq i,i+1}\limits (-1)^{|k|}\text{sign}(k-i)E_{i,k}E_{k,i+1}-\dfrac{1}{2}H_{i,0}X^+_{i,0},\\
\varpi^-_i&=-\dfrac{i^{(0)}}{2} X^-_{i,0}+\dfrac{(-1)^{|i|}}{2}\sum_{k\neq i,i+1}\limits (-1)^{|k|}\text{sign}(k-i)E_{i+1,k}E_{k,i}-\dfrac{1}{2}X^-_{i,0}H_{i,0}.
\end{align*}
\begin{Remark}
Since we obtain
\begin{align*}
\vartheta_i&=\dfrac{1}{2}\sum_{l<k}\limits(-1)^{|l|}E_{k,l}[(-1)^{|i|}E_{i,i}-(-1)^{|i+1|}E_{i+1,i+1},E_{l,k}],\\
\varpi^+_i&=-\dfrac{1}{2}\sum_{l<k}\limits(-1)^{|l|}[E_{i,i+1},E_{k,l}]E_{l,k},\\
\varpi^-_i&=-(-1)^{|i|}\dfrac{1}{2}\sum_{l<k}\limits(-1)^{|l|}E_{k,l}[E_{l,k},E_{i+1,i}]
\end{align*}
 by a direct computation, we find that the generators $J(h_i)$ and $J(x_i^{\pm})$ are natural extensions those in Section~3 of \cite{GNW} to the super setting.
\end{Remark}
\begin{lemma}[Section 3 in \cite{GNW} and Section 3 in \cite{Ue}]
Relations in the super Yangian $Y_{\lambda}(\mathfrak{sl}(m|n))$ can be rewritten in terms of the generators $J(h_i)$ and $J(x_i^{\pm})$ as follows.
\begin{gather}
[H_{i,0},H_{j,1}]=0\iff[h_{i},J(h_j)]=0,\label{DriJ-1}\\
[H_{i,0},X^\pm_{j,1}]=\pm a_{i,j}x^\pm_{j,1}\iff[h_{i},J(x^\pm_j)]=\pm a_{i,j}J(x^\pm_j),\label{DriJ-2}\\
[\widetilde{H}_{i,1},X^\pm_{j,0}]=a_{i,j}X^\pm_{j,1}\iff[J(h_i),x^\pm_{j}]=\pm a_{i,j}J(x^+_j),\label{DriJ-2.5}\\
[X_{i,1}^{\pm},X_{j,0}^{\pm}]-[X_{i,0}^{\pm},X_{j,1}^{\pm}] = \pm a_{i,j} \frac{\lambda}{2} \{ X_{i,0}^{\pm},X_{j,0}^{\pm}\}\iff[x^\pm_{i},J(x^\pm_j)]=[J(x^\pm_i),x^\pm_{j}],\label{DriJ-3}\\
[X^+_{i,1},X^-_{j,0}]=\delta_{i,j}H_{i,1}\iff[J(x^+_i),x^-_j]=\delta_{i,j}J(h_i).\label{DriJ-4}
\end{gather}

\end{lemma}For the proof of Theorem~\ref{flor} below, we construct the elements $J(E_{a,b})\in Y_\lambda(\mathfrak{sl}(m|n))$ for $a\neq b$. By \eqref{DriJ-2.5} and \eqref{DriJ-3}, we obtain
\begin{equation}
[J(x^\pm_i),X^\pm_{j,0}]=0\text{ if }i\neq j\pm1.\label{DriJ-6}
\end{equation}
Since $E_{a,b}$ can be written as $\prod_{c=i}^{j-1}\ad(X^\pm_{c,0})X^\pm_{j,0}$ for some $i,j$, the following lemma follows from \eqref{DriJ-6}.
\begin{lemma}
Suppose that $x_\alpha\in\mathfrak{sl}(m|n)$ is a root vector associated with the root $\alpha$. Then, we obtain
\begin{equation}
    (\alpha_j,\alpha)[J(h_i),x_\alpha]-(\alpha_i,\alpha)[J(h_j),x_\alpha]=0.\label{DriJ-7}
\end{equation}
\end{lemma}
Then, we can define $J(x_\alpha)=\dfrac{1}{(\alpha_i,\alpha)}[J(h_i),x_\alpha]$.
\begin{lemma}
The following relation holds;
\begin{equation}
    [J(x_\alpha),x_\beta]=J([x_\alpha,x_\beta]).\label{DriJ-8}
\end{equation}
\end{lemma}
\begin{proof}
In the case that $\alpha+\beta\neq0$, we choose $i,j$ satisfying that $\begin{vmatrix}
(\alpha_i,\alpha) & (\alpha_i,\beta) \\
(\alpha_j,\alpha) & (\alpha_j,\beta) \\
\end{vmatrix}\neq0
$. By computing $[J(h_i),[x_\alpha,x_\beta]]$ and $[J(h_j),[x_\alpha,x_\beta]]$, we have
\begin{align*}
(\alpha_i,\alpha)[J(x_\alpha),x_\beta]+(\alpha_i,\beta)[x_\alpha,J(x_\beta)]=(\alpha_i,\alpha+\beta)J([x_\alpha,x_\beta]),\\
(\alpha_j,\alpha)[J(x_\alpha),x_\beta]+(\alpha_j,\beta)[x_\alpha,J(x_\beta)]=(\alpha_j,\alpha+\beta)J([x_\alpha,x_\beta])
\end{align*}
by \eqref{DriJ-7}. Then, we obtain \eqref{DriJ-8}.

Let us consider the case that $\alpha+\beta=0$. We assume that $\alpha$ is a positive root. We prove by the induction on $k_\alpha$ such that $\alpha=\alpha_i+\alpha_{i+1}+\cdots+\alpha_{i+k_\alpha}$. In the case that $k_\alpha=0$, \eqref{DriJ-8} is nothing but \eqref{DriJ-4}. Suppose that \eqref{DriJ-8} holds for $k_\alpha=M$. When $k_\alpha=M+1$, $\alpha$ can be written as $\gamma+\alpha_j$ for some positive root $\gamma$ such that $k_\gamma=M$. Then, it is enough to show that
\begin{align}
[J(x_\alpha),[x_{-\gamma},X^-_{j,0}]]=J([x_\alpha,[x_{-\gamma},X^-_{j,0}]]). \label{DriJ-9}
\end{align}
Since we have already proven that \eqref{DriJ-8} holds when $\alpha+\beta\neq0$, we can prove \eqref{DriJ-9} by the induction hypothesis. This completes the proof.
\end{proof}
\subsection{Schur-Weyl duality for the Degenerate Affine Hecke algebra and the Super Yangian of $\mathfrak{sl}_{m|n}$}
Let $\mathfrak{z}_i\in H^{\text{deg}}_\lambda(S_l)$ be 
\begin{gather*}
\mathfrak{z}_i=\mathsf{u}_i+\dfrac{\lambda}{2}\sum_{j\neq i}\limits\text{sign}(j-i)\sigma_{i,j}.
\end{gather*}
It follows from this definition and \eqref{ldrna} - \eqref{ldrna3} that
\begin{gather}
    \sigma \mathfrak{z}_i\sigma^{-1}=\mathfrak{z}_{\sigma(i)} \; \forall \, \sigma\in S_l ,\label{prop96-0}\\
 \mathsf{z}_i=\mathfrak{z}_i-\dfrac{\lambda}{2}\sum_{j\neq i}\sigma_{j,i}.\label{prop96-0.5}   
\end{gather} 

\begin{theorem}\label{finite case-1}
Suppose that $M$ is a right $H_{\lambda}^{\text{deg}}(S_{\ell})$-module. 
We can define an action of $Y_{\lambda}(\mathfrak{sl}(m|n))$ on $M\otimes_{\C[S_l]}\mathbb{C}(m|n)^{\otimes l}$ by
\begin{align}
H_{i,0} (\mathbf{m}\ot\mathbf{v})&=\mathbf{m}\ot ((-1)^{|i|}E_{i,i}-(-1)^{|i+1|}E_{i+1,i+1})\mathbf{v}, \label{defacth0-1} \\ 
X_{i,0}^{+} (\mathbf{m}\ot\mathbf{v})& =\mathbf{m}\ot E_{i,i+1}\mathbf{v}, \;\; 
X_{i,0}^{-} (\mathbf{m}\ot\mathbf{v})=(-1)^{|i|}\mathbf{m}\ot E_{i+1,i}\mathbf{v},\label{defactx0-1}\\
J(h_i)(\mathbf{m}\ot\mathbf{v})&=\sum_{k=1}^l\limits\mathbf{m}\mathfrak{z}_{k}\otimes ((-1)^{|i|}E^{(k)}_{i,i}-(-1)^{|i+1|}E^{(k)}_{i+1,i+1})\mathbf{v},\label{defacth1-1}\\
J(x^+_i)(\mathbf{m}\ot\mathbf{v})&=\sum_{k=1}^l\limits\mathbf{m}\mathfrak{z}_{k}\otimes E^{(k)}_{i,i+1}\mathbf{v},\;\; 
J(x^-_i)(\mathbf{m}\ot\mathbf{v})=(-1)^{|i|}\sum_{k=1}^l\limits\mathbf{m}\mathfrak{z}_{k}\otimes E^{(k)}_{i+1,i}\mathbf{v}\label{defactx1-1}
\end{align}
for $\mathbf{m}\in M$ and $\mathbf{v}\in\C(m|n)^{\otimes l}$.
\end{theorem}
\begin{proof}  We must first verify that those operator are well-defined and for this it is enough to see why \begin{equation}
J(A)(\mathbf{m}\otimes\sigma\mathbf{v})=J(A)(\mathbf{m}\sigma\otimes\mathbf{v})\label{well-defined}
\end{equation}
for $A=h_i,x_i^\pm$. Since the action is compatible with \eqref{DriJ-2.5}, it is enough to show the case when $A=h_i$, that is, that the following equality holds: \begin{equation}
\sum_{k=1}^l\limits \mathbf{m}\mathfrak{z}_k\otimes((-1)^{|i|}E^{(k)}_{i,i}-(-1)^{|i+1|}E^{(k)}_{i+1,i+1})\sigma\mathbf{v}=\sum_{k=1}^l\limits \mathbf{m}\mathfrak{z}_k\sigma\otimes((-1)^{|i|}E^{(k)}_{i,i}-(-1)^{|i+1|}E^{(k)}_{i+1,i+1})\mathbf{v}.\label{well-definedness2}
\end{equation} This is a consequence of \eqref{prop96-0}.

The operators defined in \eqref{defacth0-1}-\eqref{defactx0-1} are compatible with the relations \eqref{Eq2.11-1} and \eqref{Eq2.12-1}. For all the operators defined in \eqref{defacth0-1}-\eqref{defactx1-1}, we find that they are compatible with the right hand sides of \eqref{DriJ-1}-\eqref{DriJ-4}. Thus, we need to check the well-definedness of the action and show the compatibility with $[\widetilde{H}_{i,1},\widetilde{H}_{j,1}]=0$. This will be a consequence of the following claim.

\begin{Claim}
    Suppose that $\mathbf{v}=\bigotimes_{1\leq r\leq l}v_{i_r}$ satisfying that $i_1\leq i_2\leq \cdots\leq i_l$. We denote by $j_k$ (resp. $\widetilde{j}_k$) by the first (resp. last) term of $j$ such that $i_j=k$. Then, we have
\begin{align}
\widetilde{H}_{i,1}(\mathbf{m}\ot\mathbf{v})&=(-1)^{|i|}\sum_{k=j_i}^{\widetilde{j}_i}\limits\mathbf{m}\mathsf{u}_k\otimes\mathbf{v}-(-1)^{|i+1|}\sum_{k=j_{i+1}}^{\widetilde{j}_{i+1}}\limits\mathbf{m}\mathsf{u}_k\otimes\mathbf{v}+\dfrac{\lambda}{2}(i-1)^{(0)}H_{i,0}(\msm\ot\mathbf{v}).\label{aaa}
\end{align}
\end{Claim}
\begin{proof}
By the definition of $J(h_i)$, $$
\widetilde{H}_{i,1}(\mathbf{m}\ot\mathbf{v})
=(J(h_i)-\vartheta_i)(\mathbf{m}\ot\mathbf{v}).$$
By the definition of $\mathfrak{z}_i$,
\begin{align}
J(h_i)(\mathbf{m}\ot\mathbf{v}) &=\sum_{k=1}^l\limits\mathbf{m}\left (\mathsf{u}_k+\dfrac{\lambda}{2}\sum_{p\neq k}\text{sign}(p-k)\sigma_{v,k}\right)\otimes ((-1)^{|i|}E^{(k)}_{i,i}-(-1)^{|i+1|}E_{i+1,i+1})\mathbf{v}.
\end{align}
By a direct computations, we obtain
\begin{align}
\sum_{k=1}^l\limits\mathbf{m} & \left(\sum_{p\neq k}\text{sign}(p-k)\sigma_{p,k}\right)\otimes (-1)^{|i|}E^{(k)}_{i,i}\mathbf{v}
\nonumber\\
&=\sum_{k=1}^l\sum_{\substack{p=1\\p\neq k}}^l\limits\sum_{g=1}^{m+n}\limits \text{sign}(p-k) (-1)^{|i|+|g|} \mathbf{m}\otimes E^{(p)}_{i,g}E^{(k)}_{g,i}\mathbf{v}\nonumber\\
&=\sum_{k,p=1}^l\sum_{\substack{g=1\\g\neq i}}^{m+n}\limits \text{sign}(g-i)(-1)^{|i|+|g|}\mathbf{m}\otimes E^{(p)}_{i,g}E^{(k)}_{g,i}\mathbf{v}+\sum_{\substack{j_i\leq k,p\leq \widetilde{j}_k\\p\neq k}}\limits\text{sign}(p-k)\mathbf{m}\otimes \mathbf{v}\nonumber\\
&=\sum_{k,p=1}^l\sum_{\substack{g=1\\g\neq i}}^{m+n}\limits \text{sign}(g-i)(-1)^{|i|+|g|}\mathbf{m}\otimes E^{(p)}_{i,g}E^{(k)}_{g,i}\mathbf{v}+0,\label{align913}
\end{align}
where the second equality is due to the assumption that $i_1\leq i_2\leq \cdots\leq i_l$.
Thus, we have
\begin{align}
J(h_i)(\mathbf{m}\ot\mathbf{v}) &=\sum_{k=1}^l\limits\mathbf{m}\mathsf{u}_k\otimes ((-1)^{|i|}E^{(k)}_{i,i}-(-1)^{|i+1|}E^{(k)}_{i+1,i+1})\mathbf{v}\nonumber\\
&\quad+\dfrac{\lambda}{2}\sum_{k,p=1}^l\sum_{g\neq i}\limits \text{sign}(g-i) (-1)^{|i|+|g|} \mathbf{m}\otimes E^{(p)}_{i,g}E^{(k)}_{g,i}\mathbf{v}\nonumber\\
&\quad-\dfrac{\lambda}{2}\sum_{k,p=1}^l\sum_{g\neq i+1}\limits \text{sign}(g-i) (-1)^{|i+1|+|g|} \mathbf{m}\otimes E^{(p)}_{i+1,g}E^{(k)}_{g,i+1}\mathbf{v}.\label{align912}
\end{align}
We find that the sum of the last four terms of the right hand side of \eqref{align912} is equal to $\la\vartheta_i(\msm\ot\mathbf{v})+\dfrac{\lambda}{2}(i-1)^{(0)}H_{i,0}(\msm\ot\mathbf{v})$.
This prove \eqref{aaa} and thus the previous claim.
\end{proof}
By \eqref{well-defined}, it is enough to show $[\widetilde{H}_{i,1},\widetilde{H}_{j,1}](\mathbf{m}\otimes\mathbf{v})=0$ in the case when $\mathbf{v}=\bigotimes_{r=1}^le_{i_r}$ satisfies that $i_1\leq i_2\leq\cdots\leq i_l$.
In this case, $[\widetilde{H}_{i,1},\widetilde{H}_{j,1}](\mathbf{m}\otimes\mathbf{v})=0$ follows from \eqref{aaa} and Proposition~\ref{Prop87} (2). \end{proof}

As before, let us set $SW(M) = M \otimes_{\C[S_l]} \C(m|n)^{\otimes l}$. Theorem~\ref{floridakilos-1} can be reformulated in the present context of this section. 
\begin{theorem}\label{flor}
Let $N$ be a $Y_\lambda(\mathfrak{sl}(m|n))$-module that is finite-dimensional of level $l$. Suppose that $m,n\geq l+1$. Then there exists a module $M$ over the degenerate affine Hecke algebra $H^{\deg}_\lambda(S_l)$ such that $N\cong SW(M)$.
\end{theorem}
\begin{proof}
By the Schur-Sergeev duality, there exists a $\mathbb{C}[S_l]$-module $M$ such that $N\simeq M\otimes_{\C[S_l]}\C(m|n)^{\otimes l}$ as modules over $\mathfrak{sl}(m|n)$. 
Changing $\mathsf{z}_k$, $t^{(1)}_{a,b}$ and $t^{(2)}_{a,b}$ to $\mathfrak{z}_k$, $(-1)^{|a|}E_{a,b}$ and $(-1)^{|a|}J(E_{a,b})$ in the proof of Proposition~\ref{Prop256} and Lemma~\ref{idbits-lem}, we find that we can define the action of $\mathfrak{z}_k$ on $M$ for which
\begin{equation*}
J(E_{a,b})(\mathbf{m}\otimes\mathbf{v})=\sum_{k=1}^l\limits \mathbf{m}\mathfrak{z}_k\otimes E^{(k)}_{a,b}\mathbf{v}
\end{equation*}
for $a\neq b$. 
As in the second half of the proof of Theorem~\ref{floridakilos-1}, we can show the compatibility of that action with \eqref{prop96-0}. By \eqref{prop96-0.5}, we can prove the compatibility with \eqref{gather159} since, for $i$ fixed, 
\begin{equation*}
\sigma\left(\sum_{j\neq i}\sigma_{i,j}\right)=\left(\sum_{j\neq \sigma(i)}\sigma_{\sigma(i),j}\right)\sigma \;\; \forall \, \sigma\in S_l.
\end{equation*}

By using the proof of Proposition~\ref{prop201}, we see that it is enough to show the compatibility of the action of $\mathfrak{z}_k$ on $M$ with $[\mathsf{u}_a,\mathsf{u}_b]=0$. By the assumption that $l<m,n$, we can take $\mathbf{v}=\bigotimes_{1\leq r\leq l}e_{k_r}$ satisfying that $k_r$ are distinct and $k_a=i$, $k_b=j$ and $k_r\neq i+1,j+1$.
By \eqref{aaa}, we have
\begin{align*}
[\widetilde{H}_{i,1},\widetilde{H}_{j,1}](\mathbf{m}\otimes\mathbf{v})
&=\mathbf{m}(\mathsf{u}_b\mathsf{u}_a-\mathsf{u}_a\mathsf{u}_b)\otimes\mathbf{v}.
\end{align*}
Then, the relation $[\widetilde{H}_{i,1},\widetilde{H}_{i,1}]=0$ yields $[\mathsf{u}_b,\mathsf{u}_a]=0$.
\end{proof}
Finally, Proposition \ref{thm-sup} remains true in the context of the present section, thus we have obtained a different, but equivalent, point of view on Schur-Weyl duality for the degenerate affine Hecke algebra and the super Yangian $Y_\lambda(\mathfrak{sl}(m|n))$.

\section{Affine super Yangians and degenerate double affine Hecke algebras}\label{asYangDDAHA}

\subsection{The affine Lie superalgebra $\widehat{\mfsl}(p|m|n-p)$}
Let $\mathbf{s}=(s_0,s_1,\dots,s_{m+n})$ where $s_i=0$ or $1$ and $1$ occurs exactly $n+1$ times. Then, $\mathbf{s}$ is called a \emph{parity sequence}. The standard one is denoted by $$s^{(0)}=(1,\underbrace{0,\cdots,0}_m ,\underbrace{1 \cdots, 1}_n).$$
For $0\le p\le n$, we define the parity sequence $s^{(p)}$ by
$$s^{(p)}=(\underbrace{1, \dots, 1}_{p+1}, \underbrace{0, \dots, 0}_m, \underbrace{1,\dots,1}_{n-p} ).$$
For the sake of simplicity, hereafter, we sometimes identify $\{1,\cdots,m+n\}$ with $\mathbb{Z}/(m+n)\mathbb{Z}$.

For $0\leq p\leq n$, let $\mathfrak{gl}(p|m|n-p)=\bigoplus_{1\leq i,j\leq m+n}\limits \mathbb{C}E_{i,j}$ is the Lie superalgebra whose commutator relations are given by
\begin{equation*}
[E_{i,j},E_{k,l}]=\delta_{j,k}E_{i,l}-\delta_{i,l}(-1)^{(s^{(p)}_i+s^{(p)}_j)(s^{(p)}_k+s^{(p)}_l)}E_{k,j}.
\end{equation*}
In the case that $p=0$, $\mathfrak{gl}(p|m|n-p)$ coincides with $\mathfrak{gl}(m|n)$. We can define a non-degenerate supersymmetric bilinear form on $\mathfrak{gl}(p|m|n-p)$ by
\begin{equation*}
(E_{i,j},E_{k,l})=\delta_{j,k}\delta_{i,l}(-1)^{s^{(p)}_i}.
\end{equation*}
We introduce the Lie sub-superalgebra
\begin{equation*}
 \mathfrak{sl}(p|m|n-p)=\{X=(x_{i,j})_{1\leq i,j\leq m+n}\in\mathfrak{gl}(p|m|n-p)\mid\text{str}^{(p)}(X)=\sum_{i=1}^{m+n}\limits(-1)^{s^{(p)}_i} x_{i,i}=0\},
\end{equation*}
where $\text{str}^{(p)}$ is the supertrace of $\mathfrak{gl}(p|m|n-p)$.
The affine Lie superalgebra $\widehat{\mfsl}(p|m|n-p)$ is
\begin{equation*}
\widehat{\mfsl}(p|m|n-p)=\mfsl(p|m|n-p)\otimes_{\C}\C[t,t^{-1}]\oplus\C \mathfrak{c}
\end{equation*}
for a central element $\mathfrak{c}$ that is then equipped with the superbracket:
\begin{align}
[X_1\otimes t^r,X_2 \otimes t^s] = [X_1,X_2] \otimes t^{r+s} + r\delta_{r+s,0} (X_1,X_2) \mathfrak{c}\text{ for all }X_1,X_2\in\mfsl_{m|n}.
\end{align}
Let us set the $(m+n)\times (m+n)$-matrix $(a^{(p)}_{i,j})_{0\leq i,j\leq m+n-1}$ to be:
\begin{align}
		a^{(p)}_{i,j}=
			\begin{cases}
				(-1)^{s^{(p)}_i} +(-1)^{s^{(p)}_{i+1}} & \textrm{ if }  i=j,\\
				-(-1)^{s^{(p)}_{i+1}} & \textrm{ if } j=i+1,\\
				-(-1)^{s^{(p)}_i}& \textrm{ if }  j=i-1,\\
				-(-1)^{s^{(p)}_{m+n}}& \textrm{ if }  (i,j)=(m+n-1,0),(0,m+n-1),\\
				0& \textrm{ otherwise}.
			\end{cases}
\end{align}
We note that $a_{i,j}^{(0)}=a_{i,j}$ for $1\leq i,j\leq m+n-1$.
\begin{proposition}[Theorem 4.1.1 in \cite{Y}]
		Suppose that $m,n \ge 2$ and $m\ne n$.
		Then $\widehat{\mathfrak{sl}}(p|m|n-p)$ is isomorphic to the Lie superalgebra over $\C$ defined by the generators $\{x_i^{\pm},h_i\}$ for $0 \le i \le m+n-1$ subject to the relations:	
			\begin{align}
				[h_i,h_j]&=0, \label{oatmilk} \quad
				[h_i,x_j^{\pm}] = \pm a^{(p)}_{i,j} x_j^{\pm},\\
				[x_i^+,x_j^-] &= \delta_{i,j}h_i,\quad
				\textnormal{ad}(x_i^{\pm})^{1+|a^{(p)}_{i,j}| } x_j^{\pm}=0,\\
				[x_{i}^{\pm},x_i^{\pm}] &=0\text{ if }s^{(p)}_i\neq s^{(p)}_{i+1},\\
				[[x_{i-1}^{\pm},x_i^{\pm}],&[x_{i+1}^{\pm},x_i^{\pm}]]=0\text{ if }s^{(p)}_i\neq s^{(p)}_{i+1}, \label{latte}
			\end{align}
			where $h_i$ is even and $x_i^{\pm}$ is even (resp. odd) if $s^{(p)}_i=s^{(p)}_{i+1}$ (resp. $s^{(p)}_i\neq s^{(p)}_{i+1}$).
   
	The isomorphism $\xi$ is given by
			\begin{align*}
				\xi(h_i)&=\begin{cases}
    -(-1)^{s^{(p)}_1}E_{1,1}+(-1)^{s^{(p)}_{m+n}}E_{m+n,m+n}+\mathfrak{c} &\text{ if } i=0, \\ (-1)^{s^{(p)}_i} E_{i,i} - (-1)^{s^{(p)}_{i+1}} E_{i+1,i+1} &\text{ if } 1 \le i \le m+n-1, \end{cases}\\
				\xi(x_i^+)&=	\begin{cases}
    E_{m+n,1}\otimes t &\text{ if } i=0, \\
    E_{i,i+1} &\text{ if } 1 \le i \le m+n-1, \end{cases}\\
				\xi(x_i^-)&= 	\begin{cases} 
    (-1)^{s^{(p)}_{m+n}}E_{1,m+n}\otimes t^{-1}&\text{ if } i=0, \\
    (-1)^{s^{(p)}_i} E_{i+1,i} &\text{ if } 1 \le i \le m+n-1.  \end{cases}
			\end{align*}
	\end{proposition}
 Hereafter, we sometimes denote $\widehat{\mfsl}(0|m|n)$ by $\widehat{\mfsl}(m|n)$. The affine version of Theorem~\ref{chengbook} was given in Theorem 10.1 in \cite{Flicker}. Let us set
$\mathbb{C}[W^{\text{aff}}_l]=\mathbb{C}[z^{\pm}_1,\cdots,z^{\pm}_l]\rtimes\mathbb{C}[S_l]$.
\begin{definition}
A left $\mathfrak{sl}_{m|n}$-module is said to be integrable if it is a direct sum of finite-dimensional modules. We say that a left module $N$ over $\widehat{\mathfrak{sl}}(m|n)$ (or over the affine Yangian) is integrable of level $l$ if this is so when $N$ is viewed as a left $\mathfrak{sl}_{m|n}$-module.
\end{definition}

\begin{theorem}[Theorem 10.1 in \cite{Flicker}]\label{Flicker}
\begin{enumerate}
\item Let $M$ be a right $\C[W^{\text{aff}}_l]$-module. Then we can define an action of $\widehat{\mathfrak{sl}}(m|n)$ on $M\otimes_{\mathbb{C}[S_l]}\mathbb{C}(m|n)^{\otimes l}$ by
\begin{align*}
(E_{i,j}t^s)(\mathbf{m}\otimes\mathbf{v})&=\sum_{r=1}^l\limits\mathbf{m}z^s_r\otimes E^{(r)}_{i,j}\mathbf{v}\text{ if }i\neq j,
\end{align*}
where $E^{(r)}_{i,j}=\id^{\otimes r-1}\otimes E_{i,j}\otimes\id^{\otimes l-r}$.
\item Suppose that $m\neq n$ and $l<m+n$. We have an equivalence from the category of right $\C[W^{\text{aff}}_l]$-modules to the category of integrable left $U(\widehat{\mathfrak{sl}}(m|n))$-modules of level $l$ (with trivial action of the central element $\mathfrak{c}$) given by $M\mapsto M\otimes_{\mathbb{C}[S_l]} \mathbb{C}(m|n)^{\otimes l}$ and $f\mapsto f\otimes1^{\otimes l}$ for any right $\C[W^{\text{aff}}_l]$-module $M$ and any homomorphism of right $\C[W^{\text{aff}}_l]$-modules $f$.
\end{enumerate}
\end{theorem}
\subsection{Affine super Yangians}
We can define the affine super Yangian associated with $(a^{(p)}_{i,j})_{0\leq i,j\leq m+n-1}$ as in Proposition 2.23 in \cite{Ue2}. 
\begin{definition}\label{defaffsup}
Let $\lambda,\beta\in\mathbb{C}$. Suppose that $m, n\geq2$ and $m\neq n$. The affine super Yangian $Y^{(p)}_{\lambda,\beta}(\widehat{\mathfrak{sl}}(m|n))$ is the associative superalgebra over $\C$ generated by $X_{i,r}^{+}, X_{i,r}^{-}, H_{i,r}$ $(i \in \{0,1,\cdots, m+n-1\}, r = 0,1)$ subject to the following defining relations:
\begin{gather}
[H_{i,r}, H_{j,s}] = 0,\label{Eq2.1}\\
[X_{i,0}^{+}, X_{j,0}^{-}] = \delta_{i,j} H_{i, 0},\label{Eq2.2}\\
[X_{i,1}^{+}, X_{j,0}^{-}] = \delta_{i,j} H_{i, 1} = [X_{i,0}^{+}, X_{j,1}^{-}],\label{Eq2.3}\\
[H_{i,0}, X_{j,r}^{\pm}] = \pm a^{(p)}_{i,j} X_{j,r}^{\pm},\label{Eq2.4}\\
[\tilde{H}_{i,1}, X_{j,0}^{\pm}] = \pm a^{(p)}_{i,j} X_{j,1}^{\pm}\text{ if }(i,j)\neq(0,m+n-1),(m+n-1,0),\label{Eq2.5}\\
[\tilde{H}_{0,1}, X_{m+n-1,0}^{\pm}] = \pm a^{(p)}_{0,m+n-1}\left(X_{m+n-1,1}^{\pm}-\beta X_{m+n-1, 0}^{\pm}\right),\label{Eq2.6}\\
[\tilde{H}_{m+n-1,1}, X_{0,0}^{\pm}] = \pm a^{(p)}_{m+n-1,0} \left(X_{0,1}^{\pm}+\beta X_{0, 0}^{\pm}\right),\label{Eq2.7}\\
[X_{i, 1}^{\pm}, X_{j, 0}^{\pm}] - [X_{i, 0}^{\pm}, X_{j, 1}^{\pm}] = \pm a^{(p)}_{i,j}\dfrac{\lambda}{2} \{X_{i, 0}^{\pm}, X_{j, 0}^{\pm}\}\text{ if }(i,j)\neq(0,m+n-1),(m+n-1,0),\label{Eq2.8}\\
[X_{0, 1}^{\pm}, X_{m+n-1, 0}^{\pm}] - [X_{0, 0}^{\pm}, X_{m+n-1, 1}^{\pm}]=\pm a^{(p)}_{0,m+n-1}\dfrac{\lambda}{2} \{X_{0, 0}^{\pm}, X_{m+n-1, 0}^{\pm}\} -\beta[X_{0, 0}^{\pm}, X_{m+n-1, 0}^{\pm}],\label{Eq2.9}\\
(\ad X_{i,0}^{\pm})^{1+|a^{(p)}_{i,j}|} (X_{j,0}^{\pm})= 0 \text{ if }i \neq j, \label{Eq2.10}\\
[X^\pm_{i,0},X^\pm_{i,0}]=0\text{ if }s^{(p)}_i\neq s^{(p)}_{i+1},\label{Eq2.11}\\
[[X^\pm_{i-1,0},X^\pm_{i,0}],[X^\pm_{i,0},X^\pm_{i+1,0}]]=0\text{ if }s^{(p)}_i\neq s^{(p)}_{i+1},\label{Eq2.12}
\end{gather}
where the generators $X^\pm_{m+p, r}$ and $X^\pm_{p, r}$ are odd, all other generators are even and we set $\wt{H}_{i,1} = H_{i,1}-\dfrac{\lambda}2 H_{i,0}^2$. 
\end{definition}
In the case when $p=0$, Definition~\ref{defaffsup} is the same as the one in Proposition 2.23 in \cite{Ue2}. We also note that there exists a homomorphism from $U(\widehat{\mathfrak{sl}}(p|m|n-p))$ to $Y^{(p)}_{\lambda,\beta}(\widehat{\mathfrak{sl}}(m|n))$ given by $x^\pm_i\mapsto X^\pm_{i,0},\ h_i\mapsto H_{i,0}$.
We sometimes denote the image of $x\in U(\widehat{\mathfrak{sl}}(p|m|n-p))$ via this homomorphism by $x$.

By the definition of the affine super Yangian $Y^{(p)}_{\lambda,\beta}(\widehat{\mathfrak{sl}}(m|n))$, we obtain the following lemma.
\begin{lemma}
Let $a$ be a complex number. There exists an isomorphism $\tau_a\colon Y^{(p)}_{\lambda,\beta}(\widehat{\mathfrak{sl}}(m|n))\to Y^{(p-1)}_{\lambda,\beta}(\widehat{\mathfrak{sl}}(m|n))$ given by
\begin{align*}
\tau_a(A_{i,0})&=A_{i-1,0},\\
\tau_a(A_{i,1})&=\begin{cases}
A_{i-1,1}+aA_{i-1,0}&\text{ if }1\leq i\leq m+n-1,\\
A_{m+n-1,1}+(a-\beta)A_{m+n-1,0}&\text{ if }i=0
\end{cases}
\end{align*}
for $A=H,X^\pm$.
\end{lemma}

\subsection{DDAHA and Cherednik algebra}
Let us recall the definitions of the degenerate double affine Hecke algebra (DDAHA) and the rational Cherednik algebra.
\begin{definition}\label{ddahadef1}
Let $t,c\in\C$. The \emph{degenerate double affine Hecke algebra} (DDAHA) $\mathbb{H}_{t,c}(S_{\ell})$  of type $\mathfrak{gl}_l$ is the associative algebra over $\C$ generated by $H^{\text{deg}}_c(S_l)=\C[\mathsf{x}_1^{\pm 1},\dots,\mathsf{x}_l^{\pm  1}] \rtimes\mathbb{C}[S_l$] and $\mathbb{C}[u_1,\cdots,u_l]$ with the following relations:
		\begin{align}
		\sigma_i u_i &= u_{i+1} \sigma_i -c, \\
		\sigma_i u_j &= u_j \sigma_i \textrm{ for $1 \le j \le  l$, $j  \ne  i,i+1$},\\
		[u_i,\mathsf{x}_j] &=  \begin{cases} t \mathsf{x}_i +c \sum_{1\le k <i} \mathsf{x}_k \sigma_{k,i} + c\sum_{i< k \le l} \mathsf{x}_i \sigma_{i,k} & i=j, \\
		-c\mathsf{x}_j \sigma_{j,i} & i>j,\\
		-c\mathsf{x}_i \sigma_{i,j} & i<j.
		\end{cases}
		\end{align}
	\end{definition}
 \begin{definition}\label{defrc}
Let $t,c\in\C$. The rational Cherednik algebra $\mathsf{H}_{t,c}(S_{\ell})$ of type $\mfgl_{\ell}$ is the associative algebra generated by commuting elements $\mathsf{x}_1,\ldots,\mathsf{x}_{\ell}$, commuting elements $\mathsf{y}_1,\ldots,\mathsf{y}_{\ell}$, and $\C[S_{\ell}]$, subject to the following relations:
\begin{gather}
\sigma \mathsf{x}_i \sigma^{-1}= \mathsf{x}_{\sigma(i)},\label{defrna-1}\\
\sigma \mathsf{y}_j \sigma^{-1} =  \mathsf{y}_{\sigma(j)}\label{defrna-2},\\
[\mathsf{y}_j,\mathsf{x}_i]=-c \sigma_{i,j}\text{ if }i\neq j,\label{defrna-3}\\
[\mathsf{y}_i,\mathsf{x}_i]=t+c\sum_{k\neq i} \sigma_{k,i}.\label{defrna-4}
\end{gather}
\end{definition}
Let us introduce new elements of $\mathsf{H}_{t,c}(S_{\ell})$ as follows:
\begin{align*}
\mathcal{U}_i&=\dfrac{t}{2}+\mathsf{x}_i \mathsf{y}_i+c\sum_{j<i}\sigma_{i,j},\ 
\mathfrak{y}_i=\mathcal{U}_i+\dfrac{c}{2}\sum_{j\neq i}\limits\text{sign}(j-i)\sigma_{i,j}.
\end{align*}
As for $\mathcal{U}_i$ and $\mathfrak{y}_i$, we obtain the following lemma.
\begin{proposition}[Theorem~3.8 in \cite{DO}, Proposition~ 4.1 in \cite{suz} and Proposition 2.7 in \cite{Gu1}]\label{Prop87} 
\begin{enumerate}
\item There is an embedding from $H^{\deg}_{c}(S_l)$ into $\mathsf{H}_{t,c}(S_{\ell})$ given by
\begin{equation*}
\sigma_i\mapsto\sigma_{i,i+1},\ \mathsf{u}_i\mapsto\mathcal{U}_i.
\end{equation*}
%\textup{(2)}\  We denote $\mathfrak{y}_i\in H^{\deg}_{c}(S_l)$ via %the embedding given in (1) by $\mathfrak{y}_i$.
%Then, we find that
%\begin{equation}
%[\mathcal{U}_i,\mathcal{U}_j]=0,\ \mathsf{y}_i=\mathfrak{y}_i-%dfrac{c}{2}\sum_{j\neq i}\sigma_{j,i}.\label{prop96-4.5}
%\end{equation}
\item 
By using $\mathfrak{y}_k$, we can rewrite \eqref{defrna-2}-\eqref{defrna-4} as
\begin{gather}
\sigma \mathfrak{y}_i\sigma^{-1}=\mathfrak{y}_{\sigma(i)},\label{prop96-1}\\
[\mathfrak{y}_j,\mathsf{x}_i]=-\dfrac{c}{2}(\mathsf{x}_i+\mathsf{x}_j)\sigma_{i,j}\text{ if }i\neq j,\label{prop96-3}\\
[\mathfrak{y}_i,\mathsf{x}_i]=t\mathsf{x}_i+\dfrac{c}{2}\sum_{j\neq i}\limits(\mathsf{x}_i+\mathsf{x}_j)\sigma_{i,j}\label{prop96-4}
\end{gather}
for $\sigma\in S_l$ and $1\leq i,j\leq l$.
\end{enumerate}
\end{proposition}
Observe that, under the embedding in the first part of the previous proposition, $\mathfrak{z}_i \mapsto \mathfrak{y}_i$. Moreover, we must have $[\mathcal{U}_i,\mathcal{U}_j]=0$ for the first part to be true.

By using the rational Cherednik algebra $\mathsf{H}_{t,c}(S_{\ell})$, Suzuki \cite{suz} gave another presentation of the DDAHA $\mathbb{H}_{t,c}(S_{\ell})$.
\begin{proposition}[Proposition 4.1 in \cite{suz}]\label{Prop9800}
The degenerate double affine Hecke algebra $\mathbb{H}_{t,c}(S_{\ell})$ is the localization of $\mathsf{H}_{t,c}$ at $\mathsf{x}_1\cdots \mathsf{x}_{\ell}$. In other words, $\mathbb{H}_{t,c}(S_{\ell})$ is obtained from $\mathsf{H}_{t,c}(S_{\ell})$ by adding the inverses $\mathsf{x}_1^{-1}, \mathsf{x}_2^{-1}, \ldots, \mathsf{x}_{\ell}^{-1}$.
\end{proposition}
By \eqref{prop96-3} and \eqref{prop96-4}, we obtain the following corollary:
\begin{corollary}
We obtain
\begin{gather}
[\mathsf{x}_i^{-1},\mathfrak{y}_j]=-\dfrac{c}{2}(\mathsf{x}_i^{-1}+\mathsf{x}_j^{-1})\sigma_{i,j}\text{ if }i\neq j,\label{prop96-5}\\
[\mathsf{x}_i^{-1},\mathfrak{y}_i]=t\mathsf{x}_i^{-1}+\dfrac{c}{2}\sum_{j\neq i}\limits(\mathsf{x}_i^{-1}+\mathsf{x}_j^{-1})\sigma_{i,j}\label{prop96-6}
\end{gather}
for $1\leq i,j\leq l$.
\end{corollary}
For the proof of the Schur-Weyl duality in the affine setting, we will need the super Yangian of $\mathfrak{sl}(m|n)$ associated to a shift of the standard parity sequence.
\begin{definition}
We define $Y^{(p)}_\lambda(\mathfrak{sl}(m|n))$ as an associative superalgebra generated by $\{X^\pm_{i,r},H_{i,r}\mid 1\leq i\leq m+n-1,r=0,1\}$ with the relations \eqref{Eq2.1}-\eqref{Eq2.5}, \eqref{Eq2.8}, \eqref{Eq2.10}, \eqref{Eq2.11} and \eqref{Eq2.12}.
\end{definition}
\begin{remark}
In the case $p=0$, $Y^{(p)}_\lambda(\mathfrak{sl}(m|n))$ is nothing but $Y_{\lambda}(\mathfrak{sl}(m|n))$ defined in Definition~\ref{deffinsup}.
\end{remark}
For $1\leq i\leq m+n$, let $i^{(p)}$ be $\sum_{j=i+1}^{m+n}\limits (-1)^{s^{(p)}_j}$.
For $1\leq i\leq m+n-1$, we can define two elements of $Y_{\lambda}^{(p)}(\mathfrak{sl}_{m|n})$:
\begin{align*}
J(h_i)&=\widetilde{H}_{i,1}+\lambda\vartheta_i,\ 
J(x^\pm_i)=X^\pm_{i,1}+\lambda\varpi^\pm_i,
\end{align*}
by the same way as in the case $p=0$.

By the same way as Theorem~\ref{finite case-1}, we can prove the following theorem.
\begin{theorem}\label{finite case}
Suppose that $M$ is a right $\mathbb{H}_{t,\lambda}(S_{\ell})$-module. 
We can define an action of $Y_{\lambda}^{(p)}(\mathfrak{sl}(m|n))$ on $M\otimes_{\C[S_l]}\mathbb{C}(p|m|n-p)^{\otimes l}$ by
\begin{align}
H_{i,0} (\mathbf{m}\ot\mathbf{v})&=\mathbf{m}\ot ((-1)^{s^{(p)}_{i}}E_{i,i}-(-1)^{s^{(p)}_{i+1}}E_{i+1,i+1})\mathbf{v},\label{defacth0}\\
X_{i,0}^{+} (\mathbf{m}\ot\mathbf{v})&=\mathbf{m}\ot E_{i,i+1}\mathbf{v},\label{defactx+0}\\
X_{i,0}^{-} (\mathbf{m}\ot\mathbf{v})&=(-1)^{s^{(p)}_{i}}\mathbf{m}\ot E_{i+1,i}\mathbf{v},\label{defactx-0}\\
J(h_i)(\mathbf{m}\ot\mathbf{v})&=\sum_{k=1}^l\limits\mathbf{m}\mathfrak{y}_{k}\otimes ((-1)^{s^{(p)}_{i}}E^{(k)}_{i,i}-(-1)^{s^{(p)}_{i+1}}E^{(k)}_{i+1,i+1})\mathbf{v},\label{defacth1}\\
J(x^+_i)(\mathbf{m}\ot\mathbf{v})&=\sum_{k=1}^l\limits\mathbf{m}\mathfrak{y}_{k}\otimes E^{(k)}_{i,i+1}\mathbf{v},\label{defactx+1}\\
J(x^-_i)(\mathbf{m}\ot\mathbf{v})&=(-1)^{s^{(p)}_{i}}\sum_{k=1}^l\limits\mathbf{m}\mathfrak{y}_{k}\otimes E^{(k)}_{i+1,i}\mathbf{v}\label{defactx-1}
\end{align}
for $\textbf{m}\in M$ and $\mathbf{v}\in\C(p|m|n-p)^{\otimes l}$.
\end{theorem}

\subsection{Schur-Weyl duality for the Affine Super Yangian and DDAHA}\label{SWasYang}

For $p\in\{0,1,\ldots,n\}$, Let $\C(p|m|n-p)$ be the $(m+n)$-dimensional super vector space $\bigoplus_{1\leq i\leq m+n}\C e_i$, where $e_i$ is even if $p+1\leq i\leq p+m$ and otherwise odd. We define an isomorphism
\begin{equation*}
\phi^{(p)}\colon\C(p|m|n-p)^{\otimes l}\to\C(p+1|m|n-p-1)^{\otimes l},\ \bigotimes_{g=1}^le_{i_g}\mapsto \bigotimes_{g=1}^le_{i_g+1}.
\end{equation*}
Let us assume that $M$ is a right $\mathbb{H}_{t,\lambda}(S_{\ell})$-module.
We extend $\phi^{(p)}$ to an isomorphism from $M\otimes_{\C[S_l]}\mathbb{C}(p|m|n-p)$ to $M\otimes_{\C[S_l]}\mathbb{C}(p+1|m|n-p-1)$ by
\begin{equation*}
\phi^{(p)}\left(\mathbf{m}\otimes\bigotimes_{g=1}^le_{i_g}\right)=\mathbf{m} \left(\prod_{k=1}^l\limits\mathsf{x}_{k}^{-\delta_{i_k,m+n-1}}\right)\otimes\bigotimes_{g=1}^le_{i_g+1}. 
\end{equation*}
Hereafter, in order to simplify the notation, we will sometimes denote $\phi^{(p)}$ by $\phi$. Moreover, we denote the $i$-th term of the right-hand side of equation numbered $N$ as $(N)_i$.
\begin{lemma}\label{tau} \begin{enumerate}
\item Suppose that $a=-\dfrac{\lambda}{2}$. The following relation holds for $2\leq i\leq m+n-1$, $\mathbf{m}\in M$, $\mathbf{v}\in\mathbb{C}(p|m|n-p)^{\otimes l}$ and $A=H,X^\pm$:
\begin{equation}
        \phi\Big(\tau_a(A_{i,1})(\mathbf{m} \otimes \mathbf{v})\Big) = A_{i,1}\Big(\phi(\mathbf{m}\ot\mathbf{v})\Big).\label{dark1}
\end{equation}
\item Suppose that $a=-\dfrac{\lambda}{2}$ and $\beta=-t+\dfrac{(m-n)\lambda}{2}$. We obtain
\begin{equation}
\phi^2\Big(\tau_a^2(A_{1,1})(\mathbf{m} \otimes \mathbf{v})\Big) = A_{1,1}\Big(\phi^2(\mathbf{m}\ot\mathbf{v})\Big)\label{dark2}
 \end{equation}
for $\mathbf{m}\in M$, $\mathbf{v}\in\mathbb{C}(p|m|n-p)^{\otimes l}$ and $A=H,X^\pm$.
\end{enumerate}
\end{lemma}
\begin{proof}
By \eqref{Eq2.3} and \eqref{Eq2.4}, it is enough to show the case $A=X^+$.
\begin{enumerate}
\item
By \eqref{well-defined}, we can assume that $\mathbf{v}=\bigotimes_{g=1}^le_{i_g}$ and write $j_1<j_2<\cdots<j_S$ for all the sub-indices such that $i_{j_s}=m+n$.
By the definition of $J(x^+_i)$ and $\tau_a$, we need to show that $X^+_{i,1}\big(\phi(\mathbf{m}\ot\mathbf{v})\big)-\phi\left(\tau_a(X^+_{i,1})(\mathbf{m}\ot\mathbf{v})\right)$ vanishes. Let us start with
\begin{align}
 X^+_{i,1}\big(\phi(\mathbf{m}\ot\mathbf{v})\big)-\phi\big(\tau_a(X^+_{i,1})(\mathbf{m}\ot\mathbf{v})\big) & =J(x^+_i)\left(\phi(\mathbf{m}\ot\mathbf{v})\right)-\phi\big(J(x^+_{i-1})(\mathbf{m}\ot\mathbf{v})\big)\nonumber\\
&\quad-\lambda\varpi^+_i\big(\phi(\mathbf{m}\ot\mathbf{v})\big)+\lambda\phi\big(\varpi^+_{i-1}(\mathbf{m}\ot\mathbf{v})\big)\nonumber\\
&\quad+\dfrac{\lambda}{2}\phi\big(X^+_{i-1,0}(\mathbf{m}\ot\mathbf{v})\big).\label{eq23-1}
\end{align}
By the definition of the action of $J(x^+_i)$ and $\phi$, we obtain
\begin{align}
J(x^+_i)\Big(\phi(\mathbf{m}\ot\mathbf{v})\Big)
&=\sum_{k=1}^l\limits \mathbf{m}\left(\prod_{s=1}^S\limits \mathsf{x}_{j_s}^{-1}\right)\mathfrak{y}_{k}\otimes E^{(k)}_{i,i+1}\big(\phi(\mathbf{v})\big)\nonumber\\
&=\sum_{\substack{k=1\\k\neq j_s}}^l\limits \sum_{r=1}^S\limits\mathbf{m}\mathsf{x}_{j_1}^{-1}\cdots \mathsf{x}_{j_{r-1}}^{-1}[\mathsf{x}_{j_{r}}^{-1},\mathfrak{y}_{k}]\mathsf{x}_{j_{r+1}}^{-1}\cdots \mathsf{x}_{j_S}^{-1}\otimes E^{(k)}_{i,i+1}\big(\phi(\mathbf{v})\big)\nonumber\\
&\quad+\sum_{\substack{k=1\\k\neq j_s}}^l\limits \mathbf{m}\mathfrak{y}_{k}\left(\prod_{s=1}^S\limits \mathsf{x}_{j_s}^{-1}\right)\otimes E^{(k)}_{i,i+1}\big(\phi(\mathbf{v})\big).\label{eq23-2}
\end{align}
Observe that, since we are assuming that $2\le i\le m+n-1$, we can consider only the values of $k$ that are $\neq j_s$ for $s=1,2,\ldots, S$. By \eqref{prop96-5}, we have
\begin{align}
\eqref{eq23-2}_1&=-\dfrac{\lambda}{2}\sum_{\substack{k=1\\k\neq j_r}}^l\limits \sum_{r=1}^S\limits\mathbf{m}\mathsf{x}_{j_1}^{-1}\cdots \mathsf{x}_{j_{r-1}}^{-1}(\mathsf{x}_{j_{r}}^{-1}+\mathsf{x}_{k}^{-1})\sigma_{j_r,k}\mathsf{x}_{j_{r+1}}^{-1}\cdots \mathsf{x}_{j_S}^{-1}\otimes E^{(k)}_{i,i+1}\Big(\phi(\mathbf{v})\Big)\nonumber\\
&=-\dfrac{\lambda}{2}\sum_{\substack{k=1\\k\neq j_r}}^l\limits \sum_{r=1}^S\limits\mathbf{m}\mathsf{x}_{j_1}^{-1}\cdots \mathsf{x}_{j_{r-1}}^{-1}\mathsf{x}_{j_{r}}^{-1}\mathsf{x}_{j_{r+1}}^{-1}\cdots \mathsf{x}_{j_S}^{-1}\otimes \sigma_{j_r,k}\Big(E^{(k)}_{i,i+1}\big(\phi(\mathbf{v})\big)\Big)\nonumber\\
&\quad-\dfrac{\lambda}{2}\sum_{\substack{k=1\\k\neq j_r}}^l\limits \sum_{r=1}^S\limits\mathbf{m}\mathsf{x}_{j_1}^{-1}\cdots \mathsf{x}_{j_{r-1}}^{-1}\mathsf{x}_{k}^{-1}\mathsf{x}_{j_{r+1}}^{-1}\cdots \mathsf{x}_{j_S}^{-1}\otimes \sigma_{j_r,k}\Big(E^{(k)}_{i,i+1}\big(\phi(\mathbf{v})\big)\Big).\label{align959}
\end{align}
By the assumption that $i\neq 0$, the definition of $\phi^{(p)}$ and the action of $J(x^+_{i-1})$, we have
\begin{align}
\eqref{eq23-2}_2&=\sum_{\substack{k=1\\k\neq j_s}}^l\limits \mathbf{m}\mathfrak{y}_{k}\left(\prod_{s=1}^S\limits \mathsf{x}_{j_s}^{-1}\right)\otimes \phi\big(E^{(k)}_{i-1,i}(\mathbf{v})\big)=\phi\Big(J(x^+_{i-1})\big(\mathbf{m}\ot\mathbf{v}\big)\Big).\label{eq23-2.5}
\end{align}
By the definition of $\varpi^+_i$ and $\phi^{(p)}$, we obtain
\begin{align}
\lambda\varpi^+_i\big(\phi(\mathbf{m}\otimes \mathbf{v})\big)&=\left(\dfrac{\lambda}{2}\sum_{j\neq i,i+1}\limits\text{sign}(j-i)(-1)^{s^{(p+1)}_j}-\dfrac{\lambda}{2}i^{(p+1)}\right)\mathbf{m}\left(\prod_{s=1}^S\limits \mathsf{x}_{j_s}^{-1}\right)\otimes E_{i,i+1}\big(\phi(\mathbf{v})\big)\nonumber\\
&\quad+\dfrac{\lambda}{2}\sum_{j\neq i,i+1}\limits\sum_{\substack{k=1\\i_k=i}}^l\limits\sum_{\substack{d=1\\i_d=j-1}}^l\limits\text{sign}(j-i)\mathbf{m}\left(\prod_{s=1}^S\limits \mathsf{x}_{j_s}^{-1}\right)\otimes \sigma_{k,d}\Big(E^{(k)}_{i,i+1}\big(\phi(\mathbf{v})\big)\Big)\nonumber\\
&\quad-\dfrac{\lambda}{2}\mathbf{m}\left(\prod_{s=1}^S\limits \mathsf{x}_{j_s}^{-1}\right)\otimes h_ix^+_i\left(\phi(\mathbf{v})\right)\label{align960}
\end{align}
and
\begin{align}
 \lambda\phi \Big(\varpi^+_{i-1}(\mathbf{m}\otimes \mathbf{v})\Big)
&=\left(\dfrac{\lambda}{2}\sum_{j\neq i-1,i}\limits\text{sign}(j-i+1)(-1)^{s^{(p)}_j}-\dfrac{\lambda}{2}(i-1)^{(p)}\right)\mathbf{m}\left(\prod_{s=1}^S\limits \mathsf{x}_{j_s}^{-1}\right)\otimes \phi\big(E_{i-1,i}(\mathbf{v})\big)\nonumber\\
&\quad+\dfrac{\lambda}{2}\sum_{j\neq i-1,i}\limits\sum_{\substack{k=1\\i_k=i}}^l\limits\sum_{\substack{d=1\\i_d=j}}^l\limits\text{sign}(j-i+1)\phi\Big(\mathbf{m}\otimes \sigma_{k,d}\big(E^{(k)}_{i-1,i}(\mathbf{v})\big)\Big)\nonumber\\
&\quad-\dfrac{\lambda}{2}\mathbf{m}\left(\prod_{s=1}^S\limits \mathsf{x}_{j_s}^{-1}\right)\otimes \phi\big(h_{i-1}x^+_{i-1}(\mathbf{v})\big).\label{align961}
\end{align}
Thus, we find that
\begin{align*}
\eqref{eq23-1}&=\eqref{align959}-\eqref{align960}+\eqref{align961}+\eqref{eq23-1}_5.
\end{align*}
We easily obtain
\begin{equation}
-\eqref{align960}_3+\eqref{align961}_3=0.\label{align961.1}
\end{equation}
Since the relations
\begin{align*}
\sum_{j\neq i,i+1}\limits\text{sign}(j-i)(-1)^{s^{(p+1)}_j}&=\sum_{j\neq i-1,i}\limits\text{sign}(j-i+1)(-1)^{s^{(p)}_j}+2,\\
i^{(p+1)}&=(i-1)^{(p)}+1,
\end{align*}
hold by a direct computation, we have
\begin{align}
-\eqref{align960}_1+\eqref{align961}_1+\eqref{eq23-1}_5&=0.\label{align961.2}
\end{align}
Thus, by \eqref{align961.1} and \eqref{align961.2}, we obtain
\begin{align*}
\eqref{eq23-1}&=\eqref{align959}-\eqref{align960}_2+\eqref{align961}_2.
\end{align*}
We can divide $\eqref{align960}_2$ and $\eqref{align961}_2$ as follows:
\begin{align}
\eqref{align960}_2&=\dfrac{\lambda}{2}\sum_{j\neq 1,i,i+1}\limits\sum_{\substack{k=1\\i_k=i}}^l\limits\sum_{\substack{d=1\\i_d=j-1}}^l\limits\text{sign}(j-i)\mathbf{m}\left(\prod_{s=1}^S\limits\mathsf{x}_{j_s}^{-1}\right)\otimes \sigma_{k,d}E^{(k)}_{i,i+1}\phi(\mathbf{v})\nonumber\\
&\quad-\dfrac{\lambda}{2}\sum_{\substack{k=1\\i_k=i}}^l\limits\sum_{\substack{d=1\\i_d=m+n}}^l\limits\mathbf{m}\left(\prod_{s=1}^S\limits\mathsf{x}_{j_s}^{-1}\right)\otimes \sigma_{k,d}E^{(k)}_{i,i+1}\phi(\mathbf{v}),\label{align960-1}\\
\eqref{align961}_2&=\dfrac{\lambda}{2}\sum_{j\neq i-1,i,m+n}\limits\sum_{\substack{k=1\\i_k=i}}^l\limits\sum_{\substack{d=1\\i_d=j}}^l\limits\text{sign}(j-i+1)\mathbf{m}\left(\prod_{s=1}^S\limits\mathsf{x}_{j_s}^{-1}\right)\otimes \phi\big(\sigma_{k,d}E^{(k)}_{i-1,i}(\mathbf{v})\big)\nonumber\\
&\quad+\dfrac{\lambda}{2}\sum_{\substack{k=1\\i_k=i}}\limits\sum_{\substack{d=1\\i_d=m+n}}^l\limits\mathbf{m}\left(\prod_{\substack{s=1\\j_s\neq d}}^S\limits\mathsf{x}_{j_s}^{-1}\right)\mathsf{x}_{i_k}^{-1}\otimes \phi\big(\sigma_{k,d}E^{(k)}_{i-1,i}(\mathbf{v})\big).\label{align961-1}
\end{align}
By a direct computation, we obtain
\begin{align}
\eqref{align960-1}_1-\eqref{align961-1}_1&=0,\\
\eqref{align960-1}_2-\eqref{align959}_1&=0,\\
\eqref{align961-1}_2+\eqref{align959}_2&=0.\label{last-1}
\end{align}
By \eqref{eq23-2}-\eqref{last-1}, we find that \eqref{eq23-1} is equal to zero, which completes the proof of part \textit{1} of the previous lemma.

\item To establish that \eqref{dark2} holds, it is enough to show that $X^+_{1,1}\big(\phi^2(\mathbf{m}\ot{\mathbf{v}})\big)-\phi^2\big(\tau_a^2(X^+_{1,1})(\mathbf{m} \otimes \mathbf{v})\big)$ vanishes. Let us start with
\begin{align}
X^+_{1,1}\big(\phi^2(\mathbf{m}\ot{\mathbf{v}})\big) & -\phi^2\big(\tau_a^2(X^+_{1,1})(\mathbf{m} \otimes \mathbf{v})\big)\nonumber\\
&=J(x^+_1)\big(\phi^2(\mathbf{m}\ot{\mathbf{v}})\big)-\phi^2\big(J(x^+_{m+n-1})(\mathbf{m} \otimes \mathbf{v})\big)\nonumber\\
&\quad-\lambda\varpi^+_1\big(\phi^2(\mathbf{m}\ot{\mathbf{v}})\big)+\lambda\phi^2\big(\varpi^+_{m+n-1}(\mathbf{m} \otimes {\mathbf{v}})\big)+(\beta+\lambda)\phi^2\big(X^+_{m+n-1,0}(\mathbf{m} \otimes \mathbf{v})\big)\label{eq569}
\end{align}
is equal to zero.

Suppose that $j_1 < \cdots < j_S$ (resp. $\gamma_1 < \cdots < \gamma_R$) are exactly the values of $j$ (resp. of $\gamma$) satisfying $i_j=m+n$ (resp. $i_\gamma=m+n-1$).
By the definition of the action of $J(x^+_1)$, we have
\begin{align}
J(x^+_1)\big(\phi^2(\mathbf{m}\otimes \mathbf{v})\big)
&=\sum_{s=1}^S\limits\sum_{r=1}^R\limits\mathbf{m} \mathsf{x}^{-1}_{j_1}\cdots \mathsf{x}^{-1}_{j_S}\mathsf{x}^{-1}_{\gamma_1}\cdots \mathsf{x}^{-1}_{\gamma_{r-1}}[\mathsf{x}^{-1}_{\gamma_r},\mathfrak{y}_{j_s}]\mathsf{x}^{-1}_{\gamma_{r+1}}\cdots \mathsf{x}^{-1}_{\gamma_{R}}\otimes E^{(j_s)}_{1,2}\big(\phi^2(\mathbf{v})\big)\nonumber\\
&\quad+\sum_{s=1}^S\limits\sum_{q=1}^S\limits \mathbf{m}\mathsf{x}^{-1}_{j_1}\cdots \mathsf{x}^{-1}_{j_{q-1}}[\mathsf{x}^{-1}_{j_q},\mathfrak{y}_{j_s}]\mathsf{x}^{-1}_{j_{q+1}}\cdots \mathsf{x}^{-1}_{j_{S}}\mathsf{x}^{-1}_{\gamma_1}\cdots \mathsf{x}^{-1}_{\gamma_R}\otimes E^{(j_s)}_{1,2}\big(\phi^2(\mathbf{v})\big)\nonumber\\
&\quad+\sum_{s=1}^S\limits \mathbf{m}\mathfrak{y}_{j_s}\mathsf{x}^{-1}_{j_1}\cdots \mathsf{x}^{-1}_{j_S}\mathsf{x}^{-1}_{\gamma_1}\cdots \mathsf{x}^{-1}_{\gamma_R}\otimes E^{(j_s)}_{1,2}\big(\phi^2(\mathbf{v})\big).\label{6.35}
\end{align}
By the definition of the action of $J(x^+_{m+n-1})$, we have
\begin{align}
\eqref{6.35}_3&=\phi^2\big(J(x^+_{m+n-1})(\mathbf{m}\otimes \mathbf{v})\big).\label{6.36-1}
\end{align}
By \eqref{prop96-5}, we can rewrite
\begin{align}
\eqref{6.35}_1&=-\dfrac{\lambda}{2}\sum_{s=1}^S\limits\sum_{r=1}^R\limits \mathbf{m}\mathsf{x}^{-1}_{j_1}\cdots \mathsf{x}^{-1}_{j_S}\mathsf{x}^{-1}_{\gamma_1}\cdots \mathsf{x}^{-1}_{\gamma_{r-1}}(\mathsf{x}^{-1}_{\gamma_r} + \mathsf{x}_{j_s}^{-1})\sigma_{j_s,\gamma_r} \mathsf{x}^{-1}_{\gamma_{r+1}}\cdots \mathsf{x}^{-1}_{\gamma_{R}}\otimes E^{(j_s)}_{1,2}\big(\phi^2(\mathbf{v})\big) \label{6.36}
\end{align}
and
\begin{align}
\eqref{6.35}_2&=-\dfrac{\lambda}{2}\sum_{s=1}^S\limits\sum_{\substack{q=1\\q\neq s}}^S\limits \mathbf{m}\mathsf{x}^{-1}_{j_1}\cdots \mathsf{x}^{-1}_{j_{q-1}}(\mathsf{x}^{-1}_{j_q}+\mathsf{x}^{-1}_{j_s})\sigma_{j_s,j_q}\mathsf{x}^{-1}_{j_{q+1}}\cdots \mathsf{x}^{-1}_{j_{S}}\mathsf{x}^{-1}_{\gamma_1}\cdots \mathsf{x}^{-1}_{\gamma_R}\otimes E^{(j_s)}_{1,2}\big(\phi^2(\mathbf{v})\big)\nonumber\\
&\quad+\sum_{s=1}^S\limits \mathbf{m}\mathsf{x}^{-1}_{j_1}\cdots \mathsf{x}^{-1}_{j_{s-1}}\left( t\mathsf{x}^{-1}_{j_s}+\dfrac{\lambda}{2}\sum_{\substack{q=1\\q\neq j_s}}^l (\mathsf{x}^{-1}_{j_s}+\mathsf{x}^{-1}_q)\sigma_{j_s,q}\right)\nonumber\\
&\qquad\qquad\qquad\qquad\qquad\qquad\qquad\qquad\qquad\qquad\qquad\cdot\mathsf{x}^{-1}_{j_{s+1}}\cdots \mathsf{x}^{-1}_{j_{S}}\mathsf{x}^{-1}_{\gamma_1}\cdots \mathsf{x}^{-1}_{\gamma_R}\otimes E^{(j_s)}_{1,2}\big(\phi^2(\mathbf{v})\big).\label{6.37}
\end{align}
By \eqref{6.36-1}-\eqref{6.37}, we obtain
\begin{align}
J(x^+_1)(\phi^2(\mathbf{m}\otimes \mathbf{v})) & -\phi^2(J(x^+_{m+n-1})(\mathbf{m}\otimes \mathbf{v}))\nonumber\\
&=t\sum_{s=1}^S\limits \mathbf{m}\mathsf{x}^{-1}_{j_1}\cdots \mathsf{x}^{-1}_{j_S}\mathsf{x}^{-1}_{\gamma_1}\cdots \mathsf{x}^{-1}_{\gamma_R}\otimes E^{(j_s)}_{1,2}(\phi^2(\mathbf{v}))\nonumber\\
&\quad+\dfrac{\lambda}{2}\sum_{s=1}^S\limits\sum_{\substack{q=1\\q\neq j_d,\gamma_h}}\limits^l\mathbf{m}\mathsf{x}^{-1}_{j_1}\cdots \mathsf{x}^{-1}_{j_{s-1}}\mathsf{x}^{-1}_{j_s}\mathsf{x}^{-1}_{j_{s+1}}\cdots \mathsf{x}^{-1}_{j_S}\mathsf{x}^{-1}_{\gamma_1}\cdots \mathsf{x}^{-1}_{\gamma_R}\otimes \sigma_{j_s,q}\Big(E^{(j_s)}_{1,2}\big(\phi^2(\mathbf{v})\big)\Big)\nonumber\\
&\quad+\dfrac{\lambda}{2}\sum_{s=1}^S\limits\sum_{\substack{q=1\\q\neq j_d,\gamma_h}}\limits^l\mathbf{m}\mathsf{x}^{-1}_{j_1}\cdots \mathsf{x}^{-1}_{j_{s-1}}\mathsf{x}^{-1}_q\mathsf{x}^{-1}_{j_{s+1}}\cdots \mathsf{x}^{-1}_{j_S}\mathsf{x}^{-1}_{\gamma_1}\cdots \mathsf{x}^{-1}_{\gamma_R}\otimes \sigma_{j_s,q}\Big(E^{(j_s)}_{1,2}(\phi^2(\mathbf{v})\big)\Big).\label{6.38}
\end{align}
By the definition of $\varpi^+_i$ and since $\sum_{j\neq 1,2}\limits(-1)^{s^{(p+2)}_j} = m-n+2$, we have
\begin{align}
\lambda\varpi^+_1(\phi^2(\mathbf{m}\otimes\mathbf{v})) %&=\left(\dfrac{\lambda}{2}\sum_{j\neq 1,2}\limits\text{sign}(j-1)(-1)^{s^{(p+2)}_j}-\dfrac{(m-n+2)\lambda}{2}\right)\mathbf{m}\mathsf{x}^{-1}_{j_1}\cdots \mathsf{x}^{-1}_{j_S}\mathsf{x}^{-1}_{\gamma_1}\cdots \mathsf{x}^{-1}_{\gamma_R}\otimes E_{1,2}\left(\phi^2(\mathbf{v})\right)\nonumber\\
&=\dfrac{\lambda}{2}\sum_{j=3}^{m+n}\limits\sum_{\substack{q=1\\i_q+2=j}}^l\limits\sum_{s=1}^S\limits\text{sign}(j-1)\mathbf{m}\mathsf{x}^{-1}_{j_1}\cdots \mathsf{x}^{-1}_{j_S}\mathsf{x}^{-1}_{\gamma_1}\cdots \mathsf{x}^{-1}_{\gamma_R}\otimes \sigma_{q,j_s}\Big(E^{(j_s)}_{1,2}\big(\phi^2(\mathbf{v})\big)\Big)\nonumber\\
&\quad-\dfrac{\lambda}{2}\mathbf{m}\mathsf{x}^{-1}_{j_1}\cdots \mathsf{x}^{-1}_{j_S}\mathsf{x}^{-1}_{\gamma_1}\cdots \mathsf{x}^{-1}_{\gamma_R}\otimes h_1x^+_1\phi^2(\mathbf{v})\label{6.39}
\end{align}
and
\begin{align}
\lambda\phi^2(\varpi^+_{m+n-1}(\mathbf{m}\otimes\mathbf{v}))&=\left(\dfrac{\lambda}{2}\sum_{j=1}^{m+n-2}\limits\text{sign}(j-m-n+1)(-1)^{s^{(p)}_j}\right)\phi^2\Big(\mathbf{m}\otimes E_{m+n-1,m+n}(\mathbf{v})\Big)\nonumber\\
&\quad+\dfrac{\lambda}{2}\sum_{j=1}^{m+n-2}\sum_{\substack{q=1\\i_q=j}}^l\limits\sum_{s=1}^S\limits\text{sign}(j-m-n+1)\phi^2\Big(\mathbf{m}\otimes \sigma_{q,j_s}(E^{(j_s)}_{m+n-1,m+n}(\mathbf{v}))\Big)\nonumber\\
&\quad-\dfrac{\lambda}{2}\mathbf{m}\otimes \phi^2\Big(h_{m+n-1}x^+_{m+n-1}(\mathbf{v})\Big).\label{6.40}
\end{align}
By the definition of $\phi$, we can rewrite $\eqref{6.40}_2$ as follows:
\begin{align}
\eqref{6.40}_2&=-\dfrac{\lambda}{2}\sum_{j=1}^{m+n-2}\sum_{\substack{q=1\\i_q=j}}^l\limits\sum_{s=1}^S\limits\mathbf{m}\mathsf{x}^{-1}_{j_1}\cdots \mathsf{x}^{-1}_{j_{s-1}}\mathsf{x}^{-1}_q\mathsf{x}^{-1}_{j_{s+1}}\cdots \mathsf{x}^{-1}_{j_S}\mathsf{x}^{-1}_{\gamma_1}\cdots \mathsf{x}^{-1}_{\gamma_R}\otimes \phi^2\Big(\sigma_{q,j_s}(E^{(j_s)}_{m+n-1,m+n}(\mathbf{v}))\Big).
\end{align}
Since \begin{align*}
\eqref{eq569}&=\eqref{6.38}-\eqref{6.39}+\eqref{6.40}+\eqref{eq569}_5.
\end{align*} and 
\begin{align}
\eqref{6.38}_2-\eqref{6.39}_1&=0,\\
\eqref{6.38}_3+\eqref{6.40}_2&=0,\\
\eqref{eq569}_5+\eqref{6.38}_1+\eqref{6.40}_1&=0, \label{btl}\\
-\eqref{6.39}_2+\eqref{6.40}_3&=0,
\end{align}
we have shown that \eqref{eq569} is equal to zero. Note that the assumption that $\beta+t=\frac{(m-n)\lambda}{2}$ was needed to obtain \eqref{btl}.
\end{enumerate}
\end{proof}
\begin{Corollary}
\begin{enumerate}
\item Suppose that $a=-\dfrac{\lambda}{2}$. The following relation holds for $1\leq i\leq m+n-2$, $\mathbf{m}\in M$, $\mathbf{v}\in\mathbb{C}(p|m|n-p)^{\otimes l}$ and $A=H,X^\pm$:
\begin{equation}
        \phi\Big(A_{i,1}(\mathbf{m} \otimes \mathbf{v})\Big) = \tau_a^{-1}(A_{i,1})\Big(\phi(\mathbf{m}\ot\mathbf{v})\Big).\label{dark11}
\end{equation}
\item Suppose that $a=-\dfrac{\lambda}{2}$ and $\beta=-t+\dfrac{(m-n)\lambda}{2}$. We obtain
\begin{equation}
\phi^2\Big(A_{m+n-1,1}(\mathbf{m} \otimes \mathbf{v})\Big) = \tau_a^{-2}(A_{m+n-1,1})\Big(\phi^2(\mathbf{m}\ot\mathbf{v})\Big)\label{dark21}
 \end{equation}
for $\mathbf{m}\in M$, $\mathbf{v}\in\mathbb{C}(p|m|n-p)^{\otimes l}$ and $A=H,X^\pm$.
\end{enumerate}
\end{Corollary}
By \eqref{Eq2.5}, in order to define the action of $Y^{(0)}_{\lambda,\beta}(\widehat{\mathfrak{sl}}(m|n))$ on $M\otimes_{\C[S_l]}\mathbb{C}(m|n)^{\otimes l}$, it is enough to set compatible actions of $X^\pm_{0,0}$ and $H_{i,r},X_{i,r}^{\pm}$ for $1\leq i\leq m+n-1$ and $r=0,1$.
\begin{theorem}\label{main theorem}
Assume that $c=\lambda$ and $\beta+t =\dfrac{(m-n)\la}2$.
For a right $\mathbb{H}_{t,c}(S_{\ell})$-module $M$, we can define the action of the affine super Yangian $Y^{(0)}_{\lambda,\beta}(\widehat{\mathfrak{sl}}(m|n))$ on $M\otimes_{\C[S_l]}\mathbb{C}(m|n)^{\otimes l}$ by \eqref{defacth0}-\eqref{defactx-1} and
\begin{gather}
X^+_{0,0}(\mathbf{m}\otimes\mathbf{v})=\sum_{r=1}^l\limits\mathbf{m}\mathsf{x}_r\otimes E^{(r)}_{m+n,1}\mathbf{v},\ 
X^-_{0,0}(\mathbf{m}\otimes\mathbf{v})=-\sum_{r=1}^l\limits\mathbf{m}\mathsf{x}_r^{-1}\otimes E^{(r)}_{1,m+n}\mathbf{v}.\label{0-}
\end{gather}
\end{theorem}
\begin{proof}
We prove this under the assumption that $n\geq3$. The case when $n=2$ and $m\geq3$ can be proven in a similar way. 
%We can define an action of $\widehat{\mathfrak{sl}}(p|m|n-p)$ on $M\otimes_{\C[S_l]}\C(p|m|n-p)$ by
%\begin{align*}
%(E_{a,b}t^s)(\mathbf{m}\otimes \mathbf{v})=\sum_{k=1}%^l\limits\mathbf{m}\mathsf{x}_k^s\otimes E^{(k)}_{a,b}\mathbf{v},\ \mathfrak{c}(\mathbf{m}\otimes \mathbf{v})=0
%\end{align*}
%for $1\leq a\neq b\leq m+n$, $\mathbf{m}\in M$ and $\mathbf{v}\in\C(p|m|n-p)^{\otimes l}$.
%In particular, in the case that $p=0$, we find that this action is compatible with \eqref{Eq2.11} and \eqref{Eq2.12}. 
By Theorem~\ref{Flicker}, we can naturally define the action of $\widehat{\mathfrak{sl}}(p|m|n-p)$ on $M\otimes_{\C[S_l]}\C(p|m|n-p)^{\otimes l}$. This action coincides with \eqref{0-} in the case when $p=0$.
By the definition of that action, we have
\begin{gather}
A_{0,0}(\mathbf{m}\ot{\mathbf{v}})=(\phi^{(p)})^{-1}\Big(A_{1,0}\big(\phi^{(p)}(\mathbf{m}\ot{\mathbf{v}})\big)\Big)\label{dark31}
\end{gather}
for $A=H,X^\pm$.
We  define the action of $A_{0,1}\in Y^{(p)}_{\lambda,\beta}(\widehat{\mathfrak{sl}}(m|n))$ on $M\otimes_{\C[S_l]}\C(p|m|n-p)^{\otimes l}$ as follows:
\begin{gather}
A_{0,1}(\mathbf{m}\ot{\mathbf{v}})=(\phi^{(p)})^{-1}\Big(\tau_a^{-1}(A_{0 ,1})\big(\phi^{(p)}(\mathbf{m}\ot{\mathbf{v}})\big)\Big) \;\text{ for } A=H,X^{\pm}.\label{dark32}
\end{gather}
By \eqref{dark11}, we deduce that
\begin{gather}
A_{0,r}(\mathbf{m}\ot{\mathbf{v}})=\phi^{-2}\Big(\tau_a^{-2}(A_{0,r})\left(\phi^2(\mathbf{m}\ot{\mathbf{v}})\right)\Big)\label{dark33}
\end{gather}
for $r=0,1$ and $A=H,X^{\pm}$. Consider a defining relation of $Y^{(0)}_{\lambda,\beta}(\widehat{\mathfrak{sl}}(m|n))$ involving $A_{0,r_1}^{\pm}$ and $B_{i,r_2}$ for $A,B=H,X^\pm$, $r_1,r_2=0,1$ and $0\leq i\leq m+n-2$. The same relation holds for the corresponding operators on $M\ot_{\C[S_l]} \C(m|n)^{\ot l}$ if and only if it holds for $\tau_a^{-1}(A_{0,r_1}^{\pm})$ and $\tau_a^{-1}(B_{i,r_2})$ by \eqref{dark11}, \eqref{dark31} and \eqref{dark32}. Since
 $\tau_a^{-1}(A_{0,r_1}^{\pm})$ and $\tau_a^{-1}(B_{i,r_2})$ do not depend on $A_{0,0}^{\pm}$ or $A_{0,1}^{\pm}$, those relations are known to hold by Theorem~\ref{finite case}. The same argument applies to the relations involving $A_{0,r_1}^{\pm}$ and $B_{m+n-1,r_2}^{\pm}$ by replacing $\phi$ by $\phi^2$, $\tau_a$ by $\tau_a^2$ and using \eqref{dark21} and \eqref{dark33} instead.
\end{proof}

\begin{theorem}\label{reverse}
Suppose that $m,n>l$ and $c=\lambda$, $\beta=-t+\dfrac{(m-n)\lambda}{2}$. We obtain an equivalence of categories from the category of right $\mathbb{H}_{t,c}(S_{\ell})$-modules to the category of integrable left modules over $Y^{(0)}_{\lambda,\beta}(\widehat{\mathfrak{sl}}(m|n))$ of level $l$ determined by $M\mapsto M\otimes_{\C[S_l]} \C(m|n)^{\otimes l}$ and $f\mapsto f\otimes 1^{\otimes l}$ for any right $\mathbb{H}_{t,c}(S_{\ell})$-module $M$ and any homomorphism of right $\mathbb{H}_{t,c}(S_{\ell})$-modules $f$.
\end{theorem} 
\begin{proof}
We prove this in the case that $m\geq3$. The cases $m=2$ and $n\geq3$ can be proven in a similar way. That $SW$ is full and faithful follows from Proposition \ref{thm-sup} and Theorem \ref{Flicker}, so it is enough to explain why it is essentially surjective.

Let us fix an integrable left $Y^{(0)}_{\lambda,\beta}(\widehat{\mathfrak{sl}}(m|n))$-module $N$ of level $l$.
Since there exists a natural homomorphism from $U(\widehat{\mathfrak{sl}}(m|n))$ to $Y^{(0)}_{\lambda,\beta}(\widehat{\mathfrak{sl}}(m|n))$, we can regard $N$ as a left $\widehat{\mathfrak{sl}}(m|n)$-module of level $l$.
By Theorem~\ref{Flicker}, there exists a $\mathbb{C}[W^{\text{aff}}_l]$-right module $M_1$ satisfying $M_1\otimes_{\mathbb{C}[S_l]} \mathbb{C}(m|n)^{\otimes l}\cong N$ as $\widehat{\mathfrak{sl}}(m|n)$-modules. Moreover, since there exists a natural homomorphism from $Y_{\lambda}(\mathfrak{sl}(m|n))$ to $Y^{(0)}_{\lambda,\beta}(\widehat{\mathfrak{sl}}(m|n))$, we can regard $N$ as a left $Y_{\lambda}(\mathfrak{sl}(m|n))$-module. Thus, by Theorem~\ref{flor}, we find that there exists an $H^{\deg}_\lambda(S_l)$-right module $M_2$ such that $M_2\otimes_{\mathbb{C}[S_l]} \mathbb{C}(m|n)^{\otimes l}\simeq N$. Since $\mathbb{C}[S_l]\subset\mathbb{C}[W^{\text{aff}}_l]$ and $\mathbb{C}[S_l]\subset H^{\deg}_\lambda(S_l)$, we have
an isomorphism $M_1\simeq M_2$ of $\mathbb{C}[S_l]$-right modules. We can denote them simply by $M$. It is enough to show that $M$ is an $\mathbb{H}_{t,c}(S_l)$-module when $c=\lambda$ and $t=\frac{(m-n)\lambda}{2}-\beta$. We already have actions of the generators $\mathsf{x}_i^{\pm 1},\mathcal{U}_j$ and $\sigma\in S_l$. By the first part of Proposition~\ref{Prop87}, it is enough to show that those actions are compatible with \eqref{defrna-1} and \eqref{prop96-1}-\eqref{prop96-4}. The relations \eqref{defrna-1} (resp. \eqref{prop96-1}) follow since $M$ has a structure of $\mathbb{C}[W^{\text{aff}}_l]$ (resp. $H^{\deg}_\lambda(S_l)$)-right module.

Fix $1\leq j, k\leq l, j\neq k$. 
We take $\mathbf{v}=\bigotimes_{1\leq d\leq l}e_{i_d}\in\mathbb{C}(m|n)^{\otimes l}$, where 
\begin{align*}
i_d&=\begin{cases}
d+3&\text{ if }d<j,d\neq k,\\
d+2&\text{ if }d>j,d\neq k,\\
2&\text{ if }d=j,\\
1&\text{ if }d=k.
\end{cases}
\end{align*} 
Since
\begin{align*}
[E^{(r)}_{m+n,1},\varpi^-_2]
&=[E^{(r)}_{m+n,1},\dfrac{1}{2}\sum_{s,g=1}^l\limits\sum_{q\neq2,3}\limits (-1)^{|q|}\text{sign}(q-2)E^{(s)}_{3,q}E^{(g)}_{q,2}]\\
&=-\dfrac{1}{2}\sum_{g=1}^l\limits E^{(r)}_{3,1}E^{(g)}_{m+n,2}-\dfrac{1}{2}\sum_{s=1}^l\limits E^{(s)}_{3,1}E^{(r)}_{m+n,2}
\end{align*}
holds by a direct computation, we have
\begin{align*}
[E^{(r)}_{m+n,1},\varpi^-_2](\mathbf{v})&=
\begin{cases}
-\dfrac{1}{2}E^{(k)}_{3,1}E^{(j)}_{m+n,2}(\mathbf{v})&\text{ if }r=j,k\\
0&\text{ if }r\neq j,k.
\end{cases}
\end{align*}
Thus, we obtain
\begin{align}
(X^-_{2,1}X^+_{0,0}-&X^+_{0,0}X^-_{2,1})(\mathbf{m}\ot\mathbf{v})\nonumber\\
&=\sum_{r=1}^l\limits\sum_{s=1}^l\limits\left(\mathbf{m}\mathsf{x}_r\mathfrak{y}_s\otimes E^{(s)}_{3,2}E^{(r)}_{m+n,1}(\mathbf{v})-\mathbf{m}\mathfrak{y}_s\mathsf{x}_r\otimes E^{(r)}_{m+n,1}E^{(s)}_{3,2}(\mathbf{v})\right)-\lambda[\varpi^-_2,X^+_0](\mathbf{m}\otimes\mathbf{v})\nonumber\\
&=\sum_{r=1}^l\limits\sum_{s=1}^l\limits\mathbf{m}[\mathsf{x}_r,\mathfrak{y}_s]\otimes E^{(s)}_{3,2}E^{(r)}_{m+n,1}(\mathbf{v})+\lambda\sum_{r=1}^l\mathbf{m}\mathsf{x}_r\otimes [E^{(r)}_{m+n,1},\varpi^-_2](\mathbf{v})\nonumber\\
&=\mathbf{m}[\mathsf{x}_k,\mathfrak{y}_j]\otimes E^{(j)}_{3,2}E^{(k)}_{m+n,1}(\mathbf{v})-\dfrac{\lambda}{2}\mathbf{m}\mathsf{x}_j\otimes E^{(k)}_{3,1}E^{(j)}_{m+n,2}(\mathbf{v})-\dfrac{\lambda}{2}\mathbf{m}\mathsf{x}_k\otimes E^{(k)}_{3,1}E^{(j)}_{m+n,2}(\mathbf{v})\nonumber\\
&=\mathbf{m}\left([\mathsf{x}_k,\mathfrak{y}_j]-\dfrac{\lambda}{2}\mathsf{x}_j\sigma_{j,k}-\dfrac{\lambda}{2}\mathsf{x}_k\sigma_{j,k}\right)\otimes E^{(j)}_{3,2}E^{(k)}_{m+n,1}(\mathbf{v}).\label{eeee1}
\end{align}
By \eqref{Eq2.3}, we find that the left hand side of \eqref{eeee1} is equal to zero. By the definition of $\mathbf{v}$, we deduce that the components of $E^{(j)}_{3,2}E^{(k)}_{m+n,1}(\mathbf{v})$ are distinct. By Lemma~\ref{chp}, we find that
\begin{equation*}
\mathbf{m}\left([\mathsf{x}_k,\mathfrak{y}_j]-\dfrac{\lambda}{2}\mathsf{x}_j\sigma_{j,k}-\dfrac{\lambda}{2}\mathsf{x}_k\sigma_{j,k}\right)=0.
\end{equation*}
Thus, the action of the generators of $\mathbb{H}_{t,c}(S_l)$ is compatible with \eqref{prop96-3}. It remains to check compatibility with \eqref{prop96-4}.

Let us take $\mathbf{w}=\bigotimes_{1\leq d\leq l}e_{k_d}\in\mathbb{C}(m|n)^{\otimes l}$, where
\begin{align*}
k_d&=\begin{cases}
d+2&\text{ if }d<k,\\
d+1&\text{ if }d>k,\\
1&\text{ if }d=k.
\end{cases}
\end{align*}
By a direct computation, we obtain
\begin{align*}
[\vartheta_{1},E^{(r)}_{m+n,1}]
&=\dfrac{m-n-1}{2}E^{(r)}_{m+n,1}+\left[\dfrac{1}{2}\sum_{s,g=1}^l\limits\sum_{q\neq 1}\limits (-1)^{|q|}E^{(s)}_{1,q}E^{(g)}_{q,1},E^{(r)}_{m+n,1}\right]\\
&\quad-\left[\dfrac{1}{2}\sum_{s,g=1}^l\limits\sum_{q\neq 2}\limits (-1)^{|q|}\text{sign}(q-2)E^{(s)}_{2,q}E^{(g)}_{q,2},E^{(r)}_{m+n,1}\right]\\
&=\dfrac{m-n-1}{2}E^{(r)}_{m+n,1}-\dfrac{1}{2}\sum_{g=1}^l\limits\sum_{q\neq 1}\limits (-1)^{|q|}E^{(r)}_{m+n,q}E^{(g)}_{q,1}+\dfrac{1}{2}\sum_{s=1}^l\limits E^{(s)}_{1,1}E^{(r)}_{m+n,1}\\
&\quad-\sum_{s=1}^l\limits E^{(s)}_{2,1}E^{(r)}_{m+n,2}-\sum_{g=1}^l\limits E^{(r)}_{2,1}E^{(g)}_{m+n,2}.
\end{align*}
Thus, we have
\begin{align*}
[E^{(r)}_{m+n,1},\vartheta_1](\mathbf{w})&=\dfrac{1}{2}\sum_{q=3}^{m+n}\limits(-1)^{|q|} E^{(r)}_{m+n,q}E^{(k)}_{q,1}(\mathbf{w})=\dfrac{1}{2}\sigma_{k,r}E^{(k)}_{m+n,1}(\mathbf{w})\text{ if }r\neq k,\\
[E^{(k)}_{m+n,1},\vartheta_1](\mathbf{w})&=0.
\end{align*}
Let us compute 
\begin{align}
(\widetilde{H}_{1,1}X^+_0-&X^+_0\widetilde{H}_{1,1})(\mathbf{m}\ot\mathbf{w})\nonumber\\
&=\sum_{r=1}^l\limits\sum_{s=1}^l\limits\mathbf{m}\mathsf{x}_r\mathfrak{y}_s\otimes h^{(s)}_{1}E^{(r)}_{m+n,1}(\mathbf{w})-\sum_{r=1}^l\limits\sum_{s=1}^l\limits\mathbf{m}\mathfrak{y}_s\mathsf{x}_r\otimes E^{(r)}_{m+n,1}h^{(s)}_{1}(\mathbf{w})-\lambda[\vartheta_1,x^+_0](\mathbf{m}\otimes\mathbf{w})\nonumber\\
&=-\mathbf{m}\mathfrak{y}_k\mathsf{x}_k\otimes E^{(k)}_{m+n,1}h^{(k)}_{1}(\mathbf{w})+\lambda\sum_{r=1}^l\limits \mathbf{m}\mathsf{x}_r\otimes[E^{(r)}_{m+n,1},\vartheta_1](\mathbf{w})\nonumber\\
&=-\mathbf{m}\mathfrak{y}_k\mathsf{x}_k\otimes E^{(k)}_{m+n,1}(\mathbf{w})+\dfrac{\lambda}{2}\sum_{\substack{r=1\\r\neq k}}^l\mathbf{m}\mathsf{x}_r\sigma_{k,r}\otimes E^{(k)}_{m+n,1}(\mathbf{w}).\label{Eq2.5-1}
\end{align}
By the definition of $\vartheta_{m+n-1}$, we have
\begin{align*}
[&\vartheta_{m+n-1},E^{(r)}_{m+n,1}]\\
&=\dfrac{1}{2}E^{(r)}_{m+n,1}+\left[(-1)^{|m+n-1|}\dfrac{1}{2}\sum_{s,g=1}^l\limits\sum_{q\neq m+n-1}\limits (-1)^{|q|}\text{sign}(q-m-n+1)E^{(s)}_{m+n-1,q}E^{(g)}_{q,m+n-1},E^{(r)}_{m+n,1}\right]\\
&\quad-\left[(-1)^{|m+n|}\dfrac{1}{2}\sum_{s,g=1}^l\limits\sum_{q\neq m+n}\limits (-1)^{|q|}\text{sign}(q-m-n)E^{(s)}_{m+n,q}E^{(g)}_{q,m+n},E^{(r)}_{m+n,1}\right]\\
&=\dfrac{1}{2}E^{(r)}_{m+n,1}+\dfrac{1}{2}\sum_{s=1}^l\limits E^{(s)}_{m+n-1,1}E^{(r)}_{m+n,m+n-1}+\dfrac{1}{2}\sum_{g=1}^l\limits E^{(r)}_{m+n-1,1}E^{(g)}_{m+n,m+n-1}\\
&\quad-\dfrac{1}{2}\sum_{s=1}^l\limits\sum_{q\neq m+n}\limits (-1)^{|q|}E^{(s)}_{m+n,q}E^{(r)}_{q,1}-\dfrac{1}{2}\sum_{g=1}^l\limits E^{(r)}_{m+n,1}E^{(g)}_{m+n,m+n}.
\end{align*}
Then, we find that, since $m+n-1 > l+1$ and by our choice of $\mathbf{w}$,
\begin{align*}
[E^{(r)}_{m+n,1},\vartheta_{m+n-1}](\mathbf{w})&=0\text{ if }r\neq k,\\
[E^{(k)}_{m+n,1},\vartheta_{m+n-1}](\mathbf{w})&=-\dfrac{1}{2}E^{(k)}_{m+n,1}(\mathbf{w})+\dfrac{1}{2}\sum_{q=1}^{m+n-1}\limits\sum_{\substack{s=1}}^l\limits (-1)^{|q|}E^{(s)}_{m+n,q}E^{(k)}_{q,1}(\mathbf{w})\\
&=\dfrac{1}{2}\sum_{\substack{s=1\\s\neq k}}^{l}\limits \sigma_{k,s}E^{(k)}_{m+n,1}(\mathbf{w})+\dfrac{m-n}{2}\mathbf{m}\otimes E^{(k)}_{m+n,1}(\mathbf{w}).
\end{align*}
This allows us to compute the following:
\begin{align}
(\widetilde{H}_{m+n-1,1}X^+_0&-X^+_0\widetilde{H}_{m+n-1,0})(\mathbf{m}\otimes\mathbf{w})\nonumber\\
&=\sum_{r,s=1}^l\big(\mathbf{m}\mathsf{x}_r\mathfrak{y}_s\otimes h^{(s)}_{m+n-1}E^{(r)}_{m+n,1}(\mathbf{w})-\mathbf{m}\mathfrak{y}_s\mathsf{x}_r\otimes E^{(r)}_{m+n,1}h^{(s)}_{m+n-1}(\mathbf{w})\big)\nonumber\\
&\quad-\lambda[\vartheta_{m+n-1},X^+_0](\mathbf{m}\otimes\mathbf{w})\nonumber\\
&=\mathbf{m}\mathsf{x}_k\mathfrak{y}_k\otimes E^{(k)}_{m+n,1}(\mathbf{w})+\dfrac{\lambda}{2}\sum_{\substack{s=1\\s\neq k}}^l\mathbf{m}\mathsf{x}_k\sigma_{k,s}\otimes E^{(k)}_{m+n,1}(\mathbf{w})+\dfrac{(m-n)\lambda}{2}\mathbf{m}\otimes E^{(k)}_{m+n,1}(\mathbf{w}).\label{Eq2.7-1}
\end{align}
By \eqref{Eq2.5} and \eqref{Eq2.7}, we find that 
\begin{align}
\eqref{Eq2.5-1}+\eqref{Eq2.7-1}&=\beta \sum_{k=0}^l\limits\mathbf{m}\mathsf{x}_k\otimes E_{m+n,1}^{(k)}(\mathbf{w}).\label{eee2}
\end{align}
By \eqref{eee2}, we find that
\begin{align*}
\mathbf{m}\left([\mathsf{x}_k,\mathfrak{y}_k]+\dfrac{\lambda}{2}\sum_{\substack{s=1\\s\neq k}}^l(\mathsf{x}_k+\mathsf{x}_s)\sigma_{s,k}+\dfrac{(m-n)\lambda}{2}\mathsf{x}_k\right)\otimes E_{m+n,1}^{(k)}(\mathbf{w})
&=\beta \sum_{k=0}^l\limits\mathbf{m}\mathsf{x}_k\otimes E_{m+n,1}^{(k)}(\mathbf{w}).
\end{align*}
By the definition of $\mathbf{w}$, we can see that the components of $E_{m+n,1}^{(k)}(\mathbf{w})$ are distinct. By Lemma~\ref{chp}, we find that
\begin{equation*}
[\mathsf{x}_k,\mathfrak{y}_k]+\dfrac{\lambda}{2}\sum_{\substack{s=1\\s\neq k}}^l(\mathsf{x}_k+\mathsf{x}_s)\sigma_{s,k}+\left(\dfrac{(m-n)\lambda}{2}-\beta\right)\mathsf{x}_k=0.
\end{equation*}
By the assumption that $c=\lambda$ and $t=\dfrac{(m-n)\lambda}{2}-\beta$, we deduce that the action is compatible with \eqref{prop96-4}.
\end{proof}

\section{Shur-Weyl duality for the Cherednik algebra and the deformed double current superalgebra}

A Schur-Weyl functor was constructed in \cite{Gu2} between deformed double current algebras and rational Cherednik algebras. In this section, we introduce new quantum superalgebras $\mc{D}_{t,\kappa}(\mathfrak{sl}_{m|n})$ of double affine type. This is done by deforming the defining relations of the universal central extension of $\mathfrak{sl}_{m|n}(\C[u,v])$ in order to produce a Schur-Weyl functor from right $\mathsf{H}_{t,c}(S_l)$-modules to left $\mc{D}_{t,\kappa}(\mathfrak{sl}_{m|n})$-modules.

\subsection{Steinberg Lie Superalgebras}\label{Steinberg}

The following theorem is obtained by restricting the main result of Iohara-Koga \cite{IK} to the case when the Lie superalgebra $\mf{g}$ is of type $A(m,n)$ and the base field is $\C$. 
			\begin{theorem} [Theorem~4.7 in \cite{IK}]\label{gira}
			Let $A$ be an associate algebra over $\C$ and let $HC_1(A)$ denote the first cyclic homology group of $A$. 
			The universal central extension $\mf{g}(A)$ of $\mf{g}\otimes_{\C} A$ is given by
			$$ \mf{g}(A) \simeq \begin{cases}
			\mf{g} \otimes A \oplus HC_1(A) & \text{if $\mf{g}$ is not of type $A(n,n)$}, \\
			\mf{sl}_{n+1|n+1} \otimes A \oplus HC_1(A)  & \text{if $\mf{g}$ is of type $A(n,n)$ for $n> 1$.}
			\end{cases} $$
			\end{theorem}

For example, $\mf{g}\otimes\C[u]$ is its own UCE (universal central extension). In the rest of this subsection, we will consider the universal central extensions when $\mf{g}= \mathfrak{sl}_{m|n}$ and $m\neq n$. 

One goal is to find a reasonable definition of the deformed double current superalgebra associated with $\mfsl_{m|n}$, given that it should be a deformation of the universal central extension of $\mfsl_{m|n} \otimes \C[u,v]$.
	 By Theorem \ref{gira}, the universal central extension of $\mfsl_{m|n} \otimes \C[u,v]$ is $\mfsl_{m|n} \otimes \C[u,v] \oplus HC_1(\C[u,v] )$.
As discussed more generally in \cite{Ka}, we have
\begin{equation*}
HC_1 (\C[u,v] ) \cong \Omega^1  (\C[u,v] )  / d(\C[u,v] ),
\end{equation*}
where $\Omega^1  (\C[u,v] ) =\C[u,v] du \oplus \C [u,v] dv$. The superbracket of  $\mfsl_{m|n} \otimes \C[u,v] \oplus HC_1(\C[u,v] )$ is given by
	 	$$[X \otimes p_1(u,v), Y \otimes p_2(u,v)]=[X,Y] \otimes p_1(u,v)p_2(u,v) + (X,Y)\overline{p_2(u,v)d  p_1(u,v)} $$
	 for $p_1(u,v), p_2(u,v) \in \C[u,v]$, $X,Y \in \mfsl_{m|n}$ and where $(\cdot,\cdot)$ is the Killing form on $\mfsl_{m|n}$, $d$ is the differential, and $\overline{\phantom{a}\cdot\phantom{a}}$ is the quotient homomorphism from $\Omega^1$ to  $\Omega^1/ d(\C[u,v] )$.
	Another description of this universal central extension is given by the Steinberg Lie superalgebra, which is defined below.

	\begin{definition}[Definition 4.1 in \cite{CSUCE}]\label{stndef}
		Let $A$ be an associative superalgebra over $\C$. For $m+n \ge 3$, the Steinberg Lie superalgebra $\mf{st}_{m|n}(A)$ is defined to be the Lie superalgebra over $\C$ generated by the homogeneous elements $F_{i,j}(a)$, $a \in A$ homogeneous, for $1 \le i \ne j \le m+n$ with $\mathrm{deg}(F_{i,j}(a))=|i|+|j|+|a|\in \Z_2$, subject to the following relations for $a,b \in A$:
			\begin{align}
			&a \mapsto F_{i,j}(a) \textrm{ is a $\C$-linear map,} \\
			& [F_{i,j}(a),  F_{j,k}(b)]=F_{i,k}(ab), \textrm{ for distinct $i$, $j$, $k$,}
			\\ & [F_{i,j}(a), F_{k,l}(b)] = 0 \textrm{ for $i\ne j\ne k \ne l \ne i$, i.e. when $[E_{i,j},E_{k,l}]=0$}.
			\end{align}
	\end{definition}

	\begin{theorem}[Proposition 4.1 and Theorem 6.1 in \cite{CSUCE}] Let $$ \varphi: \mf{st}_{m|n} (A) \to \mf{sl}_{m|n}(A)$$ be the map given by $F_{i,j}(a) \mapsto E_{i,j}(a)$. If $m+n\ge 5$, this is a universal central extension and its kernel is isomorphic to $HC_1(A)$ as $\C$-modules.
	\end{theorem}

Specifically, we will consider the case that $A=\C[u,v]$. Since all elements are even in $\C[u,v]$, we find that $\mathrm{deg}(F_{i,j}(a))=|i|+|j|$. The previous theorem provides the motivation for finding another presentation for the Steinberg Lie superalgebra. 

 				\begin{proposition}\label{242}
			Let $\widetilde{\mf{st}}_{m|n}(\C[u,v]) $ be the Lie superalgebra generated by elements $\widetilde{F}_{a,b}(1)$, $\widetilde{F}_{a,b}(u)$, and $\widetilde{F}_{a,b}(v)$ for $ 1 \le a \ne b \le m+n$ and $\mathrm{deg}(\widetilde{F}_{a,b}(1))=\mathrm{deg}(\widetilde{F}_{a,b}(u))=\mathrm{deg}(\widetilde{F}_{a,b}(v))=|a|+|b|\in\Z_2$. These generators are subject to the following relations for $a\ne b \ne c \ne d \ne a$:
				\begin{align*}
					[\widetilde{F}_{a,b}(u), \widetilde{F}_{c,d}(u)] =0 ,&&
					[\widetilde{F}_{a,b}(u), \widetilde{F}_{c,d}(v) ]=0, \\
					[\widetilde{F}_{a,b}(v), \widetilde{F}_{c,d}(v) ]=0 ,&&
					[\widetilde{F}_{a,b}(u), \widetilde{F}_{c,d}(1)] =0 ,\\
					[\widetilde{F}_{a,b}(v), \widetilde{F}_{c,d}(1) ]=0 ,&&
					[\widetilde{F}_{a,b}(1), \widetilde{F}_{c,d}(1)] =0 .\end{align*}
				 For $a,b,c$ and $a,c,d$ all distinct, 
					\begin{align*}
					[\widetilde{F}_{a,b}(u), \widetilde{F}_{b,c}(u)] = [\widetilde{F}_{a,d}(u), \widetilde{F}_{d,c}(u)] ,&&
					[\widetilde{F}_{a,b}(u), \widetilde{F}_{b,c}(v)] = [\widetilde{F}_{a,d}(v), \widetilde{F}_{d,c}(u)], \\
					[\widetilde{F}_{a,b}(v), \widetilde{F}_{b,c}(v)] = [\widetilde{F}_{a,d}(v), \widetilde{F}_{d,c}(v)] ,&&[\widetilde{F}_{a,b}(u^i), \widetilde{F}_{b,c}(u^j)] =\widetilde{F}_{a,c}(u^{i+j})\\
[\widetilde{F}_{a,b}(v^i), \widetilde{F}_{b,c}(v^j)] =\widetilde{F}_{a,c}(v^{i+j}),&&
				\end{align*}
					where $i+j=0,1$ and $m+n \ge 5$. Then, we have an isomorphism $\widetilde{\mf{st}}_{m|n}(\C[u,v]) \simeq \mf{st}_{m|n}(\C[u,v])$.	
		\end{proposition}
The rest of this subsection is devoted to the proof of Proposition~\ref{242} and we will assume that $m+n\ge 5$.

There is a surjective homomorphism of Lie superalgebras $$\widetilde{\mf{st}}_{m|n}(\C[u,v])  \twoheadrightarrow \mf{st}_{m|n}(\C[u,v])$$ given by $$ \widetilde{F}_{a,b}(X) \mapsto F_{a,b}(X)$$ for $X=1,u,v$.  To show that this is an isomorphism, we need to find an inverse map. This is equivalent to constructing elements $\widetilde{F}_{a,b}(u^r v^s)$ in $\widetilde{\mf{st}}_{m|n}(\C[u,v]) $ that could serve as the images of $F_{a,b}(u^r v^s)$ under the inverse map $$\mf{st}_{m|n}(\C[u,v]) \to \widetilde{\mf{st}}_{m|n}(\C[u,v]).  $$
It is enough to construct inductively these elements $\widetilde{F}_{a,b}(u^r v^s)$ (where $a\ne b$, $r,s\ge 0$) satisfying
 $$ [\widetilde{F}_{a,b}(u^{r_1} v^{s_1}  ), \widetilde{F}_{c,d}(u^{r_2} v^{s_2}) ]= \delta_{b,c}\widetilde{F}_{a,d}(u^{r_1+r_2} v^{s_1+s_2})$$
for $a\ne b$, $c \ne d \ne a$, $r_1,r_2,s_1,s_2 \ge 0$.

\begin{lemma}\label{raylem}
\begin{enumerate}
\item For $1\leq a\neq b\leq m+n$ and $l\geq1$, we fix  $c_{a,b}\neq a,b$ and set inductively
\begin{align}
\widetilde{F}_{a,b}(u^{l+1})&=[\widetilde{F}_{a,c_{a,b}}(u),\widetilde{F}_{c_{a,b},b}(u^{l})],\label{242-3}\\
\widetilde{F}_{a,b}(v^{l+1})&=[\widetilde{F}_{a,c_{a,b}}(v),\widetilde{F}_{c_{a,b},b}(v^{l})]\label{242-4}
\end{align}
For $d\neq a,b,c$, 
\begin{align}
[\widetilde{F}_{a,b}(u), \widetilde{F}_{c,d}(u^l)] =0\text{ if }b\neq c,\label{242-7}\\
 [\widetilde{F}_{a,b}(v), \widetilde{F}_{c,d}(v^l)] =0\text{ if }b\neq c,\label{242-8}\\
[\widetilde{F}_{a,b}(1),\widetilde{F}_{c,d}(u^{l})]=\delta_{b,c}\widetilde{F}_{a,d}(u^{l}),\label{242-5}\\
[\widetilde{F}_{a,b}(1),\widetilde{F}_{c,d}(v^{l})]=\delta_{b,c}\widetilde{F}_{a,d}(v^{l}),\label{242-6}\\
[\widetilde{F}_{a,c}(u),\widetilde{F}_{c,b}(u^{l})]=[\widetilde{F}_{a,d}(u),\widetilde{F}_{d,b}(u^{l})]\text{ if }c\neq a,b,\label{242-1}\\
[\widetilde{F}_{a,c}(v),\widetilde{F}_{c,b}(v^{l})]=[\widetilde{F}_{a,d}(v),\widetilde{F}_{d,b}(v^{l})]\text{ if }c\neq a,b.\label{242-2}
\end{align}
In particular, the definitions of $\widetilde{F}_{a,b}(u^{l})$ and $\widetilde{F}_{a,b}(v^{l})$ are independent of the choice of $c_{a,b}$.
\item
For $a \ne b$, $c \ne d \ne a$, $k,l \ge 0$,
\begin{align}
[\widetilde{F}_{a,b}(u^k), \widetilde{F}_{c,d}(u^l)] =\delta_{b,c}\widetilde{F}_{a,d}(u^{k+l}),\label{rayquaza}\\
 [\widetilde{F}_{a,b}(v^k), \widetilde{F}_{c,d}(v^l)] =\delta_{b,c}\widetilde{F}_{a,d}(v^{k+l}). \label{rayquaza-1}
\end{align}
\end{enumerate}
\end{lemma}
\begin{proof}
\begin{enumerate}
\item We only show \eqref{242-7}, \eqref{242-5} and \eqref{242-1} by induction on $l$, the relations \eqref{242-8}, \eqref{242-6} and \eqref{242-2} can be proven in a similar way.  In the case when $l=0,1$, the relations \eqref{242-7}, \eqref{242-5} and \eqref{242-1} are nothing but the defining relation of $\widetilde{\mf{st}}_{m|n}(\C[u,v])$. Suppose that they hold in the case when $l=p-1\ge 1$. Let us take $e\neq a,b,c,d$. We obtain
\begin{align}
[\widetilde{F}_{a,d}(u), \widetilde{F}_{b,c}(u^{p})] = {} & [\widetilde{F}_{a,d}(u),[\widetilde{F}_{b,e}(u), \widetilde{F}_{e,c}(u^{p-1})]] \nonumber\\
= {} &[[\widetilde{F}_{a,d}(u),\widetilde{F}_{b,e}(u)], \widetilde{F}_{e,c}(u^{p-1})] + (-1)^{(|a|+|d|)(|b|+|e|)} [\widetilde{F}_{b,e}(u) ,[\widetilde{F}_{a,d}(u) , \widetilde{F}_{e,c}(u^{p-1})]] \nonumber\\
= {} & 0,\label{242-10}
\end{align}
where the first and third equalities are due to the induction hypothesis. Similarly, we have
\begin{align}
[\widetilde{F}_{a,b}(1), \widetilde{F}_{c,d}(u^{p})] &= [\widetilde{F}_{a,b}(1), [\widetilde{F}_{c,e}(u), \widetilde{F}_{e,d}(u^{p-1})]]\nonumber\\
&=[ [\widetilde{F}_{a,b}(1), \widetilde{F}_{c,e}(u)], \widetilde{F}_{e,d}(u^{p-1})] + (-1)^{(|a|+|b|)(|c|+|e|)}[\widetilde{F}_{c,e}(u), [\widetilde{F}_{a,b}(1), \widetilde{F}_{e,d}(u^{p-1})]]\nonumber\\
&=\delta_{b,c} [\widetilde{F}_{a,e}(u), \widetilde{F}_{e,d}(u^{p-1})]=\delta_{b,c} \widetilde{F}_{a,d}(u^{p}),\label{242-11}
\end{align}
where the first and third equalities are due to the induction hypothesis. We also obtain
\begin{align}
[\widetilde{F}_{a,b}(u), \widetilde{F}_{b,c}(u^p)] = {} & [[\widetilde{F}_{a,d}(u), \widetilde{F}_{d,b}(1)],\widetilde{F}_{b,c}(u^p)]\nonumber\\
= {} & [\widetilde{F}_{a,d}(u), [ \widetilde{F}_{d,b}(1), \widetilde{F}_{b,c}(u^p)]] - (-1)^{(|a|+|d|)(|d|+|b|)} [ \widetilde{F}_{d,b}(1), [ \widetilde{F}_{a,d}(u), \widetilde{F}_{b,c}(u^p)]]\nonumber\\
= {} & [\widetilde{F}_{a,d}(u), \widetilde{F}_{d,c}(u^p)],\label{242-12}
\end{align}
where the third equality is due to \eqref{242-10} and \eqref{242-11}. By induction, we find that \eqref{242-7}, \eqref{242-5} and \eqref{242-1} hold.

\item We prove \eqref{rayquaza} by the induction on $k$, the cases when $k=0,1$ being \eqref{242-7}, \eqref{242-5} and \eqref{242-1}. Suppose that \eqref{rayquaza} holds when $k=p-1$. Let us take $e\neq a,b,c,d$. We obtain 
\begin{align*}
[\widetilde{F}_{a,b}(u^p),\widetilde{F}_{c,d}(u^l)]&= [[\widetilde{F}_{a,e}(u),\widetilde{F}_{e,b}(u^{p-1})],\widetilde{F}_{c,d}(u^l)] \\
&=[\widetilde{F}_{a,e}(u),[ \widetilde{F}_{e,b}(u^{p-1}), \widetilde{F}_{c,d}(u^l)]] - (-1)^{(|a|+|e|)(|e|+|b|)}[\widetilde{F}_{e,b}(u^{p-1}), [\widetilde{F}_{a,e}(u), \widetilde{F}_{c,d}(u^l)]]\\
&= \delta_{b,c}[\widetilde{F}_{a,e}(u), \widetilde{F}_{e,d}(u^{k+l-1})]=\delta_{b,c}\widetilde{F}_{a,d}(u^{k+l}),
			\end{align*}
where the first and fourth equalities are due to \eqref{242-1} and the second equality is due to \eqref{242-7}, \eqref{242-1} and the induction hypothesis. Thus, we find that (\ref{rayquaza}) holds when $k=p$. By induction, \eqref{rayquaza} holds for any $k,l\ge 0$. \end{enumerate} \end{proof}
Finally, we introduce elements $\widetilde{F}_{a,b}(u^r v^s)$ while proving Proposition \ref{dopey} below.
\begin{proposition} \label{dopey}
\begin{enumerate}
\item For $1\leq a\neq b\leq m+n$ and $l\geq1$. We choose $c_{a,b}\neq a,b$ and set
\begin{align}
\widetilde{F}_{a,b}(u^{l+1}v^k)&=[\widetilde{F}_{a,c_{a,b}}(u),\widetilde{F}_{c_{a,b},b}(u^{l}v^k)]\label{243-3}
\end{align}
inductively.  Then, we have
\begin{align}
[\widetilde{F}_{a,b}(u), \widetilde{F}_{c,d}(u^lv^k)] & = 0 \text{ if }b\neq c\text{ and }d\neq a,c,\label{243-7}\\
[\widetilde{F}_{a,b}(1),\widetilde{F}_{c,d}(u^{l}v^k)] & = \delta_{b,c}\widetilde{F}_{a,d}(u^{l}v^k)\text{ if }d\neq a,c,\label{243-5}\\
[\widetilde{F}_{a,c}(u),\widetilde{F}_{c,b}(u^{l}v^k)] & = [\widetilde{F}_{a,d}(u),\widetilde{F}_{d,b}(u^{l}v^k)]\text{ if }c\neq a,b\text{ and }d\neq a,b.\label{243-2}
\end{align}
In particular, the definition of $\widetilde{F}_{a,b}(u^{l}v^k)$ is independent of the choice of $c_{a,b}$.
\item For $a\ne b$, $c \ne d \ne a$, 
\begin{equation}
[\widetilde{F}_{a,b}(u^{k_1}), \widetilde{F}_{c,d} (u^{k_2} v^{l_2}) ]= \delta_{b,c} \widetilde{F}_{a,d} (u^{k_1+k_2} v^{l_2}).\label{gazz1}
\end{equation}
\item For $a\ne b$, $c \ne d \ne a$, 
\begin{align} [\widetilde{F}_{a,b}(v), \widetilde{F}_{c,d} (u^{k_2} v^{l_2}) ]= \delta_{b,c} \widetilde{F}_{a,d} (u^{k_2} v^{1+l_2}). 
	\label{grazz2}
	\end{align}
\item For $a\ne b$, $c \ne d \ne a$, and $k_1,k_2,l_1,l_2 \ge0$, 
	\begin{align} [\widetilde{F}_{a,b}(u^{k_1} v^{l_1}), \widetilde{F}_{c,d} (u^{k_2} v^{l_2}) ]= \delta_{b,c} \widetilde{F}_{a,d} (u^{k_1+k_2} v^{l_1+l_2}). 
	\label{grazz}
	\end{align}
\end{enumerate}
\end{proposition}
\begin{proof}
\begin{enumerate}
\item The first part can be proved as \eqref{242-7}, \eqref{242-5} and \eqref{242-1} above. 

\item This can be proved as \eqref{rayquaza}.

\item We prove by the induction on $k_2$. The case when $k_2=0$ is nothing but \eqref{242-2}. Suppose that \eqref{grazz2} holds when $k_2=p-1$. Let us take $e\neq a,b,c,d$. We obtain
\begin{align*}
[\widetilde{F}_{a,b}(v), & \widetilde{F}_{c,d} (u^{p} v^{l_2}) ] =[\widetilde{F}_{a,b}(v), [\widetilde{F}_{c,e}(u),\widetilde{F}_{e,d} (u^{p-1} v^{l_2}) ]]\\
&=[[\widetilde{F}_{a,b}(v), \widetilde{F}_{c,e}(u)],\widetilde{F}_{e,d} (u^{p-1} v^{l_2}) ]+(-1)^{(|a|+|e|)(|e|+|b|)}[\widetilde{F}_{c,e}(u),[\widetilde{F}_{a,b}(v), \widetilde{F}_{e,d} (u^{p-1} v^{l_2}) ]]\\
&=[[\widetilde{F}_{a,b}(u), \widetilde{F}_{c,e}(v)],\widetilde{F}_{e,d} (u^{p-1} v^{l_2}) ]\\
&=[\widetilde{F}_{a,b}(u), [\widetilde{F}_{c,e}(v),\widetilde{F}_{e,d} (u^{p-1} v^{l_2}) ]]-(-1)^{(|a|+|e|)(|e|+|b|)}[\widetilde{F}_{c,e}(v),[\widetilde{F}_{a,b}(u), \widetilde{F}_{e,d} (u^{p-1} v^{l_2}) ]]\\
&=[\widetilde{F}_{a,b}(u), \widetilde{F}_{c,d} (u^{p-1} v^{l_2+1}) ] =\delta_{b,c}\widetilde{F}_{a,d} (u^{p} v^{l_2+1}),
\end{align*}
where the first and last equality are due to part 1 above and the third equality is due to the definition of $\widetilde{\mf{st}}_{m|n}(\C[u,v])$ and the induction hypothesis. Then, by induction, we obtain \eqref{grazz2}.

\item The proof proceeds by induction on $l_1+l_2$. The case when $l_1+l_2=0$ is nothing but Lemma~\ref{raylem}. Suppose that \eqref{grazz} holds in the case that $l_1+l_2=p$.
Let us consider the case that $l_1 \ge 1$ and $l_1+l_2=p+1$. We obtain
	\begin{align*} 
		[&\widetilde{F}_{a,b}(u^{k_1}v^{l_1}), \widetilde{F}_{c,d}(u^{k_2} v^{l_2})] = [ [ \widetilde{F}_{a,e}(v), \widetilde{F}_{e,b}(u^{k_1} v^{l_1 -1})],\widetilde{F}_{c,d}(u^{k_2} v^{l_2})]\\
&=[\widetilde{F}_{a,e}(v), [ \widetilde{F}_{e,b}(u^{k_1} v^{l_1-1}), \widetilde{F}_{c,d}(u^{k_2} v^{l_2})]- (-1)^{(|a|+|e|)(|e|+|b|)} [ \widetilde{F}_{e,b}(u^{k-1} v^{l_1-1}),  [\widetilde{F}_{a,e}(v), \widetilde{F}_{c,d}(u^{k_2} v^{l_2})] \\
& =\delta_{b,c} [\widetilde{F}_{a,e}(v), \widetilde{F}_{e,d}(u^{k_1+k_2} v^{l_1+l_2-1})] = \delta_{b,c} \widetilde{F}_{a,d} (u^{k_1+k_2} v^{l_1+l_2}),
	\end{align*}
where the first and last equalities are due to \eqref{grazz2} and the third equalites and the third equality is due to \eqref{grazz2} and the induction hypothesis. Then, by induction, we obtain \eqref{grazz}.
\end{enumerate}
\end{proof}
\subsection{Deformed Double Current Superalgebras}
\begin{definition} \label{moneypowerglory}
Let $t,\kappa$ be complex numbers. The deformed double current superalgebra $\mc{D}_{t,\kappa}(\mathfrak{sl}_{m|n})$ is the associative superalgebra over $\C$ generated by 
\begin{gather*}
\{E_{a,b},\msk(E_{a,b}), \msq(E_{a,b})\mid 1\leq a\neq b\leq m+n\}
\end{gather*}
with the following commutator relations:
\begin{align}
[E_{a,b},E_{c,d}] & =\delta_{b,c}E_{a,d}-(-1)^{(|a|+|b|)(|c|+|d|)}\delta_{a,d}E_{c,b},\label{def-1}\\
[E_{a,b},\msk(E_{c,d})] & =\delta_{b,c}\msk(E_{a,d})\text{ for }a\neq d,\label{def-2}\\
[\msk(E_{a,b}),\msk(E_{c,d})] & =0\text{ for }a\neq d,b\neq c,\label{def-3}\\
[\msk(E_{a,c}),\msk(E_{c,d})] & =[\msk(E_{a,b}),\msk(E_{b,d})],\label{def-4}\\
[E_{a,b},\msq(E_{c,d})] & =\delta_{b,c}\msq(E_{a,d})\text{ for }a\neq d,\label{def-5}\\
[\msq(E_{a,b}),\msq(E_{c,d})] & =0\text{ for }a\neq d,b\neq c,\label{def-6}\\
[\msq(E_{a,c}),\msq(E_{c,d})] & =[\msq(E_{a,b}),\msq(E_{b,d})],\label{def-7}
\end{align}
\begin{align}
[\msk(E_{a,b}),\msq(E_{b,c})]&-[\msq(E_{a,d}),\msk(E_{d,c})]\nonumber\\ 
&= tE_{a,c} -\kappa((-1)^{|b|}E_{b,b}+(-1)^{|d|}E_{d,d})E_{a,c}\nonumber\\
&\quad+\kappa\sum_{e=1}^{m+n}\limits (-1)^{|e|+(|e|+|c|)(|a|+|e|)}  E_{e,c} E_{a,e} \text{ for distinct }a,b,c\text{ and distinct }a,c,d,\label{def-8}
\end{align} 
\begin{align}
[\msk(E_{a,b}), \msq(E_{c,d})]& =- \kappa (-1)^{ |a||b| + |a||c|+|b||c|} E_{c,b} E_{a,d} \text{ when } a\neq b \neq c \neq d \neq a.\label{def-9}
\end{align} 
\end{definition}
By Proposition~\ref{242}, $\mc{D}_{t,\kappa}(\mathfrak{sl}_{m|n})$ is 
a deformation of the Steinberg Lie superalgebra $\mf{st}_{m|n}(\C[u,v])$: $$\mc{D}_{t=0,\kappa=0}(\mathfrak{sl}_{m|n}) \cong \mf{st}_{m|n}(\C[u,v]) \text{ via } E_{a,b}\longleftrightarrow F_{a,b}(1), \;\msk(E_{a,b}) \longleftrightarrow F_{a,b}(u), \;\msq(E_{a,b}) \longleftrightarrow F_{a,b}(v).$$
\begin{remark}
We note that the right-hand side of \eqref{def-8} contains $E_{b,b}$ and $E_{d,d}$. However, considering separately the cases $e=a$ and $e=c$ in the large summation in \eqref{def-8}, we can rewrite it as:
\begin{align*}
\kappa((-1)^{|c|}E_{c,c}-(-1)^{|b|}E_{b,b})E_{a,c}+\kappa E_{a,c}((-1)^{|a|}E_{a,a}&-(-1)^{|d|}E_{d,d})\\
&+\kappa\sum_{\substack{1\leq e\leq m+n\\e\neq a,c}}\limits (-1)^{|e|+(|a|+|e|)(|c|+|e|)}E_{e,c}E_{a,e}+tE_{a,c}.
\end{align*}
Thus, the right hand side of \eqref{def-8} is also contained in $U(\mathfrak{sl}_{m|n})$.
\end{remark}
\begin{proposition}
Let $M$ be a right module over the rational Cherednik algebra $\mathsf{H}_{t,\kappa}(S_l)$. Then, $M\otimes_{\C[S_l]} \C(m|n)^{\otimes l}$ is a left module over $\mc{D}_{t,\kappa}(\mathfrak{sl}_{m|n})$ which we denote $SW(M)$. Consequently, we obtain a functor $SW:M\mapsto M\otimes_{\C[S_l]} \C(m|n)^{\otimes l}$ and $f\mapsto f\otimes 1^{\otimes l}$ for any right $\mathsf{H}_{t,\kappa}(S_{\ell})$-module $M$ and any homomorphism of right $\mathsf{H}_{t,\kappa}(S_{\ell})$-modules $f$. 
\end{proposition}
\begin{proof} Any homomorphism $f$ of right $\mathsf{H}_{t,\kappa}(S_l)$-modules $M_1\rightarrow M_2$ is also a homomorphism of right $\C[S_l]$-modules, hence induces a linear map $f\otimes 1^{\otimes l}: SW(M_1)\rightarrow SW(M_2)$ which is also a homomorphism of left $\mc{D}_{t,\kappa}(\mathfrak{sl}_{m|n})$-modules if we let $\mc{D}_{t,\kappa}(\mathfrak{sl}_{m|n})$ act on $M\otimes_{\C[S_l]}  \C(m|n)^{\otimes l}$ as follows:
$$ \msk(E_{a,b})(\mathbf{m}\otimes\underline{\mathbf{v}}) =  \sum_{k=1}^l \mathbf{m}\mathsf{x}_k \otimes E_{a,b}^{(k)} (\underline{\mathbf{v}}), \  \msq(E_{a,b})(\mathbf{m}\otimes\underline{\mathbf{v}}) =  \sum_{k=1}^l \mathbf{m}\mathsf{y}_k \otimes E_{a,b}^{(k)} (\underline{\mathbf{v}}),$$
for any $\mathbf{m}\in M$ and $\underline{\mathbf{v}}\in\C(m|n)^{\otimes l}$.
It is enough to show the compatibility with \eqref{def-1}-\eqref{def-9}. The compatibility with \eqref{def-1}-\eqref{def-7} follows from the definition of the action of $\mc{D}_{t,\kappa}(\mathfrak{sl}_{m|n})$, so we will focus on \eqref{def-8} and \eqref{def-9}.

We need to compute $[\msk(E_{a,b}), \msq(E_{c,d})] (\mathbf{m}\otimes \underline{\mathbf{v}} )$ for any choice of distinct indices: $a\neq b$ and $c\neq d$. By the definition of the action of $\mc{D}_{t,\kappa}(\mathfrak{sl}_{m|n})$, we have
\begin{align} 
[\msk(E_{a,b}), \msq(E_{c,d})] (\mathbf{m}\otimes \underline{\mathbf{v}} )=\sum_{k,r=1}^l\limits \mathbf{m}\mathsf{y}_r \mathsf{x}_k \otimes E_{a,b}^{(k)} E_{c,d}^{(r)} (\underline{\mathbf{v}}) -(-1)^{(|a|+|b|)(|c|+|d|)} \sum_{k,r=1}^l\limits \mathbf{m}\mathsf{x}_k  \mathsf{y}_r \otimes E_{c,d}^{(r)} E_{a,b}^{(k)}  (\underline{\mathbf{v}})\label{1004}.
\end{align}
By separating the right hand side of \eqref{1004} into the two cases $k=r$ and $k\neq r$, we obtain:
\begin{align}
[\msk(E_{a,b}), \msq(E_{c,d})] (\mathbf{m}\otimes \underline{\mathbf{v}} ) &=\sum_{\substack{1\leq k,r\leq l\\k\neq r}}\limits\mathbf{m}\mathsf{y}_r \mathsf{x}_k \otimes E_{a,b}^{(k)} E_{c,d}^{(r)} (\underline{\mathbf{v}}) -\sum_{\substack{1\leq k,r\leq l\\k\neq r}}\limits\mathbf{m}\mathsf{x}_k  \mathsf{y}_r \otimes E_{a,b}^{(k)}E_{c,d}^{(r)}   (\underline{\mathbf{v}})\nonumber\\
&\quad+\sum_{k=1}^l\limits \mathbf{m}\mathsf{y}_k \mathsf{x}_k \otimes E_{a,b}^{(k)} E_{c,d}^{(k)} (\underline{\mathbf{v}}) -(-1)^{(|a|+|b|)(|c|+|d|)} \sum_{k=1}^l\limits \mathbf{m}\mathsf{x}_k  \mathsf{y}_k \otimes E_{c,d}^{(k)} E_{a,b}^{(k)}  (\underline{\mathbf{v}})\nonumber\\
&=\sum_{\substack{1\leq k,r\leq l\\k\neq r}}\limits \mathbf{m}[\mathsf{y}_r,\mathsf{x}_k] \otimes E_{a,b}^{(k)}E_{c,d}^{(r)} (\underline{\mathbf{v}})+\delta_{b,c}\sum_{k=1}^l\limits \mathbf{m}\mathsf{y}_k \mathsf{x}_k \otimes E_{a,d}^{(k)} (\underline{\mathbf{v}}) \nonumber \\
& \quad-\delta_{a,d}(-1)^{(|a|+|b|)(|c|+|d|)} \sum_{k=1}^l\limits \mathbf{m}\mathsf{x}_k  \mathsf{y}_k \otimes E_{c,b}^{(k)}  (\underline{\mathbf{v}}).
\label{1004-1}
\end{align}
By \eqref{defrna-3}, we can rewrite the first term of the right hand side of \eqref{1004-1} as: 
\begin{align}
\sum_{\substack{1\leq k,r\leq l\\k\neq r}}\limits \mathbf{m}[\mathsf{y}_r,\mathsf{x}_k] \otimes E_{a,b}^{(k)}E_{c,d}^{(r)} (\underline{\mathbf{v}})&=-\kappa\sum_{\substack{1\leq k,r\leq l\\k\neq r}}\limits \mathbf{m}\sigma_{k,r}\otimes E_{a,b}^{(k)}E_{c,d}^{(r)}(\underline{\mathbf{v}})\nonumber\\
&=-(-1)^{|a||b|+|b||c|+|c||a|}\kappa\sum_{\substack{1\leq k,r\leq l\\k\neq r}}\limits \mathbf{m}\otimes E_{c,b}^{(k)}E_{a,d}^{(r)}(\underline{\mathbf{v}}).\label{1004-2}
\end{align}
\begin{comment}
By applying \eqref{1004-2} to \eqref{1004-1}, we obtain
\begin{align}
&\quad[\msk(E_{a,b}), \msq(E_{c,d})] (\mathbf{m}\otimes \underline{\mathbf{v}} )\nonumber\\
&=-(-1)^{|a||b|+|b||c|+|c||a|}\kappa\sum_{\substack{1\leq k,r\leq l\\k\neq r}}\limits \mathbf{m}\otimes E_{c,b}^{(k)}E_{a,d}^{(r)}(\underline{\mathbf{v}})\nonumber\\
&\quad+\delta_{b,c}\sum_{k=1}^l\limits \mathbf{m}\mathsf{y}_k \mathsf{x}_k \otimes E_{a,d}^{(k)} (\underline{\mathbf{v}}) -\delta_{a,d}(-1)^{(|a|+|b|)(|c|+|d|)} \sum_{k=1}^l\limits \mathbf{m}\mathsf{x}_k  \mathsf{y}_k \otimes E_{c,b}^{(k)}  (\underline{\mathbf{v}}).\label{1005}
\end{align}
\end{comment}

Now, we prove the compatibility with \eqref{def-9}. If $a\neq b \neq c \neq d \neq a$ then, by \eqref{1004-1} and \eqref{1004-2}, we obtain:
\begin{align*}
[\msk(E_{a,b}), \msq(E_{c,d})] (\mathbf{m}\otimes \underline{\mathbf{v}} )
&=-\kappa\sum_{\substack{1\leq k,r\leq l\\k\neq r}}\limits  \mathbf{m} \otimes  (-1)^{|a||b|+|b||c|+|c||a|}E^{(k)}_{c,b}E^{(r)}_{a,d}(\underline{\mathbf{v}})\nonumber\\
&=-\kappa(-1)^{|a||b|+|b||c|+|c||a|}E_{c,b}E_{a,d}(\mathbf{m}\otimes \underline{\mathbf{v}}),
\end{align*}
where the last equality is due to $b\neq a$.

Next, we show the compatibility with \eqref{def-8}. For distinct indices $a,b,c$ and $a,c,d$, by \eqref{1004-1} and \eqref{1004-2}, we have
\begin{align}
[\msk(E_{a,b}), \msq(E_{b,c})](\mathbf{m}\otimes \underline{\mathbf{v}} )&=-(-1)^{|b|}\kappa\sum_{\substack{1\leq k,r\leq l,\\k\neq r}} \mathbf{m} \otimes  E^{(k)}_{b,b}E^{(r)}_{a,c}(\underline{\mathbf{v}})+\sum_{k=1}^l \mathbf{m}\mathsf{y}_k \mathsf{x}_k \otimes  E^{(k)}_{a,c}(\underline{\mathbf{v}}),\label{1006-1}\\
[\msq(E_{a,d}), \msk(E_{d,c})](\mathbf{m}\otimes \underline{\mathbf{v}} )&=-(-1)^{|d|}\kappa\sum_{\substack{1\leq k,r\leq l,\\k\neq r}} \mathbf{m} \otimes E^{(k)}_{a,c}E^{(r)}_{d,d}(\underline{\mathbf{v}})-\sum_{k=1}^l \mathbf{m}\mathsf{x}_k  \mathsf{y}_k \otimes E^{(k)}_{a,c}(\underline{\mathbf{v}}).\label{1006-2}
\end{align}
By \eqref{defrna-4}, we can rewrite the difference between the second terms of \eqref{1006-1} and \eqref{1006-2} as
\begin{align}
\sum_{k=1}^l \mathbf{m}(\mathsf{y}_k\mathsf{x}_k-\mathsf{x}_k\mathsf{y}_k) \otimes E^{(k)}_{a,c}(\underline{\mathbf{v}})
&=t\sum_{k=1}^l \mathbf{m}\otimes E^{(k)}_{a,c}(\underline{\mathbf{v}})+\kappa\sum_{\substack{1\leq r,k\leq l\\r\neq k}} \mathbf{m}\sigma_{r,k} \otimes E^{(k)}_{a,c}(\underline{\mathbf{v}})\nonumber\\
&=t\sum_{k=1}^l \mathbf{m} \otimes E^{(k)}_{a,c}(\underline{\mathbf{v}})+\kappa\sum_{\substack{1\leq r,k\leq l\\r\neq k}}\sum_{e=1}^{m+n}\mathbf{m} \otimes (-1)^{|e|}E^{(k)}_{a,e}E^{(r)}_{e,c}(\underline{\mathbf{v}}).\label{1006.1}
\end{align}
Applying \eqref{1006.1} to $\eqref{1006-1}+\eqref{1006-2}$, we obtain:
\begin{align}
([\msk(E_{a,b}), \msq(E_{b,c})] -& [\msq(E_{a,d}), \msk(E_{d,c})]) (\mathbf{m}\otimes \underline{\mathbf{v}} ) \nonumber\\
%&=-\kappa(-1)^{|b|} \sum_{\substack{1\leq r,k\leq l\\r\neq k}}\mathbf{m} \otimes  E^{(k)}_{b,b}E^{(r)}_{a,c} (\underline{\mathbf{v}})-(-1)^{|d|}\kappa\sum_{\substack{1\leq r,k\leq l\\r\neq k}}\mathbf{m} \otimes E^{(k)}_{a,c}E^{(r)}_{d,d}(\underline{\mathbf{v}})\nonumber\\
%&\quad+t\sum_{k=1}^l \mathbf{m}\otimes E^{(k)}_{a,c}(\underline{\mathbf{v}})+ \kappa \sum_{\substack{1\leq r,k\leq l\\r\neq k}}\sum_{e=1}^{m+n}\mathbf{m} \otimes (-1)^{|e|}E^{(k)}_{a,e}E^{(r)}_{e,c}(\underline{\mathbf{v}})\nonumber\\
& =-\kappa  (-1)^{|b|}E_{b,b}E_{a,c}(\mathbf{m} \otimes\underline{\mathbf{v}})-\kappa  (-1)^{|d|}E_{d,d}E_{a,c}(\mathbf{m} \otimes\underline{\mathbf{v}}) \nonumber \\
&\quad+tE_{a,c}(\mathbf{m}\otimes\underline{\mathbf{v}})+\kappa\sum_{e=1}^{m+n}(-1)^{|e|+(|e|+|c|)(|a|+|e|)}  E_{e,c} E_{a,e}(\mathbf{m} \otimes \underline{\mathbf{v}}),
\end{align}
because of the assumption that $a,b,c$ are distinct and so are $a,c,d$. We have thus proven the compatibility with \eqref{def-8}.
\end{proof}
\begin{definition}
We say that $N$ is integrable of level $l$ as a left $\mc{D}_{t,\kappa}(\mathfrak{sl}_{m|n})$-module if $N$ is integrable of level $l$ as a $\mathfrak{sl}_{m|n}$-module.
\end{definition}
\begin{theorem}\label{syth}
    The functor $SW$ provides an equivalence between the category of right $\mathsf{H}_{t,\kappa}(S_l)$-modules and the category of integrable left $\mc{D}_{t,\kappa}(\mathfrak{sl}_{m|n})$-modules of level $l$ if $m,n>l$.
\end{theorem} 
\begin{proof}
Take $m,n>l$. As in the proof of Theorem \ref{reverse}, it is enough to show that $SW$ is essentially surjective. Let $N$ be an integrable left-module over the deformed double current superalgebra $\mc{D}_{t,\kappa}(\mathfrak{sl}_{m|n})$ of level $l$. 
In a similar way to the proof of Theorem~\ref{floridakilos-1}, by Theorem~\ref{chengbook}, we find that there exists a right $S_l$-module $M$ satisfying that $N\cong M\otimes_{\C[S_l]}\C(m|n)^{\otimes l}$. 

We need to determine what should be the actions of $\mathsf{x}_k$ and of $\mathsf{y}_k$ on $M$. Since there exists a homomorphism $U(\mathfrak{sl}_{m|n}\otimes\C[u]) \rightarrow \mc{D}_{t,\kappa}(\mathfrak{sl}_{m|n})$ sending $E_{a,b} \otimes u \mapsto \msk(E_{a,b})$, we can regard $N$ as an integrable module over $U(\mathfrak{sl}_{m|n}\otimes\C[u])$ of level $l$. By Theorem~\ref{Flicker}, $M$ can be made into a right module over $\C[\mathsf{x}_1,\dots,\mathsf{x}_l] \rtimes \C[S_l]$. The analogous result holds replacing $u$ with $v$  and instead sending $E_{a,b} \otimes v \mapsto \msq(E_{a,b})$, hence $M$ can also be viewed as a right module over $\C[\mathsf{y}_1,\dots,\mathsf{y}_l] \rtimes \C[S_l]$.  It is enough to show that these actions are compatible with the defining relations \eqref{defrna-3}-\eqref{defrna-4}. 

Let us set  $\underline{\mathbf{v}}=\bigotimes_{1\leq d\leq l}\limits e_d$. 
By \eqref{def-9} and the assumption that $m,n>l$, we have
\begin{align}
[\msk(E_{l+1,i}),\msq(E_{l+2,j})](\mathbf{m}\otimes\underline{\mathbf{v}})&= -\kappa (-1)^{|l+1||i|+|l+1||l+2|+|i||l+2|} E_{l+2,i}E_{l+1,j}(\mathbf{m}\otimes\underline{\mathbf{v}})\label{1011}
\end{align}
for $\mathbf{m}\in M$. By the definition of the action of $\mathsf{x}_k$ and $\mathsf{y}_k$, we can rewrite the left-hand side of \eqref{1011} as follows:
\begin{align}
[\msk(E_{l+1,i}),\msq(E_{l+2,j})]&(\mathbf{m}\otimes\underline{\mathbf{v}}) \nonumber\\
&=\sum_{k,r=1}^l \mathbf{m}\mathsf{y}_r \mathsf{x}_k \otimes E_{l+1,i}^{(k)} E_{l+2,j}^{(r)} (\underline{\mathbf{v}}) - (-1)^{(|l+1|+|i|)(|l+2|+|j|)}  \mathbf{m}\mathsf{x}_k \mathsf{y}_r \otimes E_{l+2,j}^{(r)} E_{l+1,i}^{(k)} (\underline{\mathbf{v}}) \nonumber\\
&=\mathbf{m}\left( \mathsf{y}_j \mathsf{x}_i-\mathsf{x}_i\mathsf{y}_j\right)\otimes E_{l+1,i}^{(i)} E_{l+2,j}^{(j)} (\underline{\mathbf{v}}),\label{1012}
\end{align}
where the last equality is due to the definition of $\underline{\textbf{v}}$. Similarly, by the definition of $\underline{\textbf{v}}$, we also have
\begin{align}
(-1)^{|l+1||i|+|l+1||l+2|+|i||l+2|} E_{l+2,i}E_{l+1,j}(\mathbf{m}\otimes\underline{\mathbf{v}})&= (-1)^{|l+1||i|+|l+1||l+2|+|i||l+2|} \mathbf{m}\otimes E_{l+2,i}^{(i)}E_{l+1,j}^{(j)}(\underline{\mathbf{v}})\nonumber\\
&=\mathbf{m}\otimes \sigma_{i,j}E_{l+1,i}^{(i)} E_{l+2,j}^{(j)} (\underline{\mathbf{v}})=\mathbf{m}\sigma_{i,j}\otimes E_{l+1,i}^{(i)} E_{l+2,j}^{(j)} (\underline{\mathbf{v}}),\label{1013}
\end{align}
where the second equality is due to $\sigma_{i,j}=\sum_{a,b=1}^{m+n}\limits(-1)^{|b|}E^{(i)}_{a,b}E^{(j)}_{b,a}$. Combining \eqref{1012} and \eqref{1013} with \eqref{1011}, we obtain
\begin{align}
\mathbf{m}\left( \mathsf{y}_j \mathsf{x}_i-\mathsf{x}_i\mathsf{y}_j\right)\otimes E_{l+1,i}^{(i)} E_{l+2,j}^{(j)} (\underline{\mathbf{v}})&=-\kappa\mathbf{m}\sigma_{i,j}\otimes E_{l+1,i}^{(i)} E_{l+2,j}^{(j)} (\underline{\mathbf{v}}).\label{fallenfruit}
\end{align}
Since the entries of $\underline{\mathbf{v}}$ are distinct, by Lemma \ref{chp}, \eqref{fallenfruit} implies that $\mathbf{m}(\mathsf{y}_j \mathsf{x}_i-\mathsf{x}_i \mathsf{y}_j+\kappa \sigma_{i,j} )=0$. We have proven the compatibility with \eqref{defrna-3}. 

We are left to show the compatibility with \eqref{defrna-4}. Let us set
\begin{align*}
\underline{\mathbf{v}}'&=\left(\bigotimes_{2\leq g\leq i}e_g\right)\otimes e_{m+n-1}\otimes\left(\bigotimes_{i+2\leq g\leq l+1}e_g\right).
\end{align*}
By \eqref{def-8} and the definition of $\underline{\mathbf{v}}'$, we have
\begin{align}
&\quad([\msk(E_{m+n,1}), \msq(E_{1,m+n-1})]-[\msq(E_{m+n,1}), \msk(E_{1,m+n-1})]) (\mathbf{m}\otimes \underline{\mathbf{v}}' )\nonumber\\
&=tE_{m+n,m+n-1}(\mathbf{m}\otimes \underline{\mathbf{v}}' )-\kappa((-1)^{|1|}E_{1,1}E_{m+n,m+n-1})(\mathbf{m}\otimes \underline{\mathbf{v}}' )-\kappa((-1)^{|1|}E_{m+n,m+n-1}E_{1,1})(\mathbf{m}\otimes \underline{\mathbf{v}}' )\nonumber\\
&\quad+\kappa\sum_{e=1}^{m+n}\limits (-1)^{|e|+(|m+n-1|+|e|)(|m+n|+|e|)}E_{e,m+n-1}E_{m+n,e}(\mathbf{m}\otimes \underline{\mathbf{v}}' )\nonumber\\
&= t\mathbf{m}\otimes E^{(k)}_{m+n,m+n-1}(\underline{\mathbf{v}}')+\kappa\mathbf{m}\otimes \sum_{\substack{1\leq r,k\leq l\\r\neq k}}\sum_{e=1}^{m+n}\sigma_{r,k}E_{e,e}^{(r)}E_{m+n,m+n-1}^{(k)}(\underline{\mathbf{v}}')\nonumber\\
&= t\mathbf{m}\otimes E_{m+n,m+n-1}^{(i)}(\underline{\mathbf{v}}')+\kappa\sum_{\substack{r=1\\r\neq i}}^l\mathbf{m}\otimes\sigma_{r,i}E_{m+n,m+n-1}^{(i)}(\underline{\mathbf{v}}').\label{env0}
\end{align}
By the definition of $\underline{\mathbf{v}}'$ and the action of $\mathsf{x}_k$ and $\mathsf{y}_k$, we can rewrite 
\begin{align}
[\msk(E_{m+n,1}), \msq(E_{1,m+n-1})](\mathbf{m}\otimes \underline{\mathbf{v}}' )&=\mathbf{m}\mathsf{y}_i\mathsf{x}_i \otimes E^{(i)}_{m+n,m+n-1}(\underline{\mathbf{v}'}),\label{env1-1}\\
[\msq(E_{m+n,1}), \msk(E_{1,m+n-1})] (\mathbf{m}\otimes \underline{\mathbf{v}}' )&=\mathbf{m}\mathsf{x}_i\mathsf{y}_i \otimes E^{(i)}_{m+n,m+n-1}(\underline{\mathbf{v}'}).\label{env1-2}
\end{align}
\begin{comment}
the left hand side of \eqref{env0} as
\begin{align}
\text{the left hand side of \eqref{env0}}
&=\sum_{k=1}^l \mathbf{m}(\mathsf{y}_k\mathsf{x}_k-\mathsf{x}_k\mathsf{y}_k) \otimes E^{(k)}_{m+n,m+n-1}\underline{\mathbf{v}'}=\sum_{k=1}^l \mathbf{m}(\mathsf{y}_k\mathsf{x}_k-\mathsf{x}_k\mathsf{y}_k) \otimes\underline{\mathbf{v}''}.\label{env1}
\end{align}
\end{comment}
By applying \eqref{env1-1} and \eqref{env1-2} to \eqref{env0}, we obtain
\begin{align}
\mathbf{m}(\mathsf{y}_i\mathsf{x}_i-\mathsf{x}_i\mathsf{y}_i) \otimes E^{(i)}_{m+n,m+n-1}(\underline{\mathbf{v}'})&=\mathbf{m}\left( t+\kappa\sum_{\substack{r=1\\r\neq i}}^l\sigma_{r,i}\right)\otimes E_{m+n,m+n-1}^{(i)}(\underline{\mathbf{v}}').\label{env3}
\end{align}
As the entries of $E^{(i)}_{m+n,m+n-1}(\underline{\mathbf{v}}')$ are distinct, by Lemma \ref{chp}, the compatibility with \eqref{defrna-4} follows from \eqref{env3}.
\end{proof}

\end{document}